\documentclass[12pt,reqno,english]{amsart}
\usepackage[T1]{fontenc}
\usepackage[latin9]{inputenc}
\usepackage{geometry}
\usepackage{color}
\usepackage{prettyref}
\usepackage{mathrsfs}
\usepackage{mathtools}
\usepackage{bm}
\usepackage{amsbsy}
\usepackage{amstext}
\usepackage{amsthm}
\usepackage{amssymb}
\usepackage{graphicx}
\usepackage{setspace}
\usepackage{esint}
\makeatletter
\numberwithin{equation}{section}
\numberwithin{figure}{section}
\theoremstyle{plain}
\newtheorem{thm}{\protect\theoremname}[section]
\theoremstyle{remark}
\newtheorem{notation}[thm]{\protect\notationname}
\theoremstyle{plain}
\newtheorem{lem}[thm]{\protect\lemmaname}
\theoremstyle{remark}
\newtheorem{rem}[thm]{\protect\remarkname}
\theoremstyle{plain}
\newtheorem{cor}[thm]{\protect\corollaryname}
\theoremstyle{plain}
\newtheorem{prop}[thm]{\protect\propositionname}
\theoremstyle{remark}
\newtheorem*{acknowledgement*}{\protect\acknowledgementname}

\usepackage{babel}
\usepackage{stmaryrd}

\makeatother

\usepackage{babel}
\providecommand{\acknowledgementname}{Acknowledgement}
\providecommand{\corollaryname}{Corollary}
\providecommand{\lemmaname}{Lemma}
\providecommand{\notationname}{Notation}
\providecommand{\propositionname}{Proposition}
\providecommand{\remarkname}{Remark}
\providecommand{\theoremname}{Theorem}

\begin{document}
\title[Non-reversible Metastable Diffusions with Gibbs Invariant Measure
I]{Non-reversible Metastable Diffusions with Gibbs Invariant Measure
I: Eyring--Kramers Formula}
\author{Jungkyoung Lee and Insuk Seo}
\address{J. Lee. Department of Mathematical Sciences, Seoul National University,
Republic of Korea.}
\email{ljk9316@snu.ac.kr}
\address{I. Seo. Department of Mathematical Sciences and Research Institute
of Mathematics, Seoul National University, Republic of Korea.}
\email{insuk.seo@snu.ac.kr}
\begin{abstract}
In this article, we prove the Eyring--Kramers formula for non-reversible
metastable diffusion processes that have a Gibbs invariant measure.
Our result indicates that non-reversible processes exhibit faster
metastable transitions between neighborhoods of local minima, compared
to the reversible process considered in {[}Bovier, Eckhoff, Gayrard,
and Klein, J. Eur. Math. Soc. 6: 399--424, 2004{]}. Therefore, by
adding non-reversibility to the model, we can indeed accelerate the
metastable transition. Our proof is based on the potential theoretic
approach to metastability through accurate estimation of the capacity
between metastable valleys. We carry out this estimation by developing
a novel method to compute the sharp asymptotics of the capacity \textit{without
relying on variational principles such as the Dirichlet principle
or the Thomson principle. }
\end{abstract}

\maketitle

\section{\label{sec1}Introduction}

In the study of the metastability of stochastic dynamical systems,
one of the most important models is the overdamped Langevin dynamics
given by a stochastic differential equation (SDE) of the form 
\begin{equation}
d\boldsymbol{y}_{\epsilon}(t)\,=\,-\nabla U(\boldsymbol{y}_{\epsilon}(t))\,dt+\sqrt{2\epsilon}\,d\bm{w}_{t}\;,\label{e_SDEy}
\end{equation}
where $(\boldsymbol{w}_{t})_{t\ge0}$ represents the standard $d$-dimensional
Brownian motion, $\epsilon>0$ is a small constant parameter corresponding
to the magnitude of the noise, and $U:\mathbb{R}^{d}\rightarrow\mathbb{R}$
is a smooth Morse function\footnote{All the critical points of $U$ are non-degenerate (i.e., the Hessian
at each critical point is invertible) and isolated from others.} with finite critical points. In addtion to its importance in large-deviation
theory, mathematical physics, and engineering (cf. \cite{FW} and
references therein), this process is also well-known for approximating
the minibatch gradient descent algorithm widely used in deep learning
(cf. \cite{GSS} and references therin). 

The analysis of the metastability of this model has attracted considerable
attention in recent decades. Its first successful mathematical treatment
was carried out in a sequence of pioneering studies by Freidlin and
Wentzell in the 1960s from a large-deviation theoretical perspective,
and these achievements have been summarized in \cite{FW}. Subsequently,
the next breakthrough was achieved in \cite{BEGK1} from a potential
theoretical perspective. In particular, the so-called Eyring--Kramers
formula for \eqref{e_SDEy} was eastablished as a refinement of the
large-deviation result obtained in \cite{FW}. 

Recently, several alternative approaches have been developed in the
study of the metastable behavior of the process $\boldsymbol{y}_{\epsilon}(\cdot)$.
We refer to \cite{RS} written by an author of the current article
and Rezakhanlou for the Poisson equation approach, and \cite{new_GLLNreview}
for the quasi-stationary distribution approach. 

\subsubsection*{Metastable behavior of overdamped Langevin dynamics}

To heuristically explain the metastable behavior of the process, we
first consider the overdamped Langevin dynamics $\boldsymbol{y}_{\epsilon}(\cdot)$.
We regard this process as a small random perturbation of the dynamical
system given by an ordinary differential equation (ODE) of the form
\begin{equation}
d\boldsymbol{y}(t)\,=\,-\nabla U(\boldsymbol{y}(t))\,dt\;.\label{e_ODEy}
\end{equation}
Note that the stable equilibria of this dynamical system are given
by the local minima of $U$. Hence, provided that $\epsilon\simeq0$,
the process $\boldsymbol{y}_{\epsilon}(\cdot)$ starting from a neighborhood
of a local minimum of $U$ will remain there for a sufficeiently long
time, as the noise is small compared to the drift term that pushes
the process toward the local minimum. 

The metastability issue arises for the process $\boldsymbol{y}_{\epsilon}(\cdot)$
if $U$ has multiple local minima. To illustrate the corresponding
metastable behavior more clearly, we simply assume that $U$ has two
local minima $\boldsymbol{m}_{1}$ and $\boldsymbol{m}_{2}$ as shown
in Figure \ref{fig1}, and we suppose that the process $\boldsymbol{y}_{\epsilon}(\cdot)$
starts at $\boldsymbol{m}_{1}$. If there is no noise, i.e., $\epsilon=0$,
the process always remains at $\boldsymbol{m}_{1}$. However, when
$\epsilon$ is small but positive, random noise accumulates over a
sufficiently long time and enables the process $\boldsymbol{y}_{\epsilon}(\cdot)$
to make a transition to a neighborhood of another minimum $\boldsymbol{m}_{2}$,
where it then remains for a long time before making another transition.
Such rare transitions between the neighborhoods of local minima constitute
the dynamical metastable behavior of the process $\boldsymbol{y}_{\epsilon}(\cdot)$.
We can expect richer behaviors when $U$ has a more complex landscape.
\begin{figure}
\includegraphics[scale=0.16]{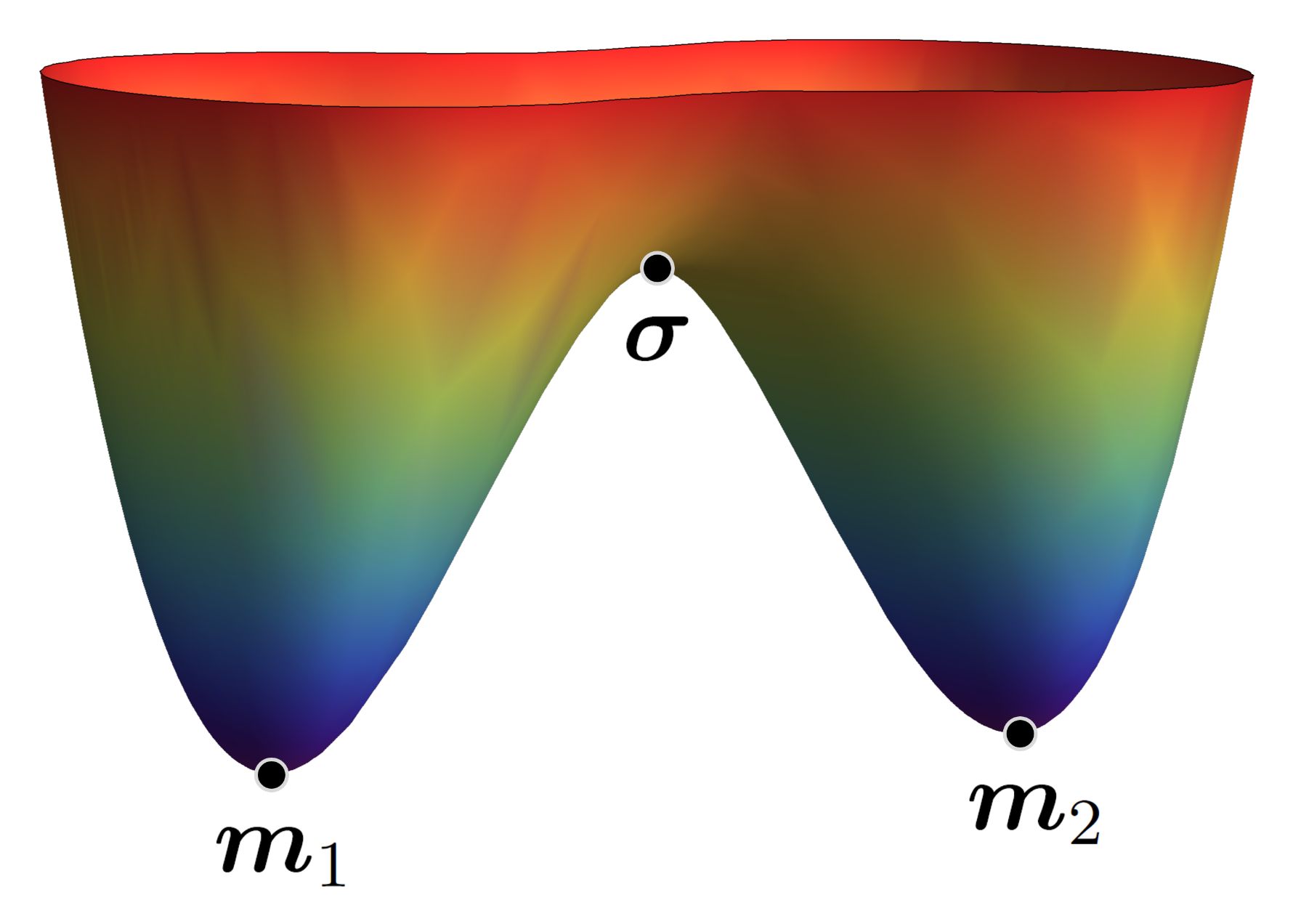}

\caption{\label{fig1}Double-well potential $U$ with two minima $\boldsymbol{m}_{1}$
and $\boldsymbol{m}_{2}$ and a saddle point $\boldsymbol{\sigma}$
between them.}
 
\end{figure}

\subsubsection*{Eyring--Kramers formula}

The Eyring--Kramers formula is the sharp asymptotics, as $\epsilon\rightarrow0$,
of the mean of the time required to observe the transition described
above. It was obtained for the one-dimensional case in classical studies
\cite{Ey,Kra} conducted in the 1930s on the basis of explicit computation.
The generalization of this result to arbitrary dimensions was finally
accomplished in \cite{BEGK1} a few decades later. We recall the double-well
situation illustrated in Figure \ref{fig1} to explain the Eyring--Kramers
formula in a simple form. Let $\tau_{\mathcal{D}_{\epsilon}(\boldsymbol{m}_{2})}$
denote the hitting time with respect to the process $\boldsymbol{y}_{\epsilon}(\cdot)$
of the set $\mathcal{D}_{\epsilon}(\boldsymbol{m}_{2})$, which is
a ball of radius $\epsilon$ centered at $\boldsymbol{m}_{2}$. Then,
the Eyring--Kramers formula is the sharp estimate of the mean transition
time $\mathbb{E}[\tau_{\mathcal{D}_{\epsilon}(\boldsymbol{m}_{2})}|\boldsymbol{y}_{\epsilon}(0)=\boldsymbol{m}_{1}]$.
The Freidlin--Wentzell theory gives the large deviation estimate
for this quantity as 
\begin{equation}
\lim_{\epsilon\rightarrow0}\epsilon\,\log\mathbb{E}[\,\tau_{\mathcal{D}_{\epsilon}(\boldsymbol{m}_{2})}\,|\,\boldsymbol{y}_{\epsilon}(0)=\boldsymbol{m}_{1}\,]\,=\,U(\boldsymbol{\sigma})-U(\boldsymbol{m}_{1})\ ,\label{eqFW}
\end{equation}
where $\boldsymbol{\sigma}$ is the saddle point between the two wells
as shown in Figure \ref{fig1}. The Eyring--Kramers formula is a
refinement of this result (cf. Corollary \ref{cor37} of the current
article), and it gives the precise asymptotics of the expectation
in \eqref{eqFW}. 

The mean transition time is related to the quantification of the mixing
property of the process $\boldsymbol{y}_{\epsilon}(\cdot)$. To explain
it more precisely, we remark that the unique invariant measure for
the process $\boldsymbol{y}_{\epsilon}(\cdot)$ is given by 
\begin{equation}
\mu_{\epsilon}(d\boldsymbol{x})\,=\,\frac{1}{Z_{\epsilon}}\,e^{-U(\bm{x})/\epsilon}\,d\boldsymbol{x}\;,\label{einv}
\end{equation}
where $Z_{\epsilon}$ is the constant given by 
\begin{equation}
Z_{\epsilon}\,=\,\int_{\mathbb{R}^{d}}\,e^{-U(\bm{x})/\epsilon}\,d\bm{x}\,<\,\infty\;,\label{epart}
\end{equation}
where we will impose suitable growth conditions for $U$ in Section
\ref{sec2} to guarantee the finiteness of $Z_{\epsilon}$. The measure
$\mu_{\epsilon}(\cdot)$ corresponds to the\textit{ Gibbs measure}
associated to the energy function $U$ and inverse temperature $\epsilon$
and hence the constant $Z_{\epsilon}$ denotes the associated partition
function. Therefore, we can regard the process $\boldsymbol{y}_{\epsilon}(\cdot)$
as a sampler of the Gibbs distribution $\mu_{\epsilon}(\cdot)$, which
is exponentially concentrated on the global minima of $U$. There
are two representative quantities for measuring this mixing property
of the sampler $\boldsymbol{y}_{\epsilon}(\cdot)$: the spectral gap
\cite{BEGK2} and the mean transition time of the process from one
local minimum to another \cite{BEGK1}. Thus, by estimating the latter
using the Eyring--Kramers formula, one can precisely measure the
mixing property of $\boldsymbol{y}_{\epsilon}(\cdot)$. 

\subsubsection*{Main contribution of this article}

In this article, we consider a variant of the classical overdamped
Langevin dynamics $\boldsymbol{y}_{\epsilon}(\cdot)$, which is obtained
by adding a vector field to the drift term of the SDE \eqref{e_SDEy}.
More precisely, we focus on the Eyring--Kramers formula for the diffusion
process given by an SDE of the form
\begin{equation}
d\bm{x}_{\epsilon}(t)\,=\,-(\nabla U+\boldsymbol{\ell})(\bm{x}_{\epsilon}(t))\,dt+\sqrt{2\epsilon}\,d\bm{w}_{t}\;,\label{e_SDEx}
\end{equation}
where $U$ is the smooth potential function as described above. Further,
$\boldsymbol{\ell}:\mathbb{R}^{d}\rightarrow\mathbb{R}^{d}$ is a
vector field that is orthogonal to the gradient field $\nabla U$,
i.e., 
\begin{equation}
\nabla U(\boldsymbol{x})\cdot\boldsymbol{\ell}(\boldsymbol{x})\,=\,0\;\;\text{ for all }\boldsymbol{x}\in\mathbb{R}^{d}\;,\label{econ_ell1}
\end{equation}
and it is incompressible:
\begin{equation}
(\nabla\cdot\boldsymbol{\ell})(\boldsymbol{x})\,=\,0\;\;\;\text{for all }\boldsymbol{x}\in\mathbb{R}^{d}\;.\label{econ_ell2}
\end{equation}
The condition \eqref{econ_ell1} guarantees that the quasi-potential
of the process $\boldsymbol{x}_{\epsilon}(\cdot)$ is $U$ (cf. \cite[Theorem 3.3.1]{FW}),
and the condition \eqref{econ_ell2} ensures that the invariant measure
of the process $\boldsymbol{x}_{\epsilon}(\cdot)$ is the Gibbs measure
$\mu_{\epsilon}(\cdot)$ (cf. Theorem \ref{t25}). In this sense,
the process $\boldsymbol{x}_{\epsilon}(\cdot)$ is another sampler
of the Gibbs measure $\mu_{\epsilon}(\cdot)$. Indeed, we prove in
Theorem \ref{t25} that the conditions \eqref{econ_ell1} and \eqref{econ_ell2}
are the necessary and sufficient conditions for the process $\boldsymbol{x}_{\epsilon}(\cdot)$
to have as an invariant measure the Gibbs distribution $\mu_{\epsilon}(\cdot)$
defined in \eqref{einv} for all $\epsilon>0$. For this reason, this
generalized model has been investigated in many studies from different
perspectives, e.g., \cite{DLP,HHS,HHS2,LNP,LM,ReS1,ReS2}.

The main contribution of the current article is the proof of the Eyring--Kramers
formula for the process $\boldsymbol{x}_{\epsilon}(\cdot)$ (Theorem
\ref{t33}). We verify in Theorem \ref{t22} that the stable points
of the process $\bm{x}_{\epsilon}(\cdot)$ are the local minima of
$U$ and hence identical to those of the process $\boldsymbol{y}_{\epsilon}(\cdot)$.
Hence, we can compare the Eyring--Kramers formula of $\boldsymbol{x}_{\epsilon}(\cdot)$
with that of $\boldsymbol{y}_{\epsilon}(\cdot)$, and this comparison
reveals that \textit{the mean transition time of the dynamics $\boldsymbol{x}_{\epsilon}(\cdot)$
from one local minimum of $U$ to another is always faster than that
of the overdamped Langevin dynamics $\boldsymbol{y}_{\epsilon}(\cdot)$}.
This implies that we can accelerate the stochastic gradient descent
algorithm by adding the incompressible field $\boldsymbol{\ell}$,
which is orthogonal to $\nabla U$. We remark that such an acceleration
has been observed for the model when the diffusivity $\epsilon$ is
kept constant (see \cite{DLP,HHS,HHS2,LNP,ReS1,ReS2} and references
therein). In particular, we refer to \cite{GMZ} for the explicit
relation with the stochastic gradient descent algorithm. 

We also remark that in a recent study \cite{LM}, the model considered
in this article was investigated in view of the low-lying spectra.
Sharp estimates were established for the exponentially small eigenvalues
of the generator associated with the process $\boldsymbol{x}_{\epsilon}(\cdot)$.
See Corollary \ref{cor ref23} to understand how our discovery is
related to the result presented in \cite{LM}.

\subsubsection*{General methodology of capacity estimation}

Another main result of our study is the establishment of a straightforward
and robust method for estimating a potential theoretic notion known
as the capacity. In the proof of Eyring--Kramers formula based on
the potential theoretic approach developed in \cite{BEGK1}, it is
crucial to estimate the capacity between metastable valleys. In all
the existing results based on this approach, such an estimation is
carried out via variational principles such as the Dirichlet principle
or the Thomson principle. 

For the reversible case, this approach is less complex as the Dirichlet
principle is an optimization problem over a space of functions. Hence,
by taking a suitable test function that approximates the known optimizer
of the variational principle, we can bound the capacity in a precise
manner. This strategy is the essence of the potential theoretic approach
to metastability. By contrast, for the non-reversible case, the variational
expression of the capacity is destined to involve both the function
and the so-called flow (cf. \cite[Theorems 3.2 and 3.3]{LS1}). Therefore,
one must construct both the test function and the test flow to estimate
the capacity precisely. Accordingly, when this approach is adopted
for the non-reversible model, the major technical difficulty arises
in the construction of the test flow. This problem has been resolved
in existing studies such as \cite{LMS,LS1,Seo} based on considerable
computations. 

In this article, we develop a robust methodology to estimate the capacity
without relying on these variational principles. We use only a test
function in the estimation of the capacity; \textit{no test flow is
used even in the non-reversible case}. Hence, our methodology significantly
reduces the complexity of the analysis of metastable non-reversible
processes to the level of the reversible models. Therefore, our methodology
is expected to present new possibilities for the analysis of non-reversible
metastable random processes. 

In summary, we develop a new methodology to estimate the capacity
and use it to establish the Eyring--Kramers formula for the non-reversible
and  metastable diffusions $\boldsymbol{x}_{\epsilon}(\cdot)$. 

\subsubsection*{Related question 1: Markov chain description of metastable behavior }

Now, we consider two important questions. The first deals with a more
comprehensive description of the metastable behavior of the process
$\boldsymbol{x}_{\epsilon}(\cdot)$. To view this problem in a concrete
form, suppose that $U(\boldsymbol{m}_{1})=U(\boldsymbol{m}_{2})$
in the double-well situation illustrated in Figure \ref{fig1}. The
Eyring--Kramers formula focuses on a single metastable transition.
However, this transition will occur repeatedly between the neighborhoods
of two metastable points $\boldsymbol{m}_{1}$ and $\boldsymbol{m}_{2}$,
and one might be interested in describing these repeated transitions
simultaneously. To this end, we can try to prove that a suitably time-rescaled
process converges in some sense to a Markov chain whose state space
consists of two valleys. By doing so, we can completely describe successive
metastable transitions as this Markov chain. We consider this problem
for the process $\boldsymbol{x}_{\epsilon}(\cdot)$ in our companion
paper \cite{LeeS2}. 

\subsubsection*{Related question 2: metastable behavior of the general model}

For a vector field $\boldsymbol{b}:\mathbb{R}^{d}\rightarrow\mathbb{R}^{d}$,
consider the dynamical system in $\mathbb{R}^{d}$ given by an ODE
of the form 
\begin{equation}
d\boldsymbol{z}(t)\,=\,-\boldsymbol{b}(\boldsymbol{z}(t))\,dt\;\;\;;\;t\ge0\;.\label{e_ODEz}
\end{equation}
Suppose that this dynamics has several stable equilibria. An open
problem in the study of metastability is to determine the Eyring--Kramers
formula for the following small random perturbation of \eqref{e_ODEz}:
\begin{equation}
d\boldsymbol{z}_{\epsilon}(t)\,=\,-\boldsymbol{b}(\boldsymbol{z}_{\epsilon}(t))\,dt+\sqrt{2\epsilon}\,d\bm{w}_{t}\;\;\;;\;t\ge0\;.\label{e_SDEz}
\end{equation}
We refer to \cite{BR,FW,LS3} for the study of various aspects of
this question. There are two sources of difficulties in this open
problem. The first one is the non-reversibility of the dynamics, and
the second one is the fact that the invariant measure cannot be written
in an explicit form in general. In the present article, we make a
significant step toward addressing this problem by completely overcoming
the former difficulty. However, since we considered only models with
a Gibbs invariant measure, the latter difficulty is not addressed
and remains to be resolved. 

\section{\label{sec2}Model}

In this section, we introduce the fundamental features of the model.
The results stated in this section regarding the process $\boldsymbol{x}_{\epsilon}(\cdot)$
constitute the essence of this field. However, we could not find a
suitable reference that provides detailed proofs. Hence, we decided
to develop the full details. 

\subsubsection*{Potential function $U$}

To introduce the model rigorously, we must explain the potential function
$U:\mathbb{R}^{d}\rightarrow\mathbb{R}$ in the SDE \eqref{e_SDEx}.
We assume that the potential function $U\in C^{3}(\mathbb{R}^{d})$
is a Morse function that satisfies the growth conditions
\begin{align}
 & \lim_{n\to\infty}\inf_{|\bm{x}|\geq n}\frac{U(\bm{x})}{|\bm{x}|}\,=\,\infty\;,\label{econ_U1}\\
 & \lim_{|\bm{x}|\to\infty}\frac{\bm{x}}{|\bm{x}|}\cdot\nabla U(\bm{x})\,=\,\infty\;,\;\text{and}\label{econ_U2}\\
 & \lim_{|\bm{x}|\to\infty}\left\{ |\nabla U(\bm{x})|-2\Delta U(\bm{x})\right\} \,=\,\infty\;,\label{econ_U3}
\end{align}
where $|\boldsymbol{x}|$ denotes the Euclidean distance in $\mathbb{R}^{d}$.
These conditions have been introduced in previous studies such as
\cite{BEGK1,LMS,RS} to guarantee the positive recurrence of the diffusion
process $\boldsymbol{y}_{\epsilon}(\cdot)$ given by \eqref{e_SDEy}
and the finiteness of $Z_{\epsilon}$ in \eqref{epart}. More precisely,
it is well known (cf. \cite{BEGK1}) that these conditions imply the
tightness condition
\begin{equation}
\int_{\{\bm{x}:U(\bm{x})\geq a\}}e^{-U(\bm{x})/\epsilon}d\bm{x}\,\leq\,C_{a}\,e^{-a/\epsilon}\text{\;\;for all }a\in\mathbb{R}\;,\label{etight_U}
\end{equation}
where $C_{a}$ is a constant that depends only on $a$, and hence
imply the finiteness of the partition function $Z_{\epsilon}$. Finally,
we remark that the metastability of the reversible process $\boldsymbol{y}_{\epsilon}(\cdot)$
has been analyzed in \cite{BEGK1} under the same set of assumptions. 

\subsubsection*{Deterministic dynamical system $\boldsymbol{x}(\cdot)$ }

To explain the metastable behavior of the process $\boldsymbol{x}_{\epsilon}(\cdot)$,
we first consider a deterministic dynamical system given by the ODE
\begin{equation}
d\bm{x}(t)\,=\,-(\nabla U+\boldsymbol{\ell})(\bm{x}(t))\,dt\;.\label{e_odeX}
\end{equation}
We can demonstrate that this dynamical system has essentially the
same phase portrait as $\boldsymbol{y}(\cdot)$ defined in \eqref{e_ODEy}. 
\begin{thm}
\label{t22}The following hold.
\begin{enumerate}
\item We have $\boldsymbol{\ell}(\boldsymbol{c})=0$ for all critical points
$\boldsymbol{c}\in\mathbb{R}^{d}$ of $U$. 
\item A point $\boldsymbol{c}\in\mathbb{R}^{d}$ is an equilibrium of the
dynamical system \eqref{e_odeX} if and only if $\boldsymbol{c}\in\mathbb{R}^{d}$
is a critical point of $U$.
\item An equilibrium $\boldsymbol{c}\in\mathbb{R}^{d}$ of the dynamical
system \eqref{e_odeX} is stable if and only if $\boldsymbol{c}$
is a local minimum of $U$. 
\end{enumerate}
\end{thm}

The proof is given in Section \ref{sec4}. We emphasize that the divergence-free
condition \eqref{econ_ell2} is not used in the proof of this theorem,
whereas the orthogonality condition \eqref{econ_ell1} plays a significant
role. In view of part (3) of the previous theorem, we can observe
that the process $\boldsymbol{x}_{\epsilon}(\cdot)$ is expected to
exhibit metastable behavior when $U$ has multiple local minima, and
this is the situation that we are going to discuss in the current
article. 

\subsubsection*{Diffusion process $\boldsymbol{x}_{\epsilon}(\cdot)$}

Now, we focus on the diffusion process $\boldsymbol{x}_{\epsilon}(\cdot)$.
Under the conditions \eqref{econ_U1}--\eqref{econ_U3} and condition
\eqref{econ_ell1}, we can prove the following property of the process
$\boldsymbol{x}_{\epsilon}(\cdot)$. Note again that the condition
\eqref{econ_ell2} is not used. 
\begin{thm}
\label{t23}The following hold.
\begin{enumerate}
\item There is no explosion for the diffusion process $\boldsymbol{x}_{\epsilon}(\cdot)$.
\item The diffusion process $\boldsymbol{x}_{\epsilon}(\cdot)$ is positive
recurrent. 
\end{enumerate}
\end{thm}

The proof of this result is given in Section \ref{sec5}.

\subsubsection*{Invariant measure}

Since the process $\boldsymbol{x}_{\epsilon}(\cdot)$ is positive
recurrent, we know that this process has an invariant measure. Now,
we prove that $\mu_{\epsilon}$ is the unique invariant measure for
the process $\boldsymbol{x}_{\epsilon}(\cdot)$. 

Before proceeding to the statement of this result, we first explain
the role of the conditions \eqref{econ_ell1} and \eqref{econ_ell2}.
Recall the general model $\boldsymbol{z}_{\epsilon}(\cdot)$ given
by the SDE \eqref{e_SDEz}. It is known from \cite[Theorem 3.3.1]{FW}
that if the quasi-potential $V$ associated with \eqref{e_SDEz} is
of class $C^{1}$, we can write $\boldsymbol{b}=\nabla V+\boldsymbol{\ell}$
where $\nabla V\cdot\boldsymbol{\ell}\equiv0$. Hence, the assumption
\eqref{econ_ell1} is nothing more than the regularity assumption
on the quasi-potential. The special assumption regarding the field
$\boldsymbol{\ell}$ is \eqref{econ_ell2}, and the role of this assumption
is summarized below. 
\begin{thm}
\label{t25}The following hold. 
\begin{enumerate}
\item If $\boldsymbol{\ell}$ satisfies the conditions \eqref{econ_ell1}
and \eqref{econ_ell2}, then the Gibbs measure $\mu_{\epsilon}(\cdot)$
is the unique invariant measure for the diffusion process $\boldsymbol{x}_{\epsilon}(\cdot)$. 
\item On the other hand, suppose that the Gibbs measure $\mu_{\epsilon}(\cdot)$
is the invariant measure for the diffusion process $\boldsymbol{z}_{\epsilon}(\cdot)$
defined in \eqref{e_SDEz} for all $\epsilon>0$. Then, the vector
field $\boldsymbol{b}$ can be written as $\boldsymbol{b}=\nabla U+\boldsymbol{\ell}$,
where $U$ and $\boldsymbol{\ell}$ satisfy \eqref{econ_ell1} and
\eqref{econ_ell2}. 
\end{enumerate}
\end{thm}

The proof of this theorem is given in Section \ref{sec5}. Therefore,
heuristically, the condition \eqref{econ_ell2} can be regarded as
a necessary and sufficient condition (up to the regularity of the
quasi-potential) for the diffusion process $\boldsymbol{z}_{\epsilon}(\cdot)$
has the Gibbs invariant measure. 

\subsubsection*{Construction of $\boldsymbol{\ell}$}

The result obtained in this article might be nearly useless if it
is extremely difficult to find a non-trivial $\boldsymbol{\ell}$
satisfying the conditions \eqref{econ_ell1} and \eqref{econ_ell2}
simultaneously. However, there is a simple way to generate a variety
of $\boldsymbol{\ell}$'s when the potential $U$ is given. Let $\mathcal{M}_{d\times d}(\mathbb{R})$
be a space of $d\times d$ real matrices and let $J:\mathbb{R}\rightarrow\mathcal{M}_{d\times d}(\mathbb{R})$
be a smooth function such that the range of $J$ consists of only
skew-symmetric matrices. Then, a vector field of the form $\boldsymbol{\ell}(\boldsymbol{x})=J(U(\boldsymbol{x}))\,\nabla U(\boldsymbol{x})$
satisfies the conditions \eqref{econ_ell1} and \eqref{econ_ell2}.
This has been observed in \cite[Section 1]{LM}. Moreover, unless
$J$ is a constant function, the model considered here is different
from the one considered in \cite{LMS}. 

\subsubsection*{Notations regarding $\boldsymbol{x}_{\epsilon}(\cdot)$}

We conclude this section by defining some notations regarding the
process $\boldsymbol{x}_{\epsilon}(\cdot)$. Let $\mathscr{L}_{\epsilon}$
denote the generator associated with the process $\boldsymbol{x}_{\epsilon}(\cdot)$.
Then, $\mathscr{L}_{\epsilon}$ acts on $f\in C^{2}(\mathbb{R}^{d})$
such that 
\begin{equation}
\mathscr{L}_{\epsilon}f\,=\,-(\nabla U+\boldsymbol{\ell})\cdot\nabla f+\epsilon\Delta f\;.\label{egenL1}
\end{equation}
Under the conditions \eqref{econ_ell1} and \eqref{econ_ell2} on
$\boldsymbol{\ell},$ we can rewrite this generator in the divergence
form as 
\begin{equation}
\mathscr{L}_{\epsilon}f\,=\,\epsilon e^{U/\epsilon}\nabla\cdot\,\Big[\,e^{-U/\epsilon}\Big(\,\nabla f-\frac{1}{\epsilon}f\,\boldsymbol{\ell}\,\Big)\,\Big]\;.\label{egenL2}
\end{equation}
Let $\mathbb{P}_{\boldsymbol{x}}^{\epsilon}$ denote the law of the
process $\boldsymbol{x}_{\epsilon}(\cdot)$ starting from $\boldsymbol{x}$,
and let $\mathbb{E}_{\boldsymbol{x}}^{\epsilon}$ denote the expectation
with respect to $\mathbb{P}_{\boldsymbol{x}}^{\epsilon}$.

\section{\label{sec3}Main Result}

In this section, we explain the Eyring--Kramers formula for the diffusion
process $\boldsymbol{x}_{\epsilon}(\cdot)$. The main result is stated
in Theorem \ref{t33} (and Corollary \ref{cor37} for the simple double-well
case). 

\subsection{\label{sec31}Structure of metastable valleys }

Let $\mathcal{M}$ denote the set of local minima of $U$. The starting
point $\boldsymbol{m}_{0}\in\mathcal{M}$ of the process $\boldsymbol{x}_{\epsilon}(\cdot)$
is fixed throughout the article. Note that $\boldsymbol{m}_{0}$ is
a stable equilibrium of $\boldsymbol{x}(\cdot)$ by Theorem \ref{t22}. 

Let us fix $H\in\mathbb{R}$ such that $U(\boldsymbol{m}_{0})<H$
and define $\Sigma$ as the set of saddle points of level $H$:

\[
\Sigma\,=\,\{\,\bm{\sigma}:U(\bm{\sigma})=H\text{ and }\boldsymbol{\sigma}\text{ is a }\text{saddle point of }U\,\}\ .
\]
We take $H$ such that $\Sigma\neq\emptyset$. We define 
\begin{equation}
\mathcal{H}\,=\,\{\,\bm{x}\in\mathbb{R}^{d}\,:\,U(\bm{x})<H\,\}\;,\label{eH}
\end{equation}
and we assume that $\mathcal{H}$ has multiple connected components;
hence, metastability occurs. 

We decompose $\mathcal{H}=\mathcal{H}_{0}\cup\mathcal{H}_{1}$, where
$\mathcal{H}_{0}$ is the connected component of $\mathcal{H}$ containing
$\boldsymbol{m}_{0}$ and $\mathcal{H}_{1}=\mathcal{H}\setminus\mathcal{H}_{0}.$
Note that $\mathcal{H}_{1}$ may not be connected. Let $\mathcal{M}_{0}$
and $\mathcal{M}_{1}$ denote the sets of local minima belonging to
$\mathcal{H}_{0}$ and $\mathcal{H}_{1}$, respectively. Let $\mathcal{D}_{r}(\bm{x})$
denote an open ball in $\mathbb{R}^{d}$ centered at $\boldsymbol{x}$
with radius $r$, and define 
\[
\mathcal{U}_{\epsilon}\,:=\,\bigcup_{\boldsymbol{m}\in\mathcal{M}_{1}}\mathcal{D}_{\epsilon}(\boldsymbol{m})\ .
\]
In this article, we focus on the sharp asymptotics of the mean of
the transition time from $\boldsymbol{m}_{0}$ to $\mathcal{U}_{\epsilon}$.
Figure \ref{fig2} illustrates the notations introduced above.

\begin{figure}
\includegraphics[scale=0.21]{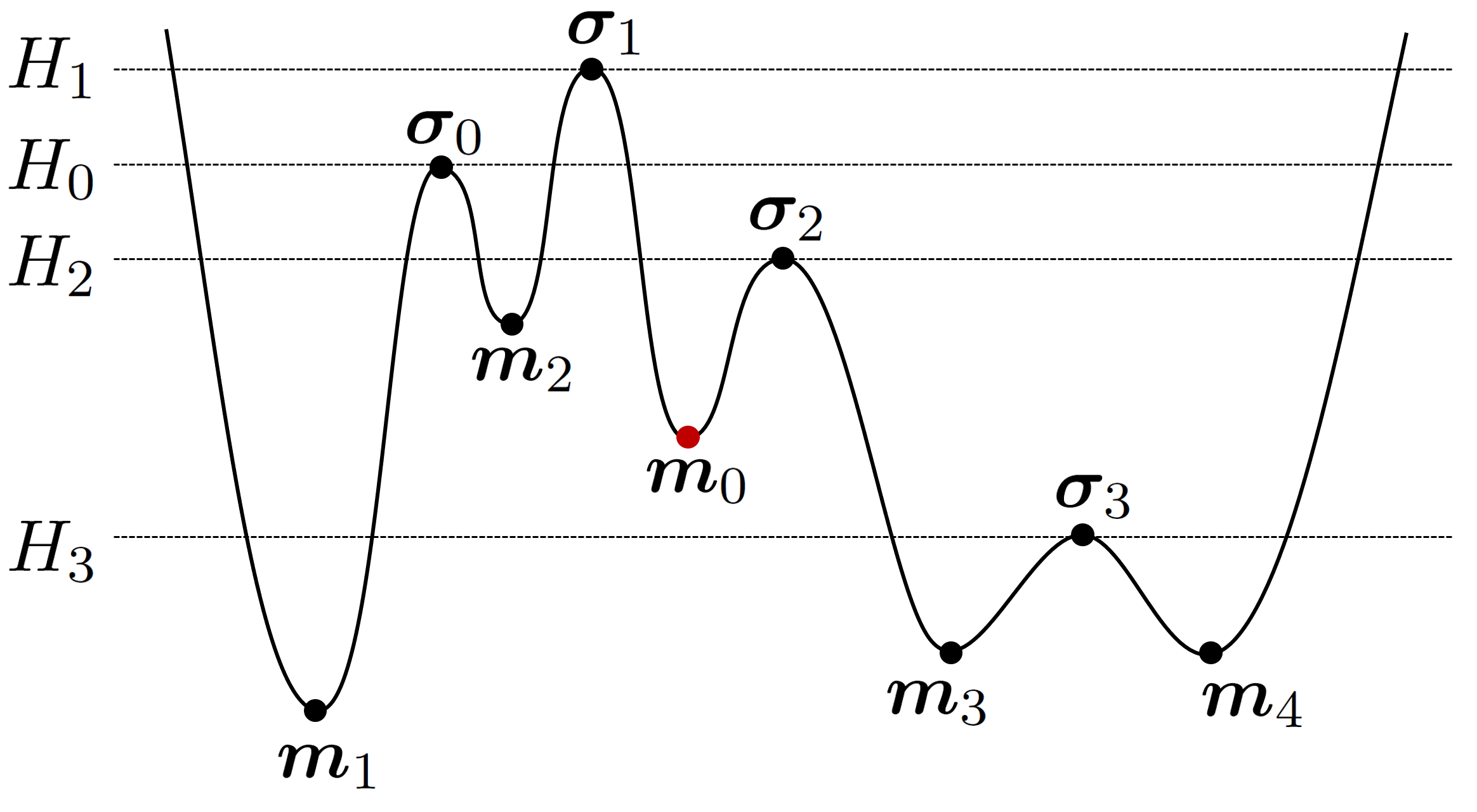}

\caption{\label{fig2}Example of landscape of the potential function $U$ with
five local minima $\{\boldsymbol{m}_{i}:0\le i\le4\}$ and four saddle
points $\{\boldsymbol{\sigma}_{i}:0\le i\le3\}$. We assume that $U(\boldsymbol{m}_{3})=U(\boldsymbol{m}_{4})$
and write $H_{i}=U(\boldsymbol{\sigma_{i}})$, $0\le i\le3$. Our
objective is to compute the transition time from the local minimum
$\boldsymbol{m}_{0}$ to other local minima. We can select the level
$H$ according to our detailed objective.\textbf{ }By taking $H=H_{1},$
we have $\mathcal{M}_{1}=\{\boldsymbol{m}_{1},\,\boldsymbol{m}_{2}\}$;
hence, we focus on the transition time from $\boldsymbol{m}_{0}$
to $\mathcal{D}_{\epsilon}(\boldsymbol{m}_{1})\cup\mathcal{D}_{\epsilon}(\boldsymbol{m}_{2})$.
This occurs at the level of $H_{1}$ since the process must pass through
$\boldsymbol{\sigma}_{1}$ to make such a transition. For this case,
we have $\mathcal{M}_{0}=\{\boldsymbol{m}_{0},\,\boldsymbol{m}_{3},\,\boldsymbol{m}_{4}\}$
and $\mathcal{M}_{0}^{\star}=\{\boldsymbol{m}_{3},\,\boldsymbol{m}_{4}\}$.
On the other hand, by taking $H=H_{2}$, we have $\mathcal{M}_{1}=\{\boldsymbol{m}_{1},\,\boldsymbol{m}_{2},\,\boldsymbol{m}_{3},\,\boldsymbol{m}_{4}\}$.
For this case, we compute the escape time from the metastable valley
around $\boldsymbol{m}_{0}$. The selection $H=H_{3}$ is not available
since the condition $U(\boldsymbol{m}_{0})<H$ is violated; hence
$\mathcal{H}$ does not contain $\boldsymbol{m}_{0}$. This level
is meaningful when we start from, e.g., $\boldsymbol{m}_{3}$. Finally,
the selection $H=H_{0}$ is not appropriate as $\Sigma_{0}$ becomes
an empty set. For this case, we refer to Remark \ref{rmk35} (4) for
further details.}
\end{figure}

\begin{notation}
\label{not31}Since the sets such as $\Sigma$ and $\mathcal{U}_{\epsilon}$
depend on $H$, we add the superscript $H$ to these notations, e.g.,
$\Sigma^{H}$, when we want to emphasize the dependency on $H$. 
\end{notation}

\subsection{\label{sec32}Eyring--Kramers constant for $\boldsymbol{x}_{\epsilon}(\cdot)$}

In the remainder of the article, we use the following notations. 
\begin{notation}
\label{not32}For each critical point $\boldsymbol{c}$ of $U$, let
$\mathbb{H}^{\boldsymbol{c}}=(\nabla^{2}U)(\boldsymbol{c})$ denote
the Hessian of $U$ at $\boldsymbol{c}$ and let $\mathbb{L}^{\boldsymbol{c}}=(D\boldsymbol{\ell})(\boldsymbol{c})$
denote the Jacobian of $\boldsymbol{\ell}$ at $\boldsymbol{c}$. 
\end{notation}

In this subsection, we fix $\boldsymbol{\sigma}\in\Sigma$ and suppose
that $\mathbb{H}^{\boldsymbol{\sigma}}$ has only one negative eigenvalue
$-\lambda^{\boldsymbol{\sigma}}$. In the Eyring--Kramers formula
for the reversible process $\boldsymbol{y}_{\epsilon}(\cdot)$ obtained
in \cite{BEGK1}, an important constant is the so-called \textit{Eyring--Kramers
constant} defined by 
\begin{equation}
\omega_{\textrm{rev}}^{\boldsymbol{\sigma}}\,=\,\frac{\lambda^{\bm{\sigma}}}{2\pi\sqrt{-\det\mathbb{H}^{\boldsymbol{\sigma}}}}\;.\label{eEKconst_rev}
\end{equation}
Now, we introduce the corresponding constant for the process $\boldsymbol{x}_{\epsilon}(\cdot)$.
To this end, we first introduce the following lemma. 
\begin{lem}
\label{lem32}For $\boldsymbol{\sigma}\in\Sigma$, suppose that $\mathbb{H}^{\boldsymbol{\sigma}}$
has only one negative eigenvalue. Then, the matrix $\mathbb{H}^{\boldsymbol{\sigma}}+\mathbb{L}^{\boldsymbol{\sigma}}$
has only one negative eigenvalue and is invertible. 
\end{lem}

Although this has been verified already in \cite[Lemma 1.8]{LM},
we provide the proof of this Lemma in Section \ref{sec43} for the
completeness of the article. Let $-\mu^{\boldsymbol{\sigma}}$ denote
the unique negative eigenvalue obtained in this lemma and define the
Eyring--Kramers constant at $\boldsymbol{\sigma}$ by 
\begin{equation}
\omega^{\boldsymbol{\sigma}}\,=\,\frac{\mu^{\bm{\sigma}}}{2\pi\sqrt{-\det\mathbb{H}^{\boldsymbol{\sigma}}}}\;.\label{eEKconst}
\end{equation}
Then, we can prove the following comparison result for the Eyring--Kramers
constant. 
\begin{lem}
\label{lem35}We have $\mu^{\boldsymbol{\sigma}}\ge\lambda^{\boldsymbol{\sigma}}$;
therefore, $\omega^{\boldsymbol{\sigma}}\ge\omega_{\textrm{rev}}^{\boldsymbol{\sigma}}$. 
\end{lem}

The proof is also given in Section \ref{sec43}. In Corollary \ref{cor38},
we prove that the process $\boldsymbol{x}_{\epsilon}(\cdot)$ is faster
than $\boldsymbol{y}_{\epsilon}(\cdot)$ on the basis of this comparison
result. 

\subsection{\label{sec33}Eyring--Kramers formula for $\boldsymbol{x}_{\epsilon}(\cdot)$}

For $\mathcal{A}\subset\mathbb{R}^{d}$, let $\overline{\mathcal{A}}$
denote the closure of $\mathcal{A}$. Define 
\begin{equation}
\Sigma_{0}\,=\,\overline{\mathcal{H}_{0}}\cap\overline{\mathcal{H}{}_{1}}\,\subset\,\Sigma\;.\label{e_sigma0}
\end{equation}
We assume that $\Sigma_{0}\neq\emptyset$\footnote{The case $\Sigma_{0}=\emptyset$ may occur, for instance, if we take
$H=H_{0}$ in Figure \ref{fig2}.\textbf{ }We can deal with this situation
using our result by modifying $H$; see Remark \ref{rmk35}(4). }. For each $\boldsymbol{\sigma}\in\Sigma_{0}$, the Hessian $\mathbb{H}^{\boldsymbol{\sigma}}$
has only one negative eigenvalue as a consequence of the Morse lemma
(cf. \cite[Lemma 2.2]{Morse theory}); hence, the Eyring--Kramers
constant $\omega^{\boldsymbol{\sigma}}$ at $\boldsymbol{\sigma}\in\Sigma_{0}$
can be defined as in the previous subsection. Then, define
\begin{equation}
\omega_{0}\,=\,\sum_{\boldsymbol{\sigma}\in\Sigma_{0}}\omega^{\boldsymbol{\sigma}}\;.\label{eomega0}
\end{equation}
Let $h_{0}$ denote the minimum of $U$ on $\mathcal{H}_{0}$ and
let $\mathcal{M}_{0}^{\star}$ denote the set of the deepest minima
of $U$ on $\mathcal{H}_{0}$:
\begin{equation}
\mathcal{M}_{0}^{\star}\,=\,\{\boldsymbol{m}\in\mathcal{M}_{0}:U(\boldsymbol{m})=h_{0}\}\;.\label{e_m0star}
\end{equation}
Define
\begin{equation}
\nu_{0}\,=\,\sum_{\boldsymbol{m}\in\mathcal{M}_{0}^{\star}}\frac{1}{\sqrt{\det\mathbb{H}^{\boldsymbol{m}}}}\;.\label{enu0}
\end{equation}
Now, we are ready to state the Eyring--Kramers formula for the non-reversible
process $\boldsymbol{x}_{\epsilon}(\cdot)$, which is the main result
of the current article. For a sequence $(a_{\epsilon})_{\epsilon>0}$
of real numbers, we write $a_{\epsilon}=o_{\epsilon}(1)$ if $\lim_{\epsilon\rightarrow0}a_{\epsilon}=0$. 
\begin{thm}
\label{t33}We have
\begin{equation}
\begin{aligned}\mathbb{E}_{\bm{m}_{0}}^{\epsilon}[\,\tau_{\mathcal{U}_{\epsilon}}\,] & \,=\,[\,1+o_{\epsilon}(1)\,]\,\frac{\nu_{0}}{\omega_{0}}\,\exp\frac{H-h_{0}}{\epsilon}\;.\end{aligned}
\label{eEK_main}
\end{equation}
\end{thm}

\begin{rem}
\label{rmk35}We state the following with regard to Theorem \ref{t33}.
\begin{enumerate}
\item Heuristically, the process $\boldsymbol{x}_{\epsilon}(\cdot)$ starting
at $\boldsymbol{m}_{0}$ first mixes among the neighborhoods of minima
of $\mathcal{M}_{0}^{\star}$, and then makes a transition to $\mathcal{U}_{\epsilon}$
by passing through a neighborhood of the saddle in $\Sigma_{0}$ according
to the Freidlin-Wentzell theory. This is the reason that the formula
\eqref{eEK_main} depends on the local properties of the potential
$U$ at $\mathcal{M}_{0}^{\star}$ and $\Sigma_{0}$. A remarkable
fact regarding the formula \eqref{eEK_main} is that the sub-exponential
prefactor is dominated only by these local properties. This is mainly
because the invariant measure is the Gibbs measure $\mu_{\epsilon}(\cdot).$
It is observed in \cite{BR} that an additional factor called ``non-Gibbsianness''
of the process should be introduced in the general case (i.e., in
the analysis of the metastable behavior of the process $\boldsymbol{z}_{\epsilon}(\cdot)$). 
\item Theorem \ref{t33} is a generalization of \cite[Theorem 3.2]{BEGK1},
as the reversible case is the special $\boldsymbol{\ell}=\boldsymbol{0}$
case of our model. Moreover, a careful reading of our arguments reveals
that the error term $o_{\epsilon}(1)$ is indeed $O(\epsilon^{1/2}\log\frac{1}{\epsilon})$
which is the one that appeared in \cite[Theorems 3.1 and 3.2]{BEGK1}.
\item The constants $\omega_{0}$, $\nu_{0}$, and $h_{0}$ and the set
$\mathcal{U}_{\epsilon}$ are not changed if we take a different starting
point $\boldsymbol{m}_{0}'\in\mathcal{M}_{0}$. In view of Theorem
\ref{t33}, this implies that all the transition times from a point
in $\mathcal{M}_{0}$ to $\mathcal{U}_{\epsilon}$ are asymptotically
the same. For instance, if we take $H=H_{1}$ in Figure \ref{fig2},
the expectation of the hitting time $\tau_{\mathcal{D}_{\epsilon}(\boldsymbol{m}_{1})\cup\mathcal{D}_{\epsilon}(\boldsymbol{m}_{2})}$
is asymptotically the same for the starting points $\boldsymbol{m}_{0},\,\boldsymbol{m}_{3}$
and $\boldsymbol{m}_{4}$. This is because the process $\boldsymbol{x}_{\epsilon}(\cdot)$
sufficiently mixes in the valley $\mathcal{H}_{0}$ before moving
to another valley.
\item Consider the case $H=H_{2}$, where the potential $U$ is given as
Figure \ref{fig2} so that we have $\mathcal{U}_{\epsilon}^{H_{2}}=\{\bm{m}_{1},\,\bm{m}_{2},\,\bm{m}_{3},\,\bm{m}_{4}\}$.
However, in time scale $\exp\left\{ \frac{H_{2}-h_{0}}{\epsilon}\right\} $,
the diffusion process cannot move to the neighborhoods of $\bm{m}_{1}$
and $\bm{m}_{2}$, since $\bm{\sigma}_{2}$ is the only saddle point
in $\Sigma_{0}^{H_{2}}$ and $\bm{m}_{3}$ and $\bm{m}_{4}$ are the
only minima in the connected components of $\mathcal{H}_{1}$ whose
boundary contains $\bm{\sigma}_{2}$. Our proof verifies this as well. 
\item We can tune $H$ such that $\boldsymbol{m}_{0}$ is the unique local
minimum of $\mathcal{H}_{0}$. For example, in Figure \ref{fig2},
we can achieve this by selecting $H=H_{2}$. Then, the formula \eqref{eEK_main}
becomes the asymptotics of the transition time from $\boldsymbol{m}_{0}$
to one of the other local minima, and this is the classic form of
the Eyring--Kramers formula. We remark that all the existing studies
\cite{BEGK1,LMS} on the Eyring--Kramers formula for metastable diffusion
processes have dealt with only this situation. On the other hand,
our result is more comprehensive in that we analyzed all the possible
levels by carefully investigating the equilibrium potential in Section
\ref{sec9}. Such a comprehensive result for a diffusion setting was
barely known previously, see \cite{Mich} where a similar setting
along with the possibility of degenerate critical points has been
discussed. 
\item We use Notation \ref{not31} and suppose that $\Sigma_{0}^{H_{0}}=\emptyset$.
Then, we have $\mathbb{E}_{\bm{m}_{0}}^{\epsilon}[\,\tau_{\mathcal{U}_{\epsilon}}\,]\gg\exp\left\{ \frac{H_{0}-h_{0}}{\epsilon}\right\} $
and the level $H_{0}$ is not appropriate to investigate this mean
transition time. Instead, we define
\[
H^{*}\,=\,\sup\,\{\,H:\mathcal{U}_{\epsilon}^{H}=\mathcal{U}_{\epsilon}^{H_{0}}\,\}
\]
so that at level $H^{*}$ the gate path from $\boldsymbol{m}_{0}$
to $\mathcal{U}_{\epsilon}^{H_{0}}$ firstly appears and hence $\Sigma_{0}^{H^{*}}\neq\emptyset$.
Thus, we can estimate $\mathbb{E}_{\bm{m}_{0}}^{\epsilon}[\,\tau_{\mathcal{U}_{\epsilon}^{H_{0}}}\,]$
by taking $H=H^{*}$. For instance, in Figure \ref{fig2}, we have
$\Sigma_{0}^{H_{0}}=\emptyset$ and $H^{*}=H_{1}$. 
\item By selecting $\boldsymbol{\ell}$ appropriately, we can make $\omega_{0}$
arbitrarily large. 
\end{enumerate}
\end{rem}

The proof of Theorem \ref{t33} is given in Section \ref{sec7}. 

\subsubsection*{Double-well case }

The Eyring--Kramers formula stated above has a simple form in the
double-well case. Recall the double-well situation illustrated in
Figure \ref{fig1}. For this case, the only meaningful selection of
$H$ is $U(\boldsymbol{\sigma})$, and $\Sigma_{0}=\{\boldsymbol{\sigma}\}$
for this choice. With this $H$, we can interpret Theorem \ref{t33}
as following corollary. 
\begin{cor}
\label{cor37}We have
\begin{equation}
\mathbb{E}_{\boldsymbol{m}_{1}}^{\epsilon}[\,\tau_{\mathcal{\mathcal{D}_{\epsilon}}(\boldsymbol{m}_{2})}\,]\,=\,[\,1+o_{\epsilon}(1)\,]\,\frac{2\pi}{\mu^{\boldsymbol{\sigma}}\,}\sqrt{\frac{-\det\mathbb{H}^{\boldsymbol{\sigma}}}{\det\mathbb{H}^{\boldsymbol{m}_{1}}}}\,\exp\frac{U(\boldsymbol{\sigma})-U(\boldsymbol{m}_{1})}{\epsilon}\;.\label{ecor37}
\end{equation}
\end{cor}

This is the classical form of the Eyring--Kramers formula. With this
simple case, we explain why this result is a refinement of the Freidlin--Wentzell
theory. By \cite[Theorem 3.3.1]{FW}, the quasi-potential $V(\cdot\,;\boldsymbol{m}_{1})$
of the process $\boldsymbol{x}_{\epsilon}(\cdot)$ with respect to
the local minimum $\boldsymbol{m}_{1}$ is given by $V(\boldsymbol{x};\boldsymbol{m}_{1})=U(\boldsymbol{x})-U(\boldsymbol{m}_{1})$
on the domain of attraction of $\boldsymbol{m}_{1}$ with respect
to the process $\boldsymbol{x}(\cdot)$. Hence, we can deduce the
following large-deviation type result from the Freidlin--Wentzell
theory:
\[
\lim_{\epsilon\rightarrow0}\epsilon\,\log\mathbb{E}_{\boldsymbol{m}_{1}}^{\epsilon}[\,\tau_{\mathcal{\mathcal{D}_{\epsilon}}(\boldsymbol{m}_{2})}\,]\,=\,U(\boldsymbol{\sigma})-U(\boldsymbol{m}_{1})\;.
\]
In the formula \eqref{ecor37}, we find the precise sub-exponential
pre-factor associated with this large-deviation estimate. 

We can also deduce from Corollary \ref{cor37} a precise relation
between the mean transition time and a low-lying spectrum of the generator
$\mathscr{L}_{\epsilon}$ for the double-well case. In \cite{LM},
the sharp asymptotics for the eigenvalue $\lambda_{\epsilon}$ of
$\mathscr{L}_{\epsilon}$ with the smallest real part was obtained.
Note that the generator $\mathscr{L}_{\epsilon}$ is not self-adjoint;
hence, the eigenvalue might be a complex number. 
\begin{cor}
\label{cor ref23}For the double-well situation, we suppose that $U(\bm{m}_{1})\ge U(\bm{m}_{2})$.
Let $\lambda_{\epsilon}$ denote the one with smallest real part among
the non-zero eigenvalues of $\mathscr{L}_{\epsilon}$. Then, the following
holds: 
\begin{equation}
\mathbb{E}_{\boldsymbol{m}_{1}}^{\epsilon}\,[\,\tau_{\mathcal{\mathcal{D}_{\epsilon}}(\boldsymbol{m}_{2})}\,]\,=\,\frac{1+o_{\epsilon}(1)}{\lambda_{\epsilon}}\;.\label{ec381}
\end{equation}
\end{cor}

Note that $\lambda_{\epsilon}$ as well as the error term $o_{\epsilon}(1)$
in \eqref{ec381} is in general a non-real complex number. Suprisingly,
it is verified in \cite[Remark 1.10]{LM} that $\lambda_{\epsilon}$
is indeed a real number if $U$ is a double-well potential and $\epsilon$
is sufficiently small. We remark that the inverse relationship between
the low-lying spectrum and the mean transition time as in \eqref{ec381}
has been rigorously verified in \cite{BEGK1,BEGK2} for a wide class
of reversible models including $\boldsymbol{y}_{\epsilon}(\cdot)$. 

\subsubsection*{Comparison with reversible case}

The Eyring--Kramers formula for the reversible process $\boldsymbol{y}_{\epsilon}(\cdot)$
has been shown in \cite[Theorem 3.2]{BEGK1}. We can also recover\footnote{Indeed, our result with $\boldsymbol{\ell}=\boldsymbol{0}$ strictly
contains what has been established in \cite{BEGK1}. See Remark \ref{rmk35}-(3).} this result by inserting $\boldsymbol{\ell}=\boldsymbol{0}$. We
now explain this result using our terminology and we provide a comparison
between reversible and non-reversible cases. Write 
\[
\omega_{0,\,\text{rev}}\,=\,\sum_{\boldsymbol{\sigma}\in\Sigma_{0}}\omega_{\boldsymbol{\sigma},\,\textrm{rev}}\;,
\]
and let $\mathbb{E}_{\bm{x},\,\textup{rev}}^{\epsilon}$ denote the
expectation with respect to the reversible process $\boldsymbol{y}_{\epsilon}(\cdot)$
starting from $\boldsymbol{x}\in\mathbb{R}^{d}$. Then, as a consequence
of Theorem \ref{t33} with $\boldsymbol{\ell}=\boldsymbol{0}$, we
get the following corollary. 
\begin{cor}
\label{cor38}The following holds:
\[
\begin{aligned}\mathbb{E}_{\bm{m}_{0},\,\textup{rev}}^{\epsilon}[\,\tau_{\mathcal{U}_{\epsilon}}\,] & \,=\,[\,1+o_{\epsilon}(1)\,]\,\frac{\nu_{0}}{\omega_{0,\,\textup{rev}}}\,\exp\frac{H-h_{0}}{\epsilon}\;.\end{aligned}
\]
Therefore, we have $\mathbb{E}_{\bm{m}_{0}}^{\epsilon}[\,\tau_{\mathcal{U}_{\epsilon}}\,]\le\mathbb{E}_{\bm{m}_{0},\,\textup{rev}}^{\epsilon}[\,\tau_{\mathcal{U}_{\epsilon}}\,]$
for all small enough $\epsilon$. 
\end{cor}

\begin{proof}
The first assertion follows immediately from the fact that $\omega_{\textrm{rev}}^{\boldsymbol{\sigma}}$
defined in \eqref{eEKconst_rev} corresponds to $\omega^{\boldsymbol{\sigma}}$
with $\boldsymbol{\ell}=\boldsymbol{0}.$ The second assertion follows
from Lemma \ref{lem35} which implies that $\omega_{0}\ge\omega_{0,\,\text{rev}}$. 
\end{proof}
In view of the fact that the dynamics $\boldsymbol{y}_{\epsilon}(\cdot)$
plays a crucial role in the stochastic gradient descent algorithm,
we might be able to accelerate this algorithm by adding a suitable
orthogonal, incompressible vector field to the drift part. 

\section{\label{sec4}Dynamical System $\boldsymbol{x}(\cdot)$}

In this section, we prove the properties of the dynamical systems
$\boldsymbol{x}(\cdot)$ given by the ODE \eqref{e_odeX}. 

\subsection{\label{sec41}Preliminary results on matrix computations}

In this section, we present few technical lemmas. We remark that all
the vectors and matrices in this subsection are real. The first lemma
below will be used to investigate the stable equilibria of the dynamical
system $\boldsymbol{x}(\cdot)$. 
\begin{lem}
\label{lem41}Let $\mathbb{A},\,\mathbb{B}$ be square matrices of
the same size and suppose that $\mathbb{\mathbb{A}}$ is symmetric
positive definite and $\mathbb{\mathbb{AB}}$ is skew-symmetric. Then,
all the eigenvalues of matrix $\mathbb{\mathbb{A}}+\mathbb{B}$ are
either positive real or complex with a positive real part. In particular,
the matrix $\mathbb{A}+\mathbb{B}$ is invertible. 
\end{lem}

\begin{proof}
By a change of basis, we may assume that $\mathbb{A}=\text{diag}(\lambda_{1},\lambda_{2},\cdots,\lambda_{d})$
for some $\lambda_{1},\,\dots,\,\lambda_{d}>0$. Let $\alpha$ be
a real eigenvalue of $\mathbb{\mathbb{A}}+\mathbb{B}$ and let $\boldsymbol{u}$
be the corresponding non-zero eigenvector. Then, we have 
\[
0\,<\,|\mathbb{A}\boldsymbol{u}|^{2}\,=\,\mathbb{A}\boldsymbol{u}\cdot(\mathbb{\mathbb{A}+\mathbb{\mathbb{B}}})\boldsymbol{u}\,=\,\alpha(\mathbb{A}\boldsymbol{u}\cdot\boldsymbol{u})\ ,
\]
where the first identity holds since $\mathbb{\mathbb{AB}}$ is skew-symmetric.
This proves that $\alpha>0$ since $\mathbb{A}$ is positive definite. 

Next, let $z=a+ib$ be a complex eigenvalue of $\mathbb{A}+\mathbb{\mathbb{B}}$
and let $\boldsymbol{u}+i\boldsymbol{w}$ be the corresponding non-zero
eigenvector, where $\bm{u}$ and $\bm{w}$ are real vectors. Since
$\mathbb{A}$ and $\mathbb{B}$ are real, we have
\[
(\mathbb{\mathbb{A}+\mathbb{B}})\boldsymbol{u}\,=\,a\boldsymbol{u}-b\boldsymbol{w}\;\;\;\text{and\;\;\;}(\mathbb{\mathbb{A}+B})\boldsymbol{w}\,=\,b\boldsymbol{u}+a\boldsymbol{w}\;.
\]
Since $\mathbb{\mathbb{AB}}$ is skew-symmetric, we get 
\begin{align*}
 & |\mathbb{A}\boldsymbol{u}|^{2}\,=\,\mathbb{\mathbb{A}}\boldsymbol{u}\cdot(\mathbb{\mathbb{A}+B})\boldsymbol{u}\,=\,\mathbb{\mathbb{A}}\boldsymbol{u}\cdot(a\boldsymbol{u}-b\boldsymbol{w})\;,\\
 & |\mathbb{A}\boldsymbol{w}|^{2}\,=\,\mathbb{A}\boldsymbol{w}\cdot(\mathbb{A+\mathbb{B}})\boldsymbol{w}\,=\,\mathbb{A}\boldsymbol{w}\cdot(b\boldsymbol{u}+a\boldsymbol{w})\;.
\end{align*}
By adding these two identities, we get 
\[
|\mathbb{A}\boldsymbol{u}|^{2}+|\mathbb{A}\boldsymbol{w}|^{2}\,=\,a(\mathbb{A}\boldsymbol{u}\cdot\boldsymbol{u}+\mathbb{A}\boldsymbol{w}\cdot\boldsymbol{w})\;.
\]
Therefore, we get $a>0$ since $\mathbb{A}$ is positive definite. 
\end{proof}
The next lemma is used to analyze the saddle points of the dynamical
system \eqref{e_odeX}. For a square matrix $\mathbb{M}$, let $\mathbb{M^{\dagger}}$
denote its transpose, and we write $\mathbb{M}^{s}=\frac{1}{2}(\mathbb{M}+\mathbb{M}^{\dagger})$. 
\begin{lem}
\label{lem42}Let $\mathbb{A},\,\mathbb{B}$ be square matrices of
the same size and suppose that $\mathbb{A}^{s}$ is positive definite
and $\mathbb{B}$ is a non-singular, symmetric matrix that has only
one negative eigenvalue. Then, $\mathbb{AB}$ is invertible and has
only one negative eigenvalue with geometric multiplicity $1$.
\end{lem}

\begin{proof}
By a change of basis, we may assume that $\mathbb{B}=\text{diag}(-\lambda_{1},\,\lambda_{2},\,\dots,\,\lambda_{d})$
for some $\lambda_{1},\,\dots,\,\lambda_{d}>0$. It is well known
that a matrix $\mathbb{A}$ such that $\mathbb{A}^{s}$ is positive
definite does not have a negative eigenvalue and $\det\mathbb{A}>0$.
Hence, we have $\det\mathbb{AB}<0$ so that $\mathbb{AB}$ is invertible
and has at least one negative eigenvalue.

First, assume that $\mathbb{AB}$ has two different negative eigenvalues,
$-a,\,-b$, and let $\bm{u}=(u_{1},\,\dots,\,u_{d})$, $\bm{w}=(w_{1},\,\dots,\,w_{d})$
be the corresponding eigenvectors. We claim that $u_{1},\,w_{1}\neq0$.
By contrast, suppose that $u_{1}=0$. Then, we have
\begin{equation}
\mathbb{B}\bm{u}\cdot\mathbb{A}^{s}\mathbb{B}\bm{u}\,=\,\mathbb{B}\bm{u}\cdot\mathbb{AB}\bm{u}\,=\,-a\mathbb{B}\bm{u}\cdot\bm{u}\,=\,-a\sum_{j=2}^{d}\lambda_{j}u_{j}^{2}\,<\,0\ ,\label{e42}
\end{equation}
which is a contradiction since $\mathbb{A}^{s}$ is positive definite.
By the same argument, we get $w_{1}\ne0$.

By the definition of $a,\,b$ and by the positive definiteness of
$\mathbb{A}^{s}$, for any $t\in\mathbb{R}$,
\[
(\bm{u}+t\bm{w})^{\dagger}\,\mathbb{B}\,(a\bm{u}+bt\bm{w})\,=\,-(\bm{u}+t\bm{w})^{\dagger}\,\mathbb{BAB}\,(\bm{u}+t\bm{w})\,<\,0\;.
\]
Let $p=-u_{1}/(bw_{1})$. By substituting $t$ with $ap$ in the previous
equation, the first coordinate of $a\bm{u}+bt\bm{w}=a(\bm{u}+bp\bm{w})$
is zero; thus, we have 
\begin{equation}
0\,>\,(\bm{u}+ap\bm{w})^{\dagger}\,\mathbb{B}\,(a\bm{u}+abp\bm{w})\,=\,a\sum_{j=2}^{d}\lambda_{j}\,(u_{j}+apw_{j})\,(u_{j}+bpw_{j})\;.\label{e43}
\end{equation}
Similarly, substituting $t$ with $bp$ makes the first coordinate
of $\bm{u}+bp\bm{w}$ zero, and we get 
\begin{equation}
0\,>\,(\bm{u}+bp\bm{w})^{\dagger}\,\mathbb{B}\,(a\bm{u}+b^{2}p\bm{w})\,=\,\sum_{j=2}^{d}\lambda_{j}\,(au_{j}+b^{2}pw_{j})\,(u_{j}+bpw_{j})\;.\label{e44}
\end{equation}
Computing $(b/a\,\times\,$\eqref{e43}$\,+\,$\eqref{e44}$)$ gives
\[
0\,>\,\sum_{j=2}^{d}\lambda_{j}\,(u_{j}+bpw_{j})(bu_{j}+abpw_{j}+au_{j}+b^{2}pw_{j})\,=\,(a+b)\sum_{j=2}^{d}\lambda_{j}(u_{j}+bpw_{j})^{2}\ ,
\]
which is a contradiction since we have assumed that $\lambda_{2},\,\dots,\,\lambda_{d}>0$.
Therefore, $\mathbb{AB}$ has only one negative eigenvalue $-a$.

Now, let us assume that there are two eigenvectors $\bm{u}$ and $\bm{w}$
corresponding to $-a$, which are linearly independent. Then, we can
repeat the same computation as that presented above to get a contradiction,
as we did not use the fact that $a\neq b$ in the computation. Hence,
the dimension of the eigenspace corresponding to the eigenvalue $-a$
is 1. 
\end{proof}
\begin{rem}
Indeed, we can show that the algebraic multiplicity of the unique
negative eigenvalue is also $1$ by considering the Jordan decomposition. 
\end{rem}

The following lemma is a direct consequence of the previous one. In
the application, we substitute $\mathbb{A}$ and $\mathbb{B}$ as
$\mathbb{H}^{\boldsymbol{\sigma}}$ and $\mathbb{L}^{\boldsymbol{\sigma}}$,
respectively, for some $\boldsymbol{\sigma}\in\Sigma_{0}$. 
\begin{lem}
\label{lem44}Let $\mathbb{A},\,\mathbb{B}$ be square matrices of
the same size and suppose that $\mathbb{A}$ is a symmetric non-singular
matrix with exactly one negative eigenvalue and $\mathbb{AB}$ is
a skew-symmetric matrix. Then, the matrix $\mathbb{A+B}$ is invertible
and has only one negative eigenvalue, and its geometric multiplicity
is $1$. 
\end{lem}

\begin{proof}
Since $\mathbb{A}$ is symmetric and $\mathbb{AB}$ is skew-symmetric,
we have $-\mathbb{AB}=(\mathbb{AB})^{\dagger}=\mathbb{B}^{\dagger}\mathbb{A}$.
Therefore, we get $\mathbb{B}\mathbb{A}^{-1}=-\mathbb{A}^{-1}\mathbb{B}^{\dagger}=-(\mathbb{B}\mathbb{A}^{-1})^{\dagger}$;
thus, the matrix $\mathbb{B}\mathbb{A}^{-1}$ is skew-symmetric. Let
$\mathbb{I}$ be the identity matrix with the same size as $\mathbb{A}$.
Then, by substituting $\mathbb{I}+\mathbb{B}\mathbb{A}^{-1}$ and
$\mathbb{A}$ for $\mathbb{A}$ and $\mathbb{B}$, respectively, in
Lemma \ref{lem42}, we conclude the proof since $\mathbb{A+B}=(\mathbb{I}+\mathbb{B}\mathbb{A}^{-1})\mathbb{A}$. 
\end{proof}

\subsection{\label{sec42}Equilibria of the dynamical system \eqref{e_odeX} }

In this subsection, we analyze the equilibria of the dynamical system
\eqref{e_odeX} by proving Theorem \ref{t22}. First, we prove part
(1) of the theorem. 
\begin{proof}[Proof of part (1) of Theorem \ref{t22}]
Let $\boldsymbol{c}\in\mathbb{R}^{d}$ be a critical point of $U$.
Since $\nabla U\cdot\boldsymbol{\ell}\equiv0$ by \eqref{econ_ell1},
we have 
\[
\boldsymbol{0}\,\equiv\,\nabla\,[\,\nabla U\cdot\boldsymbol{\ell}\,]\,=\,(\nabla^{2}U)\,\boldsymbol{\ell}+(D\boldsymbol{\ell})\,\nabla U\;.
\]
Thus, we have $(\nabla^{2}U)(\boldsymbol{c})\,\boldsymbol{\ell}(\boldsymbol{c})=\boldsymbol{0}$
as $\nabla U(\boldsymbol{c})=\boldsymbol{0}$. Since $(\nabla^{2}U)(\boldsymbol{c})$
is invertible as $U$ is a Morse function, we get $\boldsymbol{\ell}(\boldsymbol{c})=\boldsymbol{0}$. 
\end{proof}
Now, we present a lemma that is a consequence of the condition \eqref{econ_ell1}
and part (1) of Theorem \ref{t22} that we have just proved. We recall
the notations $\mathbb{H}^{\boldsymbol{c}}$ and $\mathbb{L}^{\boldsymbol{c}}$
from Notation \ref{not32}. 
\begin{lem}
\label{lem45}For any critical point $\boldsymbol{c}$ of $U$, the
matrix $\mathbb{H}^{\bm{c}}\mathbb{L}^{\bm{c}}$ is skew-symmetric. 
\end{lem}

\begin{proof}
For small $\varepsilon>0$ and $\boldsymbol{x}\in\mathbb{R}^{d}$,
the Taylor expansion implies that\textbf{ }
\begin{gather*}
\nabla U(\bm{c}+\varepsilon\bm{x})\,=\,\varepsilon\,\mathbb{H}^{\bm{c}}\,\bm{x}+O(\varepsilon^{2})\;\;\;\text{and\;\;\;}\boldsymbol{\ell}(\bm{c}+\varepsilon\bm{x})\,=\,\varepsilon\,\mathbb{L}^{\bm{c}}\,\bm{x}+O(\varepsilon^{2})\;,
\end{gather*}
since we have $\nabla U(\bm{c})=\boldsymbol{\ell}(\boldsymbol{c})=0$
by part (1) of Theorem \ref{t22}. By \eqref{econ_ell1}, we have
\[
[\,\varepsilon\mathbb{H}^{\bm{c}}\,\bm{x}+O(\varepsilon^{2})\,]\cdot[\,\varepsilon\,\mathbb{L}^{\bm{c}}\,\bm{x}+O(\varepsilon^{2})\,]\,=\,0\;.
\]
Dividing both sides by $\varepsilon^{2}$ and letting $\varepsilon\rightarrow0$,
we get $(\mathbb{H}^{\bm{c}}\,\boldsymbol{x})\cdot(\mathbb{L}^{\bm{c}}\,\bm{x})=0$.
Since the Hessian $\mathbb{H}^{\bm{c}}$ is symmetric, we can deduce
that $\bm{x}\cdot\mathbb{H}^{\bm{c}}\,\mathbb{L}^{\bm{c}}\,\bm{x}=0$
for all $\bm{x}\in\mathbb{R}^{d}$. This completes the proof. 
\end{proof}
Now, we focus on parts (2) and (3) of Theorem \ref{t22}. 
\begin{proof}[Proof of parts (2) and (3) of Theorem \ref{t22}]
First, we focus on part (2). If $\boldsymbol{c}$ is a critical point
of $U$, we have $(\nabla U+\boldsymbol{\ell})(\boldsymbol{c})=\boldsymbol{0}$
by part (1); thus, $\boldsymbol{c}$ is an equilibrium of the dynamical
system \eqref{e_odeX}. On the other hand, suppose that $\boldsymbol{c}$
is an equilibrium, i.e., $(\nabla U+\boldsymbol{\ell})(\boldsymbol{c})=\boldsymbol{0}$.
Then, by \eqref{econ_ell1}, we have $0=(\nabla U\cdot\boldsymbol{\ell})(\boldsymbol{c})=-|\nabla U(\boldsymbol{c})|^{2}$;
thus, $\nabla U(\boldsymbol{c})=0$. 

For part (3), suppose that $\boldsymbol{c}$ is a local minimum of
$U$ such that the Hessian $\mathbb{H}^{\boldsymbol{c}}$ is positive
definite. Since $\mathbb{H}^{\boldsymbol{c}}\mathbb{L}^{\boldsymbol{c}}$
is skew-symmetric by Lemma \ref{lem45}, we can insert $\mathbb{A}:=\mathbb{H}^{\boldsymbol{c}}$
and $\mathbb{B}=\mathbb{L}^{\boldsymbol{c}}$ into Lemma \ref{lem41}
to conclude that all the eigenvalues of the matrix $\mathbb{H}^{\boldsymbol{c}}+\mathbb{L}^{\boldsymbol{c}}$
are either positive real or complex with a positive real part; hence,
$\boldsymbol{c}$ a is stable equilibrium of the dynamical system
$\boldsymbol{x}(\cdot)$ since $\mathbb{H}^{\boldsymbol{c}}+\mathbb{L}^{\boldsymbol{c}}$
is the Jacobian of the vector field $\nabla U+\boldsymbol{\ell}$
at $\boldsymbol{c}$.

For the other direction, suppose that $\bm{c}$ is a stable equilibrium
of the dynamical system \eqref{e_odeX}, i.e., the matrix $\mathbb{H}^{\boldsymbol{c}}+\mathbb{L}^{\boldsymbol{c}}$
is positive definite in the sense that
\begin{equation}
\boldsymbol{x}\cdot[\mathbb{H}^{\bm{c}}+\mathbb{L}^{\bm{c}}]\boldsymbol{x}>0\;\;\;\;\text{for all }\boldsymbol{x}\neq\boldsymbol{0}\;.\label{pdc}
\end{equation}
Suppose now that the symmetric matrix $\mathbb{H}^{\bm{c}}$ is not
positive definite so that there is a negative eigenvalue $-\lambda<0$.
Let $\bm{v}$ be the corresponding unit eigenvector. Since $\mathbb{H}^{\bm{c}}\mathbb{L}^{\bm{c}}$
is skew-symmetric by Lemma \ref{lem45} and $\mathbb{H}^{\bm{c}}$
is symmetric, we have
\[
2(\mathbb{H}^{\bm{c}})^{2}=\mathbb{H}^{\bm{c}}[\mathbb{H}^{\bm{c}}+\mathbb{L}^{\bm{c}}]+[\mathbb{H}^{\bm{c}}+(\mathbb{L}^{\bm{c}})^{\dagger}]\mathbb{H}^{\bm{c}}\ ,
\]
and thus we get 
\begin{align*}
2\lambda^{2} & =\boldsymbol{v}\cdot2(\mathbb{H}^{\bm{c}})^{2}\boldsymbol{v}=\mathbb{H}^{\bm{c}}\bm{v}\cdot[\mathbb{H}^{\bm{c}}+\mathbb{L}^{\bm{c}}]\bm{v}+\bm{v}\cdot[\mathbb{H}^{\bm{c}}+(\mathbb{L}^{\bm{c}})^{\dagger}]\mathbb{H}^{\bm{c}}\bm{v}\\
 & =-\lambda\bm{v}\cdot[\mathbb{H}^{\bm{c}}+\mathbb{L}^{\bm{c}}+\mathbb{H}^{\bm{c}}+(\mathbb{L}^{\bm{c}})^{\dagger}]\bm{v}=-2\lambda\bm{v}\cdot[\mathbb{H}^{\bm{c}}+\mathbb{L}^{\bm{c}}]\bm{v}\;.
\end{align*}
This contradicts with \eqref{pdc} and therefore $\mathbb{H}^{\boldsymbol{c}}+\mathbb{L}^{\boldsymbol{c}}$
must be positive definite. This completes the proof. 
\end{proof}

\subsection{\label{sec43}Saddle points of dynamical system $\boldsymbol{x}(\cdot)$}

Now, we focus on the saddle points. First, we prove that, for $\boldsymbol{\sigma}\in\Sigma^{0}$,
the matrix $\mathbb{H}^{\boldsymbol{\sigma}}+\mathbb{L}^{\boldsymbol{\sigma}}$
has only one negative eigenvalue as the matrix $\mathbb{H}^{\boldsymbol{\sigma}}$
has only one negative eigenvalue. 
\begin{proof}[Proof of Lemma \ref{lem32}]
Suppose that $\boldsymbol{\sigma}\in\Sigma^{0}$ such that $\mathbb{H}^{\boldsymbol{\sigma}}$
has exactly one negative eigenvalue by the Morse lemma. Then, we can
insert $\mathbb{A}:=\mathbb{H}^{\bm{\sigma}}$ and $\mathbb{B}:=\mathbb{L}^{\bm{\sigma}}$
into Lemma \ref{lem44} owing to Lemma \ref{lem45}, and we can conclude
that the matrix $\mathbb{H}^{\bm{\sigma}}+\mathbb{L}^{\bm{\sigma}}$
has only one negative eigenvalue and is invertible. 
\end{proof}
Next, we prove Lemma \ref{lem35}, which compares the unique (by Lemma
\ref{lem32}) negative eigenvalues of $\mathbb{H}^{\bm{\sigma}}$
and $\mathbb{H}^{\bm{\sigma}}+\mathbb{L}^{\boldsymbol{\sigma}}$ when
$\boldsymbol{\sigma}\in\Sigma^{0}.$
\begin{proof}[Proof of Lemma \ref{lem35}]
Denote by $-\lambda_{1},\,\lambda_{2},\,\dots,\,\lambda_{d}$ the
eigenvalues of the symmetric matrix $\mathbb{H}^{\bm{\sigma}}$, where
$\lambda_{1},\,\dots,\,\lambda_{d}>0$. Thus, $\lambda^{\boldsymbol{\sigma}}=\lambda_{1}$.
Let $\boldsymbol{u}_{i},\,\dots,\,\boldsymbol{u}_{d}$ denote the
normal eigenvectors of $\mathbb{H}^{\bm{\sigma}}$ corresponding to
the eigenvalues $-\lambda_{1},\,\dots,\,\lambda_{d}$, respectively.
Let $\boldsymbol{v}$ denote the unit eigenvector of $\mathbb{H}^{\bm{\sigma}}+\mathbb{L}^{\boldsymbol{\sigma}}$
corresponding to the unique negative eigenvalue $-\mu^{\boldsymbol{\sigma}}$
and write $\boldsymbol{v}=\sum_{i=1}^{d}a_{i}\boldsymbol{u}_{i}$.
Since $\mathbb{H}^{\boldsymbol{\sigma}}\mathbb{L}^{\boldsymbol{\sigma}}$
is skew-symmetric by Lemma \ref{lem45}, we have 
\[
|\mathbb{H}^{\bm{\sigma}}\boldsymbol{v}|^{2}\,=\,\boldsymbol{v}\cdot\mathbb{H}^{\bm{\sigma}}(\mathbb{H}^{\bm{\sigma}}+\mathbb{L}^{\boldsymbol{\sigma}})\boldsymbol{v}\,=\,-\mu^{\bm{\sigma}}\boldsymbol{v}\cdot\mathbb{H}^{\bm{\sigma}}\boldsymbol{v}\;.
\]
Using the above-mentioned notations, we can rewrite this identity
as 
\begin{equation}
\sum_{i=1}^{d}a_{i}^{2}\lambda_{i}^{2}\,=\,-\mu^{\bm{\sigma}}\,\Big[\,-a_{1}^{2}\lambda_{1}+\sum_{i=2}^{d}a_{i}^{2}\lambda_{i}\,\Big]\;.\label{e45}
\end{equation}
First, suppose that $a_{1}=0$. Then, we have $\sum_{i=2}^{d}a_{i}^{2}\lambda_{i}^{2}=-\mu^{\bm{\sigma}}\sum_{i=2}^{d}a_{i}^{2}\lambda_{i}$
and hence we get $a_{2}=\cdots=a_{d}=0$. This implies that $\boldsymbol{v}=0$,
which is a contradiction. Thus, $a_{1}\neq0$. By \eqref{e45}, we
have 
\[
a_{1}^{2}\lambda_{1}^{2}\,\le\,\sum_{i=1}^{d}a_{i}^{2}\lambda_{i}^{2}\,=\mu^{\bm{\sigma}}a_{1}^{2}\lambda_{1}-\mu^{\bm{\sigma}}\sum_{i=2}^{d}a_{i}^{2}\lambda_{i}\,\le\,\mu^{\boldsymbol{\sigma}}a_{1}^{2}\lambda_{1}\;.
\]
Since $a_{1}\neq0$, we get $\mu^{\boldsymbol{\sigma}}\ge\lambda_{1}=\lambda^{\boldsymbol{\sigma}}$.
\end{proof}

\section{\label{sec5}Properties of Diffusion Process $\boldsymbol{x}_{\epsilon}(\cdot)$}

In this section, we prove the basic properties of the diffusion process
$\boldsymbol{x}_{\epsilon}(\cdot)$. 

\subsection{\label{sec51}Positive recurrence and non-explosion}

First, we establish a technical lemma. 
\begin{lem}
\label{lem51}For all $\epsilon>0$, there exists $r_{0}=r_{0}(\epsilon)>0$
such that $(\mathscr{L}_{\epsilon}U)(\boldsymbol{x})\le-3$ for all
$\boldsymbol{x}\notin\mathcal{D}_{r_{0}}(\boldsymbol{0})$. 
\end{lem}

\begin{proof}
By \eqref{econ_U2} and \eqref{econ_U3}, we can take $r_{0}$ to
be sufficiently large such that 
\begin{equation}
|\,\nabla U(\bm{x})\,|-2\,\Delta U(\bm{x})\,>\,\frac{\epsilon}{2}\text{\;\;\;and\;\;\;}|\,\nabla U(\bm{x})\,|\,>\,2\label{e54}
\end{equation}
for all $\boldsymbol{x}\notin\mathcal{D}_{r_{0}}(\boldsymbol{0})$.
Then, for $\boldsymbol{x}\notin\mathcal{D}_{r_{0}}(\boldsymbol{0})$,
we have 
\[
\Delta U(\bm{x})\,\le\,-\frac{\epsilon}{4}+\frac{1}{2}\,|\,\nabla U(\bm{x})\,|\,\le\,\frac{1}{4\epsilon}\,|\,\nabla U(\bm{x})\,|^{2}\;.
\]
Therefore, 
\[
(\mathscr{L}_{\epsilon}U)(\bm{x})\,=\,-|\,\nabla U(\bm{x})\,|^{2}+\epsilon\,\Delta U(\bm{x})\,\le\,-\frac{3}{4}\,|\,\nabla U(\bm{x})\,|^{2}\,\le\,-3\;.
\]
The last inequality follows from the second condition of \eqref{e54}. 
\end{proof}
Now, we prove Theorem \ref{t23}
\begin{proof}[Proof of Theorem \ref{t23}]
First, we prove part (1), i.e., the non-explosion property. By \cite[Theorem at page 197]{V},
it suffices to check that there exists a smooth function $u:\mathbb{R}^{d}\rightarrow(0,\,\infty)$
such that 
\begin{align}
 & u(\boldsymbol{x})\rightarrow\infty\text{\;\;as\;\;}\boldsymbol{x}\rightarrow\infty\;\;\;\text{and}\;\;\;(\mathscr{L}_{\epsilon}u)(\boldsymbol{x})\,\le\,u(\boldsymbol{x})\;\;\text{for all }\boldsymbol{x}\in\mathbb{R}^{d}\ .\label{e555}
\end{align}
We claim that $u=U+k_{\epsilon}$ with a sufficiently large constant
$k_{\epsilon}$ satisfies all these conditions. First, we take $k_{\epsilon}$
to be sufficiently large such that $u>0$. The former condition of
\eqref{e555} is immediate from \eqref{econ_U1}. Now, it suffices
to check the second condition. By Lemma \ref{lem51}, the function
$\mathscr{L}_{\epsilon}u=\mathscr{L}_{\epsilon}U$ is bounded from
above. Denote this bound by $M_{\epsilon}$ and then take $k_{\epsilon}$
to be sufficiently large such that $u(\boldsymbol{x})>M_{\epsilon}$
for all $\boldsymbol{x}\in\mathbb{R}^{d}$. Then, the second condition
of \eqref{e555} follows. 

The positive recurrence of $\boldsymbol{x}_{\epsilon}(\cdot)$ follows
from Lemma \ref{lem51} and \cite[Theorem 6.1.3]{Pin}.
\end{proof}

\subsection{\label{sec52}Invariant measure}

By a slight abuse of notation, we write $\mu_{\epsilon}(\boldsymbol{x})=Z_{\epsilon}^{-1}e^{-U(\boldsymbol{x})/\epsilon}$
(cf. \eqref{einv}). Now, we prove Theorem \ref{t25}. We can observe
from the expression \eqref{egenL2} of the generator $\mathscr{L}_{\epsilon}$
that the adjoint generator $\mathscr{L}_{\epsilon}^{\textrm{a}}$
of $\mathscr{L}_{\epsilon}$ with respect to the Lebesgue measure
$d\boldsymbol{x}$ can be written as 
\begin{equation}
\mathscr{L}_{\epsilon}^{\textrm{a}}f\,=\,\epsilon\,\nabla\cdot[\,e^{-U/\epsilon}\nabla(e^{U/\epsilon}f)\,]+\boldsymbol{\ell}\cdot\nabla(e^{U/\epsilon}\,f)\;.\label{egen_adj}
\end{equation}

\begin{proof}[Proof of Theorem \ref{t25}]
First, we prove part (1). With the expression \eqref{egen_adj} and
the explicit form of $\mu_{\epsilon}(\boldsymbol{x})$, we can check
that $\mathscr{L}_{\epsilon}^{\textrm{a}}\mu_{\epsilon}=0$. Therefore,
by \cite[Theorem at page 254]{V} and part (1) of Theorem \ref{t23},
the measure $\mu_{\epsilon}(\cdot)$ is the invariant measure for
the process $\boldsymbol{x}_{\epsilon}(\cdot)$. The uniqueness follows
from \cite[Theorem at page 259 ]{V} and \cite[Theorem at page 260 ]{V}. 

For part (2), let us assume that $\mu_{\epsilon}(\cdot)$ is the invariant
measure for the dynamics $\boldsymbol{z}_{\epsilon}(\cdot)$ given
in \eqref{e_SDEz} for all $\epsilon>0$. Note that the generator
associated with the process $\boldsymbol{z}_{\epsilon}(\cdot)$ acts
on $f\in C^{2}(\mathbb{R}^{d})$ as 
\[
\mathscr{\widetilde{L}}_{\epsilon}f\,=\,-\boldsymbol{b}\cdot\nabla f+\epsilon\Delta f\;.
\]
Hence, its adjoint generator with respect to the Lebesgue measure
is given by 
\[
\mathscr{\widetilde{L}}_{\epsilon}^{\textrm{a}}f\,=\,\nabla\cdot[\,f\boldsymbol{b}\,]+\epsilon\Delta f\;.
\]
By \cite[Theorem at page 259 ]{V}, we must have $\mathscr{\widetilde{L}}_{\epsilon}^{\textrm{a}}\mu_{\epsilon}=0$.
By writing $\boldsymbol{\ell}=\boldsymbol{b}-\nabla U$, this equation
can be expressed as $e^{-U/\epsilon}\left[\frac{1}{\epsilon}\nabla U\cdot\boldsymbol{\ell}+\nabla\cdot\boldsymbol{\ell}\right]=0$.
Since this holds for all $\epsilon>0$, the vector field $\boldsymbol{\ell}$
must satisfy both \eqref{econ_ell1} and \eqref{econ_ell2}. 
\end{proof}

\section{\label{sec6}Potential Theory}

In this section, we introduce the potential theory related to the
process $\boldsymbol{x}_{\epsilon}(\cdot)$. As in the previous studies,
we prove the Eyring--Kramers formula based on the relation between
the mean transition time and the potential theoretic notions, and
this relation is recalled in Proposition \ref{p71}. The difficulty,
especially for the non-reversible process, in using this formula arises
from the estimation of the capacity term appearing in the formula.
In this article, as explained in the Introduction section, we develop
a novel and simple way to estimate the capacity. In this section,
we explain a formula given by Proposition \ref{p62} for the capacity
which plays a crucial role in our method. We remark that this formula
itself is not new; the method for handling this formula is the innovation
of the current study, and will be explained in the remainder of this
article. To explain this formula, we start by introducing the adjoint
process, equilibrium potential, and capacity. 

\subsection{\label{sec61}Adjoint process}

The adjoint operator $\mathscr{L}_{\epsilon}^{*}$ of $\mathscr{L}_{\epsilon}$
with respect to the invariant measure $\mu_{\epsilon}$ can be written
as 
\begin{align}
\mathscr{L}_{\epsilon}^{*}f & \,=\,\epsilon e^{U/\epsilon}\nabla\cdot\,\Big[\,e^{-U/\epsilon}\,\Big(\,\nabla f+\frac{1}{\epsilon}f\boldsymbol{\ell}\,\Big)\,\Big]\,=\,-(\nabla U-\boldsymbol{\ell})\cdot\nabla f+\epsilon\,\Delta f\;.\label{eL*}
\end{align}
Note that the generator $\mathscr{L}_{\epsilon}^{\textrm{a}}$ defined
in \eqref{egen_adj} is an adjoint with respect to the Lebesgue measure,
instead of $\mu_{\epsilon}$. The adjoint process $\boldsymbol{x}_{\epsilon}^{*}(\cdot)$
is the diffusion process associated with the generator $\mathscr{L}_{\epsilon}^{*}$;
hence, it is given by the SDE 
\[
d\boldsymbol{x}_{\epsilon}^{*}(t)\,=\,-(\nabla U-\boldsymbol{\ell})(\boldsymbol{x}_{\epsilon}^{*}(t))\,dt+\sqrt{2\epsilon}\,d\boldsymbol{w}_{t}\;.
\]
Let $\mathbb{P}_{\boldsymbol{x}}^{\epsilon,\,*}$ denote the law of
the process $\boldsymbol{x}_{\epsilon}^{*}(\cdot)$. We can prove
that the process $\boldsymbol{x}_{\epsilon}^{*}(\cdot)$ is positive
recurrent and has the unique invariant measure $\mu_{\epsilon}(\cdot)$
by an argument that is identical to that for $\boldsymbol{x}_{\epsilon}(\cdot)$. 

\subsection{\label{sec62}Equilibrium potentials and capacities}

In the remainder of this section, we fix two disjoint non-empty bounded
domains $\mathcal{A}$ and $\mathcal{B}$ of $\mathbb{R}^{d}$ with
$C^{2,\,\alpha}$-boundaries for some $\alpha\in(0,\,1)$ such that
the perimeters $\sigma(\mathcal{A})$ and $\sigma(\mathcal{B})$ are
finite, and $d(\mathcal{A},\,\mathcal{B})>0$. Now, we introduce the
equilibrium potential and capacity between the two sets $\mathcal{A}$
and $\mathcal{B}$. Write $\Omega=(\overline{\mathcal{A}}\cup\overline{\mathcal{B}})^{c}$
so that $\partial\Omega=\partial\mathcal{A}\cup\partial\mathcal{B}$. 

The equilibrium potentials $h_{\mathcal{A},\mathcal{\,B}}^{\epsilon}$,
$h_{\mathcal{A},\mathcal{\,B}}^{\epsilon,\,*}:\mathbb{R}^{d}\rightarrow\mathbb{R}$
between $\mathcal{A}$ and $\mathcal{B}$ with respect to the processes
$\boldsymbol{x}_{\epsilon}(\cdot)$ and $\boldsymbol{x}_{\epsilon}^{*}(\cdot)$
are given by 
\[
h_{\mathcal{A},\mathcal{\,B}}^{\epsilon}\,(\boldsymbol{x})\,=\,\mathbb{P}_{\boldsymbol{x}}^{\epsilon}\,[\,\tau_{\mathcal{A}}<\tau_{\mathcal{B}\:}]\;\;\;\text{and\;\;\;}h_{\mathcal{A},\mathcal{\,B}}^{\epsilon,\,*}\,(\boldsymbol{x})\,=\,\ensuremath{\mathbb{P}_{\boldsymbol{x}}^{\epsilon,\,*}}\,[\,\tau_{\mathcal{A}}<\tau_{\mathcal{B}}\,]
\]
for $\boldsymbol{x}\in\mathbb{R}^{d}$, respectively. 

The capacity between $\mathcal{A}$ and $\mathcal{B}$ with respect
to the processes $\boldsymbol{x}_{\epsilon}(\cdot)$ and $\boldsymbol{x}_{\epsilon}^{*}(\cdot)$
are respectively defined by 
\begin{align}
\textrm{cap}_{\epsilon}(\mathcal{A},\,\mathcal{B}) & \,=\,\epsilon\int_{\partial\mathcal{A}}(\nabla h_{\mathcal{A},\mathcal{\,B}}^{\epsilon}\cdot\bm{n}_{\Omega})\,\sigma(d\mu_{\epsilon})\;\label{ecap}\\
\textrm{cap}_{\epsilon}^{*}(\mathcal{A},\,\mathcal{B}) & \,=\,\epsilon\int_{\partial\mathcal{A}}(\nabla h_{\mathcal{A},\mathcal{\,B}}^{\epsilon,\,*}\cdot\bm{n}_{\Omega})\,\sigma(d\mu_{\epsilon})\;,\nonumber 
\end{align}
where $\bm{n}_{\Omega}(\bm{x})$ is the outward normal vector to $\Omega$
at $\boldsymbol{x}$; hence, $\bm{n}_{\Omega}(\bm{x})=-\bm{n}_{\mathcal{A}}(\bm{x})$
for $\boldsymbol{x}\in\partial\mathcal{A}$. Here, $\int_{\partial\mathcal{A}}f\,\sigma(d\mu_{\epsilon})$
is a shorthand of $\int_{\partial\mathcal{A}}f(\boldsymbol{x})\,\mu_{\epsilon}(\boldsymbol{x})\,\sigma(d\boldsymbol{x})$.
These capacities exhibit the following well-known properties. 
\begin{lem}
\label{lem61}The following properties hold.
\begin{enumerate}
\item We have
\[
\textup{cap}_{\epsilon}(\mathcal{A},\,\mathcal{B})\,=\,\textup{cap}_{\epsilon}^{*}(\mathcal{A},\,\mathcal{B})\,=\,\text{\textup{cap}}_{\epsilon}^{*}(\mathcal{B},\,\mathcal{A})\,=\,\textup{cap}_{\epsilon}(\mathcal{B},\,\mathcal{A})\ .
\]
\item We have
\[
\textup{cap}_{\epsilon}(\mathcal{A},\,\mathcal{B})\,=\,\epsilon\int_{\Omega}|\nabla h_{\mathcal{A},\,\mathcal{B}}^{\epsilon}|^{2}\,d\mu_{\epsilon}\,=\,\epsilon\int_{\Omega}|\nabla h_{\mathcal{A},\mathcal{\,B}}^{\epsilon,\,*}|^{2}\,d\mu_{\epsilon}\;.
\]
\end{enumerate}
\end{lem}

\begin{proof}
We refer to \cite[Lemmas 3.2 and 3.1]{LMS} for the proof of parts
(1) and (2), respectively. 
\end{proof}

\subsection{\label{sec63}Representation of capacity}

We keep the sets $\mathcal{A},\,\mathcal{B}$, and $\Omega$ from
the previous subsection. Then, for a function $f:\mathbb{R}^{d}\rightarrow\mathbb{R}$
that is differentiable at $\boldsymbol{x}\in\mathbb{R}^{d}$, we define
a vector field $\Phi_{f}$ at $\boldsymbol{x}$ as\textbf{\textcolor{red}{{} }}

\begin{equation}
\Phi_{f}(\boldsymbol{x})\,=\,\nabla f(\boldsymbol{x})+\frac{1}{\epsilon}f(\boldsymbol{x})\,\boldsymbol{\ell}(\boldsymbol{x})\;.\label{ePhi}
\end{equation}
Let $C_{0}^{\infty}(\mathbb{R}^{d})$ denote the class of smooth and
compactly supported functions on $\mathbb{R}^{d}$. Let 
\begin{equation}
\mathcal{\mathscr{C}}_{\mathcal{A},\,\mathcal{B}}\,=\,\{\,f\in C_{0}^{\infty}(\mathbb{R}^{d}):f\equiv1\;\text{on}\;\mathcal{A}\;,\;\;\;f\equiv0\;\text{on}\;\mathcal{B}\,\}.\label{eCab}
\end{equation}
Hence, for $f\in\mathcal{\mathscr{C}}_{\mathcal{A},\,\mathcal{B}}$,
the vector field $\Phi_{f}$ is defined on $\mathbb{R}^{d}$. The
following expression plays a crucial role in the estimation of the
capacity. 
\begin{prop}
\label{p62}For all $f\in\mathcal{\mathscr{C}}_{\mathcal{A},\,\mathcal{B}}$,
we have
\begin{equation}
\epsilon\int_{\Omega}\,[\,\Phi_{f}\cdot\nabla h_{\mathcal{A},\,\mathcal{B}}^{\epsilon}\,]\,d\mu_{\epsilon}\,=\,\textup{cap}_{\epsilon}(\mathcal{A},\,\mathcal{B})\;.\label{ep62}
\end{equation}
\end{prop}

\begin{proof}
Since $f$ is compactly supported, we can apply the divergence theorem
to rewrite the left-hand side of \eqref{ep62} as 
\[
\epsilon\,\int_{\partial\Omega}\,f\,[\,\nabla h_{\mathcal{A},\mathcal{\,B}}^{\epsilon}\cdot\bm{n}_{\Omega}\,]\,\sigma(d\mu_{\epsilon})\,-\,\int_{\Omega}f\,(\mathscr{L}_{\epsilon}h_{\mathcal{A},\,\mathcal{B}}^{\epsilon})\,d\mu_{\epsilon}\;.
\]
Since $f=\mathbf{1}_{\partial\mathcal{A}}$ on $\partial\Omega$ by
the condition $f\in\mathcal{\mathscr{C}}_{\mathcal{A},\,\mathcal{B}}$,
the first term of the above-mentioned expression is equal to $\textrm{cap}_{\epsilon}(\mathcal{A},\,\mathcal{B})$
by \eqref{ecap}. On the other hand, the second integral is $0$ since
$\mathscr{L}_{\epsilon}h_{\mathcal{A},\mathcal{\,B}}^{\epsilon}\equiv0$
on $\Omega$ by the property of the equilibrium potential. 
\end{proof}

\section{\label{sec7}Proof of Eyring--Kramers Formula}

In this section, we prove the Eyring--Kramers formula stated in Theorem
\ref{t33} up to the construction of a test function and analysis
of the equilibrium potential. 

\subsection{\label{sec71}Proof of Theorem \ref{t33}}

For convenience of notation, we will use the following abbreviations
for the capacity and equilibrium potential between a small ball around
the minimum $\boldsymbol{m}_{0}$ and $\mathcal{U}_{\epsilon}$: 
\begin{align}
\text{cap}_{\epsilon} & \,=\,\textrm{cap}_{\epsilon}(\,\mathcal{D}_{\epsilon}(\bm{m}_{0}),\,\mathcal{U}_{\epsilon}\,)\;,\nonumber \\
h_{\epsilon}(\cdot) & \,=\,h_{\mathcal{D}_{\epsilon}(\bm{m}_{0}),\,\mathcal{U}_{\epsilon}}^{\epsilon}(\cdot)\;\;\;\;\text{and\;\;\;\;}h_{\epsilon}^{*}(\cdot)\,=\,h_{\mathcal{D}_{\epsilon}(\bm{m}_{0}),\,\mathcal{U}_{\epsilon}}^{\epsilon,\,*}(\cdot)\;.\label{epab}
\end{align}
The proof of the Eyring--Kramers formula relies on the following
formula regarding the mean transition time. 
\begin{prop}
\label{p71}We have 
\begin{equation}
\mathbb{E}_{\bm{m}_{0}}^{\epsilon}[\,\tau_{\mathcal{U}_{\epsilon}}\,]\,=\,[\,1+o_{\epsilon}(1)\,]\,\frac{1}{\textup{cap}_{\epsilon}}\,\int_{\mathbb{R}^{d}}\,h_{\epsilon}^{*}\,d\mu_{\epsilon}\;.\label{emagf}
\end{equation}
\end{prop}

This remarkable relation between the mean transition time and the
potential theoretic notions was first observed in \cite[Proposition 6.1]{BEGK1}
for the reversible case. Then, it was extended to the general non-reversible
case in \cite[Lemma 9.2]{LMS}. Our proof is identical to that of
the latter case; hence, we omit the details. Now, the proof of Theorem
\ref{t33} is reduced to computing the right-hand side of \eqref{emagf}.
We shall estimate the capacity and integral terms separately. We emphasize
here that, even if we rely on the general formula \eqref{emagf},
the estimation of these two terms is carried out in a novel manner.
For simplicity of notation, hereafter, we write 
\begin{equation}
\alpha_{\epsilon}\,=\,Z_{\epsilon}^{-1}\,e^{-H/\epsilon}\,(2\pi\epsilon)^{d/2}\;.\label{ealphae}
\end{equation}
Our main innovation in the proof of the Eyring--Kramers formula is
the new strategy to prove the following proposition.
\begin{prop}
\label{p72}For $\omega_{0}$ defined in \eqref{eomega0}, we have
\begin{equation}
\textup{cap}_{\epsilon}\,=\,[\,1+o_{\epsilon}(1)\,]\,\alpha_{\epsilon}\,\omega_{0}\;.\label{ep72}
\end{equation}
\end{prop}

We present our proof, up to the construction of a test function, in
the next subsection. Further, we need to estimate the integral term
in \eqref{emagf}. 
\begin{prop}
\label{p73}For $\nu_{0}$ defined in \eqref{enu0}, we have 
\begin{equation}
\int_{\mathbb{R}^{d}}\,h_{\epsilon}^{*}\,d\mu_{\epsilon}\,=\,[\,1+o_{\epsilon}(1)\,]\,Z_{\epsilon}^{-1}\,(2\pi\epsilon)^{d/2}\,e^{-h_{0}/\epsilon}\,\nu_{0}\;.\label{ep73}
\end{equation}
\end{prop}

We heuristically explain that the last proposition holds. Define $\mathcal{G=}\{\boldsymbol{x}:U(\boldsymbol{x})<H-\beta\}$
for small $\beta>0$ and let $\mathcal{G}_{i}=\mathcal{H}_{i}\cap\mathcal{G}$
for $i=0,\,1$. Since the process starting from a point in $\mathcal{G}_{0}$
may touch the set $\mathcal{D}_{\epsilon}(\boldsymbol{m}_{0})$ before
climbing to the saddle point at level $H$, we can expect that $h_{\epsilon}^{*}\simeq1$
on $\mathcal{G}_{0}$. By a similar logic, we have $h_{\epsilon}^{*}\simeq0$
on $\mathcal{G}_{1}$. Since $\mu_{\epsilon}(\mathcal{G}^{c})$ is
negligible by \eqref{etight_U}, we can conclude that the left-hand
side of \eqref{ep73} is approximately equal to $\mu_{\epsilon}(\mathcal{G}_{0})$,
whose asymptotics is given by the right-hand side of \eqref{ep73}.
We turn this into a rigorous argument in Section \ref{sec94} on the
basis of a delicate analysis of the equilibrium potential. 

Now, we formally conclude the proof of Eyring--Kramers formula. 
\begin{proof}[Proof of Theorem \ref{t33}]
The proof is completed by combining Propositions \ref{p71}, \ref{p72},
and \ref{p73}. 
\end{proof}

\subsection{\label{sec72}Strategy to prove Proposition \ref{p72}}

Instead of relying on the traditional approach, which uses the variational
expression of the capacity given by the Dirichlet principle or the
Thomson principle to estimate the capacity, we develop an alternative
strategy in this subsection. This strategy is suitable for non-reversible
cases in that neither the flow structure nor the test flow is used. 

In Section \ref{sec10}, we construct a smooth test function $g_{\epsilon}\in\mathscr{C}_{\mathcal{D}_{\epsilon}(\boldsymbol{m}_{0}),\,\mathcal{U}_{\epsilon}}$
(cf. \eqref{eCab}) satisfying the following property. 
\begin{thm}
\label{t74}We have
\begin{equation}
\epsilon\,\int_{\Omega_{\epsilon}}\,[\,\Phi_{g_{\epsilon}}\cdot\nabla h_{\epsilon}\,]\,d\mu_{\epsilon}\,=\,[\,1+o_{\epsilon}(1)\,]\,\alpha_{\epsilon}\,\omega_{0}+o_{\epsilon}(1)\,[\,\alpha_{\epsilon}\,\textup{cap}_{\epsilon}\,]^{1/2}\;,\label{et741}
\end{equation}
where $\Omega_{\epsilon}=(\,\overline{\mathcal{D}_{\epsilon}(\boldsymbol{m}_{0})}\cup\overline{\mathcal{U}_{\epsilon}}\,)^{c}$. 
\end{thm}

The left-hand side of \eqref{et741} corresponding to $\textrm{cap}_{\epsilon}$
by Proposition \ref{p62} is believed to be equal to the first term
at the right-hand side. Thus, the second error term is somewhat unwanted
and appears just because of a technical reason explained in more detail
at Remark \ref{rem75}. We can however absorb this second error term
to the first error term at the right-hand side of \eqref{et741} as
illustrated in the proof below of Proposition \ref{p72}. Note that
we assume Theorem \ref{t74} at this moment. 
\begin{proof}[Proof of Proposition \ref{p72}]
By Proposition \ref{p62} and Theorem \ref{t74}, we get 
\[
\textrm{cap}_{\epsilon}\,=\,[\,1+o_{\epsilon}(1)\,]\,\alpha_{\epsilon}\,\omega_{0}+o_{\epsilon}(1)\,[\,\alpha_{\epsilon}\,\textup{cap}_{\epsilon}\,]^{1/2}\;.
\]
By dividing both sides by $\alpha_{\epsilon}$ and substituting $r_{\epsilon}=[\,\textrm{cap}_{\epsilon}/\alpha_{\epsilon}\,]^{1/2}$,
we can rewrite the previous identity as 
\[
r_{\epsilon}^{2}\,=\,[\,1+o_{\epsilon}(1)\,]\,\omega_{0}+o_{\epsilon}(1)\,r_{\epsilon}\;.
\]
By solving this quadratic equation in $r_{\epsilon}$, we get $r_{\epsilon}\,=\,[\,1+o_{\epsilon}(1)\,]\,(\omega_{0})^{1/2}$.
Squaring this completes the proof. 
\end{proof}
Now we turn to Theorem \ref{t74}. The core of our strategy is to
find a suitable test function $g_{\epsilon}$ and to compute the left-hand
side of \eqref{et741}\textcolor{red}{. }Indeed, we construct $g_{\epsilon}$
as an approximation of the equilibrium potential $h_{\epsilon}^{*}(\cdot)$
for the adjoint process (cf. \eqref{epab}). The reason is that, by
the divergence theorem, we can write the left-hand side of \eqref{et741}
as
\begin{equation}
\epsilon\,\int_{\Omega_{\epsilon}}\,[\,\Phi_{g_{\epsilon}}\cdot\nabla h_{\epsilon}\,]\,d\mu_{\epsilon}\,=\,-\int_{\Omega_{\epsilon}}\,h_{\epsilon}\,\mathscr{L}_{\epsilon}^{*}g_{\epsilon}\,d\mu_{\epsilon}+(\text{boundary}\text{ terms)}\;.\label{edivf}
\end{equation}
To control the integration on the right-hand side, we try to make
$\mathscr{L}_{\epsilon}^{*}g_{\epsilon}$ as small as possible (cf.
Proposition \ref{p86}); hence, in view of the fact that $\mathcal{L}_{\epsilon}^{*}h_{\epsilon}^{*}\equiv0$
on $\Omega_{\epsilon}$ by the property of the equilibrium potential,
the test function $g_{\epsilon}$ should be an approximation of $h_{\epsilon}^{*}$.
The main contribution for the computation of the left-hand side of
\eqref{edivf} comes from the boundary terms, and relevant computations
are carried out in Proposition \ref{p87}. 

The construction of $g_{\epsilon}$ particularly focuses on the neighborhoods
of the saddle points of $\Sigma_{0}$ as the equilibrium potential
(and hence $g_{\epsilon}$, which is an approximation of the equilibrium
potential) drastically falls from $1$ to $0$ there. We carry out
this construction around the saddle point in Section \ref{sec8} on
the basis of a linearization procedure that is now routine in this
field, e.g., \cite{BEGK1,LMS}. Then, we extend these functions around
the saddle points of $\Sigma_{0}$ to a continuous function on $\mathbb{R}^{d}$
belonging to $\mathscr{C}_{\mathcal{D}_{\epsilon}(\boldsymbol{m}_{0}),\,\mathcal{U}_{\epsilon}}$.
This process will be performed in Section \ref{sec10}, and we finally
obtain $g_{\epsilon}$ in \eqref{egeps}. Then, we prove \eqref{et741}
on the basis of our analysis of the equilibrium potential carried
out in Section \ref{sec9}. 
\begin{rem}[(Comparison with reversible case)]
\textcolor{red}{\label{rem75}} Our strategy is relatively simple
when the underlying process is reversible. In order to get a continuous
test function $g_{\epsilon}$, we need a mollification procedure (cf.
Proposition \ref{p111}), and we must include an additional term $o_{\epsilon}(1)\left[\alpha_{\epsilon}\,\textrm{cap}_{\epsilon}\right]^{1/2}$
in \eqref{et741} to compensate for this additional procedure. However,
for the reversible case, we can get a continuous test function without
this mollification procedure (cf. Remark \ref{rem101}) and we can
prove that 
\[
\epsilon\,\int_{\Omega_{\epsilon}}\,[\,\Phi_{g_{\epsilon}}\cdot\nabla h_{\epsilon}\,]\,d\mu_{\epsilon}\,=\,[\,1+o_{\epsilon}(1)\,]\,\alpha_{\epsilon}\,\omega_{0}\;,
\]
instead of \eqref{et741}; hence, the proof of the Eyring--Kramers
formula is more straightforward. This is the only technical difference
between the reversible and non-reversible models in our methodology. 
\end{rem}

The remainder of this article is devoted to proving Theorem \ref{t74},
and in the course of the proof, Proposition \ref{p73} will also be
demonstrated in Section \ref{sec9}. 

\section{\label{sec8}Construction of Test Function Around Saddle Point}

We explain how we can construct the test function around a saddle
point $\boldsymbol{\sigma}\in\Sigma_{0}$. Section \ref{sec81} presents
a preliminary analysis of the geometry around the saddle point. We
acknowledge that several statements and proofs given in these sections
are similar to those given in \cite{LMS}; however, we try not to
omit the proofs of these results, as the details of the computations
are slightly different owing to the differences between the models.
Then, we construct the test function $p_{\epsilon}^{\boldsymbol{\sigma}}$
on a neighborhood of $\boldsymbol{\sigma}$ in Section \ref{sec83}.
Finally, we explain several computational properties of this test
function in Sections \ref{sec84}--\ref{sec86}. These properties
play crucial role in the proof of Theorem \ref{t74}. 

\subsubsection*{Setting}

In this section, we fix a saddle point $\bm{\sigma}\in\Sigma_{0}$
and simply write $\mathbb{H}=\mathbb{H}^{\boldsymbol{\sigma}}=(\nabla^{2}U)(\boldsymbol{\sigma})$
and $\mathbb{L}=\mathbb{L}^{\bm{\sigma}}=(D\bm{\ell})(\bm{\sigma})$.
Recall that $\mathbb{H}$ has only one negative eigenvalue because
of the Morse lemma. Let $-\lambda_{1},\,\lambda_{2},\,\cdots,\,\lambda_{d}$
denote the eigenvalues of $\mathbb{H}$, where $-\lambda_{1}=-\lambda_{1}^{\boldsymbol{\sigma}}$
denotes the unique negative eigenvalue. Let $\boldsymbol{e}_{k}=\boldsymbol{e}_{k}^{\boldsymbol{\sigma}}$
denote the eigenvector associated with the eigenvalue $\lambda_{k}$
($-\lambda_{k}$ if $k=1$). \textsl{In addition, we assume the direction
of $\bm{e}_{1}$ to be toward $\mathcal{H}_{0}$, i.e., for all sufficiently
small $r>0$, $\bm{\sigma}+r\bm{e}_{1}\in\mathcal{H}_{0}$. }

By Lemma \ref{lem32}, the matrix $\mathbb{H}+\mathbb{L}$ has a unique
negative eigenvalue $-\mu=-\mu^{\bm{\sigma}}$. We can readily observe
that the matrix $\mathbb{H}-\mathbb{L}^{\dagger}$ is similar to $\mathbb{H}+\mathbb{L}$.
To see this, first note that, since $\mathbb{H}\mathbb{L}$ is skew-symmetric
by Lemma \ref{lem45}, we have $\mathbb{H}\mathbb{L}=-(\mathbb{H}\mathbb{L})^{\dagger}=-\mathbb{L}^{\dagger}\mathbb{H}$.
Therefore, we can check the similarity as 
\begin{equation}
\mathbb{H}{}^{-1}\,(\mathbb{H}-\mathbb{L}{}^{\dagger})\,\mathbb{H}\,=\,\mathbb{H}{}^{-1}\,(\mathbb{H}{}^{2}+\mathbb{H}\mathbb{L})\,=\,\mathbb{H}+\mathbb{L}\;.\label{esim}
\end{equation}
Hence, the matrix $\mathbb{H}-\mathbb{L}{}^{\dagger}$ also has a
unique negative eigenvalue $-\mu$, and let $\bm{v}=\boldsymbol{v}^{\boldsymbol{\sigma}}$
denote the unit eigenvector of this matrix associated with the eigenvalue
$-\mu$. \textit{Finally, we assume without loss of generality that
$\bm{v}\cdot\boldsymbol{e}_{1}\ge0$}. Indeed, this cannot be $0$
because of the following lemma, which implies that $(\boldsymbol{v}\cdot\boldsymbol{e}_{1})^{2}>0$. 
\begin{lem}
\label{lem82}We have 
\[
\bm{v}\cdot\mathbb{H}^{-1}\bm{v}\,=\,-\frac{(\bm{v}\cdot\boldsymbol{e}_{1})^{2}}{\lambda_{1}}+\sum_{k=2}^{d}\frac{(\bm{v}\cdot\boldsymbol{e}_{k})^{2}}{\lambda_{k}}\,=\,-\frac{1}{\mu}\,<\,0\;.
\]
\end{lem}

\begin{proof}
The first equality is obvious if we write $\boldsymbol{v}=\sum_{i=1}^{d}a_{i}\boldsymbol{e}_{i}$.
Now, we focus on the second equality. Note that $\mathbb{H}-\mathbb{L}^{\dagger}$
is invertible by Lemma \ref{lem41} and \eqref{esim}. Hence, we can
compute
\begin{align*}
\bm{v}\cdot\mathbb{H}^{-1}\bm{v} & \,=\,\bm{v}\cdot\mathbb{H}^{-1}(\mathbb{H}-\mathbb{L}^{\dagger})(\mathbb{H}-\mathbb{L}^{\dagger})^{-1}\bm{v}\,=\,-\frac{1}{\mu}\bm{v}\cdot\mathbb{H}^{-1}(\mathbb{H}-\mathbb{L}^{\dagger})\bm{v}\\
 & \,=\,-\frac{1}{\mu}\bm{v}\cdot\bm{v}+\frac{1}{\mu}\bm{v}\cdot\mathbb{H}^{-1}\mathbb{L}^{\dagger}\mathbb{H}\mathbb{H}^{-1}\bm{v}\;.
\end{align*}
Since $|\bm{v}|^{2}=1$, the first term in the last line is $-\frac{1}{\mu}$.
On the other hand, since $\mathbb{L}^{\dagger}\mathbb{H}=-(\mathbb{H}\mathbb{L})^{\dagger}$
is skew-symmetric and $\mathbb{H}^{-1}$ is symmetric, the second
term in the last line is $0$. This completes the proof. 
\end{proof}
For two vectors $\boldsymbol{u},\,\boldsymbol{w}\in\mathbb{R}^{d}$,
let $\boldsymbol{u}\otimes\boldsymbol{w}\in\mathbb{R}^{d\times d}$
denote their tensor product, i.e., $(\boldsymbol{u}\otimes\boldsymbol{w})_{ij}=u_{i}w_{j}$,
where $u_{i}$ and $w_{j}$ are the $i$th and $j$th elements of
$\boldsymbol{u}$ and $\boldsymbol{w}$, respectively. The following
Lemma is a consequence of the previous lemma and is similar to \cite[Lemmas 4.1 and 4.2]{LS1}. 
\begin{lem}
\label{lem83}The following hold. 
\begin{enumerate}
\item The matrix $\mathbb{H}+2\mu\,\bm{v}\otimes\bm{v}$ is symmetric positive
definite and $\det\,(\mathbb{H}+2\mu\,\bm{v}\otimes\bm{v})=-\det\mathbb{H}$. 
\item The matrix $\mathbb{H}+\mu\,\bm{v}\otimes\bm{v}$ is symmetric non-negative
definite and $\det\,(\mathbb{H}+\mu\,\bm{v}\otimes\bm{v})=0$. The
null space of the matrix $\mathbb{H}+\mu\,\bm{v}\otimes\bm{v}$ is
one-dimensional and spanned by the vector $\mathbb{H}^{-1}\bm{v}$. 
\end{enumerate}
\end{lem}

\begin{proof}
By a change of coordinate, we can assume that $\boldsymbol{e}_{i}$
is the $i$th standard unit vector of $\mathbb{R}^{d}$ such that
$\mathbb{H}=\text{diag}(-\lambda_{1},\,\lambda_{2},\,\dots,\,\lambda_{d})$.
First, we show that $\mathbb{H}+\mu\,\bm{v}\otimes\bm{v}$ is non-negative
definite. If $v_{2}=\cdots=v_{d}=0$, then, we have $v_{1}^{2}=\mu/\lambda_{1}$
by Lemma \ref{lem82}; thus, $\mathbb{H}+\mu\,\bm{v}\otimes\bm{v}=\text{diag}(0,\lambda_{2},\dots,\lambda_{d})$
is non-negative definite. Otherwise, for $\bm{x}=\sum_{i=1}^{d}x_{i}\bm{e}_{i}\in\mathbb{R}^{d}$,
we can compute 
\[
\boldsymbol{x}\cdot[\,\mathbb{H}+\mu\,\bm{v}\otimes\bm{v}\,]\,\boldsymbol{x}\,=\,-\lambda_{1}x_{1}^{2}+\sum_{k=2}^{d}\lambda_{k}x_{k}^{2}+\mu\,\Big(\,\sum_{i=1}^{d}x_{i}v_{i}\,\Big)^{2}\;.
\]
By minimizing the right-hand side over $x_{1}$ and using Lemma \ref{lem82},
we get 
\[
\sum_{k=2}^{d}\lambda_{k}x_{k}^{2}-\frac{(\,\sum_{k=2}^{d}x_{k}v_{k}\,)^{2}}{\sum_{k=2}^{d}v_{k}^{2}/\lambda_{k}}\ ,
\]
which is non-negative by Cauchy--Schwarz inequality. This proves
that $\mathbb{H}+\mu\,\bm{v}\otimes\bm{v}$ is non-negative definite.
Then, the matrix $\mathbb{H}+2\mu\,\bm{v}\otimes\bm{v}$ is non-negative
definite as well. By the well-known formula 
\begin{equation}
\det\,(\mathbb{A}+\bm{x}\otimes\bm{y})\,=\,(1+\bm{y}^{\dagger}\mathbb{A}^{-1}\bm{x})\det\mathbb{A}\;,\label{edetm}
\end{equation}
along with Lemma \ref{lem82}, we can check that $\det\,(\mathbb{H}+2\mu\,\bm{v}\otimes\bm{v})=-\det\mathbb{H}>0$,
and thus, $\mathbb{H}+2\mu\,\bm{v}\otimes\bm{v}$ is indeed positive
definite. Finally, we investigate the null space of $\mathbb{H}+\mu\bm{v}\otimes\bm{v}$.
Suppose that $\bm{w}\in\mathbb{R}^{d}$ satisfies $(\mathbb{H}+\mu\bm{v}\otimes\bm{v})\bm{w}=0$.
Since $\mathbb{H}$ is invertible, we can rewrite this equation as
$\bm{w}=-\mu(\bm{v}\cdot\bm{w})\mathbb{H}^{-1}\bm{v}$. Hence, the
null space is a subspace of $\langle\,\mathbb{H}^{-1}\bm{v}\,\rangle$.
On the other hand, if $\bm{w}=a\mathbb{H}^{-1}\bm{v}$ for some $a\in\mathbb{R}$,
we can readily check that $(\mathbb{H}+\mu\bm{v}\otimes\bm{v})\bm{w}=\boldsymbol{0}$,
and hence, $\langle\,\mathbb{H}^{-1}\bm{v}\,\rangle$ is indeed the
null space. 
\end{proof}

\subsection{\label{sec81}Neighborhood of saddle points}

In this subsection, we specify the geometry around each saddle point
$\boldsymbol{\sigma}$. Figure \ref{fig3} illustrates the sets appearing
in this section.

\begin{figure}
\begin{centering}
\includegraphics[scale=0.17]{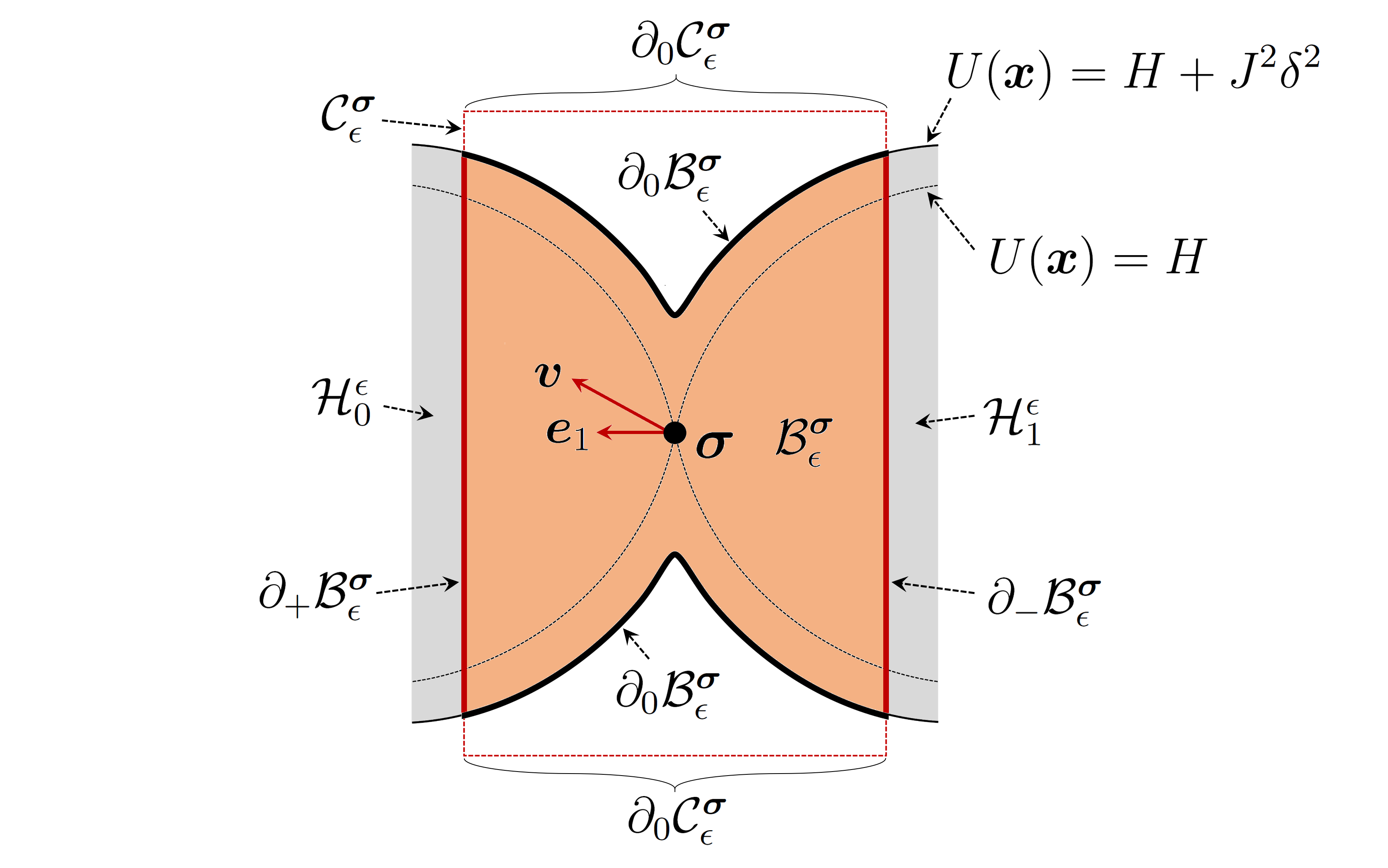}
\par\end{centering}
\caption{\label{fig3}Illustration of the neighborhood structure around a saddle
point $\boldsymbol{\sigma}$.}
\end{figure}

We focus on a neighborhood of $\boldsymbol{\sigma}$ with size of
order $\delta$, which is defined by 
\begin{equation}
\delta\,=\,\delta(\epsilon)\,:=\,\big(\,\epsilon\log\frac{1}{\epsilon}\,\big)^{1/2}\;.\label{edelta}
\end{equation}
Let $J$ be a sufficiently large constant that is independent of $\epsilon$.
There will be several class, e.g., Lemma \ref{lem113}, that require
$J$ to be sufficiently large; we suppose that $J$ satisfies all
such requirements. Define a box $\mathcal{C}_{\epsilon}^{\bm{\sigma}}$
centered at $\bm{\sigma}$ as 
\begin{align*}
\mathcal{C}_{\epsilon}^{\bm{\sigma}}\,=\,\bigg\{\,\bm{\sigma}+\sum_{i=1}^{d}\alpha_{i}\bm{e}_{i}^{\boldsymbol{\sigma}}\in\mathbb{R}^{d}: & -\frac{J\delta}{\lambda_{1}^{1/2}}\leq\alpha_{1}\leq\frac{J\delta}{\lambda_{1}^{1/2}}\\
 & \;\text{and}\,-\frac{2J\delta}{\lambda_{j}^{1/2}}\leq\alpha_{j}\leq\frac{2J\delta}{\lambda_{j}^{1/2}}\,\text{ for }\,2\leq j\leq d\,\bigg\}\;.
\end{align*}
Now, decompose the boundary $\partial\mathcal{C}_{\epsilon}^{\bm{\sigma}}$
into $\partial_{+}\mathcal{C}_{\epsilon}^{\bm{\sigma}}$, $\partial_{-}\mathcal{C}_{\epsilon}^{\bm{\sigma}}$,
and $\partial_{0}\mathcal{C}_{\epsilon}^{\bm{\sigma}}$ such that
\begin{align}
\partial_{\pm}\mathcal{C}_{\epsilon}^{\bm{\sigma}} & \,=\,\Big\{\,\bm{\sigma}+\sum_{i=1}^{d}\alpha_{i}\bm{e}_{i}^{\boldsymbol{\sigma}}\in\mathbb{R}^{d}:\alpha_{1}=\pm\frac{J\delta}{\lambda_{1}^{1/2}}\,\Big\}\;,\label{ebd_C}\\
\partial_{0}\mathcal{C}_{\epsilon}^{\bm{\sigma}} & \,=\,\partial\mathcal{C}_{\epsilon}^{\bm{\sigma}}\setminus(\partial_{+}\mathcal{C}_{\epsilon}^{\bm{\sigma}}\cup\partial_{-}\mathcal{C}_{\epsilon}^{\bm{\sigma}})\;.\nonumber 
\end{align}

\begin{lem}
\label{lem84}For $\bm{x}\in\partial_{0}\mathcal{C}_{\epsilon}^{\bm{\sigma}}$,
we have $U(\bm{x})\geq H+\frac{5}{4}J^{2}\delta^{2}$ for all sufficiently
small $\epsilon>0$. 
\end{lem}

\begin{proof}
For $\bm{x}\in\mathcal{C}_{\epsilon}^{\bm{\sigma}}$, by the Taylor
expansion of $U$ at $\bm{\sigma}$,
\begin{equation}
U(\bm{x})\,=\,H+\frac{1}{2}\,\Big[\,-\lambda_{1}x_{1}^{2}+\sum_{j=2}^{d}\lambda_{j}x_{j}^{2}\,\Big]\,+\,O(\delta^{3})\;.\label{e831}
\end{equation}
For $\bm{x}\in\partial_{0}\mathcal{C}_{\epsilon}^{\bm{\sigma}}$,
$x_{i}=\pm2J\delta/\sqrt{\lambda_{i}}$ for some $2\le i\le d$. Therefore,
\[
-\lambda_{1}x_{1}^{2}+\sum_{j=2}^{d}\lambda_{j}x_{j}^{2}\,\ge\,-J^{2}\delta^{2}+\lambda_{i}\,\Big(\,\frac{2J\delta}{\lambda_{i}^{1/2}}\,\Big)^{2}\,=\,3J^{2}\delta^{2}\;.
\]
Inserting this to \eqref{e831} completes the proof. 
\end{proof}
Hereafter, we assume that $\epsilon>0$ is sufficiently small such
that Lemma \ref{lem84} holds. Define, for $\epsilon>0$, 
\begin{align}
\mathcal{K}_{\epsilon} & \,=\,\{\,\bm{x}\in\mathbb{R}^{d}:U(\bm{x})<H+J^{2}\delta^{2}\,\}\;\;\text{\;and}\;\;\;\mathcal{K}\,=\,\{\,\bm{x}\in\mathbb{R}^{d}:U(\bm{x})<H+J^{2}\,\}\label{emck}
\end{align}
so that $\mathcal{H}\subset\mathcal{K}_{\epsilon}\subset\mathcal{K}$
holds. 

By Lemma \ref{lem84}, the boundary $\partial_{0}\mathcal{C}_{\epsilon}^{\bm{\sigma}}$
does not belong to $\mathcal{K}_{\epsilon}$. The neighborhood of
$\boldsymbol{\sigma}$ in which we focus on the construction is the
set $\mathcal{B}_{\epsilon}^{\bm{\sigma}}=\mathcal{C}_{\epsilon}^{\bm{\sigma}}\cap\mathcal{K}_{\epsilon}$.
Now, we decompose the boundary $\partial\mathcal{B}_{\epsilon}^{\bm{\sigma}}$
into $\partial_{+}\mathcal{B}_{\epsilon}^{\bm{\sigma}}$, $\partial_{-}\mathcal{B}_{\epsilon}^{\bm{\sigma}}$,
and $\partial_{0}\mathcal{B}_{\epsilon}^{\bm{\sigma}}$ such that
\[
\partial_{\pm}\mathcal{B}_{\epsilon}^{\bm{\sigma}}\,=\,\partial_{\pm}\mathcal{C}_{\epsilon}^{\bm{\sigma}}\cap\mathcal{B}_{\epsilon}^{\bm{\sigma}}\;\;\;\text{and}\;\;\;\partial_{0}\mathcal{B}_{\epsilon}^{\bm{\sigma}}\,=\,\partial\mathcal{B}_{\epsilon}\setminus(\,\partial_{+}\mathcal{B}_{\epsilon}^{\bm{\sigma}}\cup\partial_{-}\mathcal{B}_{\epsilon}^{\bm{\sigma}}\,)
\]
so that we have $U(\bm{x})=H+J^{2}\delta^{2}$ for all $\bm{x}\in\partial_{0}\mathcal{B}_{\epsilon}^{\bm{\sigma}}$
by Lemma \ref{lem84}. 

Now, the set $\mathcal{K}_{\epsilon}\setminus\cup_{\bm{\sigma}\in\Sigma_{0}}\mathcal{B}_{\epsilon}^{\bm{\sigma}}$
consists of several connected components. Let $\mathcal{H}_{0}^{\epsilon}$
denote one such component containing $\mathcal{M}_{0}$ and let $\mathcal{H}_{1}^{\epsilon}$
denote the union of the other components such that $\mathcal{M}_{1}\subset\mathcal{H}_{1}^{\epsilon}$.
By our convention on the direction of the vector $\boldsymbol{e}_{1}=\boldsymbol{e}_{1}^{\boldsymbol{\sigma}}$
mentioned earlier in the current section, we have 
\begin{equation}
\partial_{+}\mathcal{B}_{\epsilon}^{\bm{\sigma}}\,\subset\,\partial\mathcal{H}_{0}^{\epsilon}\;\;\;\text{and}\;\;\;\partial_{-}\mathcal{B}_{\epsilon}^{\bm{\sigma}}\,\subset\,\partial\mathcal{H}_{1}^{\epsilon}\;.\label{ebdryis}
\end{equation}
This is illustrated in Figure \ref{fig3}.

\subsection{\label{sec83}Construction of test function around $\boldsymbol{\sigma}$
via linearization procedure}

We construct a function $p_{\epsilon}^{\bm{\sigma}}:\mathbb{R}^{d}\rightarrow\mathbb{R}$
on $\mathcal{B}_{\epsilon}^{\boldsymbol{\sigma}}$, which acts as
a building block for the global construction carried out in the following
sections. As mentioned in Section \ref{sec72}, we would like to build
a function approximating the equilibrium potential $h_{\epsilon}^{*}$
between $\mathcal{D}_{\epsilon}(\boldsymbol{m}_{0})$ and $\mathcal{U}_{\epsilon}$.
Thus, we expect $p_{\epsilon}^{\bm{\sigma}}$ to satisfy $\mathscr{L}_{\epsilon}^{*}p_{\epsilon}^{\bm{\sigma}}\simeq0$,
where $\mathscr{L}_{\epsilon}^{*}$ is defined in \textcolor{red}{\eqref{eL*}}.
To find this function, we linearize the generator $\mathscr{L}_{\epsilon}^{*}$
around $\boldsymbol{\sigma}$ by the first-order Taylor expansion
such that, for smooth $f$, 
\[
\widetilde{\mathscr{L}}_{\epsilon}^{*}f\,=\,\epsilon\,\Delta f(\boldsymbol{x})-\nabla f(\boldsymbol{x})\cdot(\mathbb{H}-\mathbb{L})(\boldsymbol{x})\ ,
\]
and we solve the linearized equation $\mathscr{\widetilde{L}}_{\epsilon}^{*}p_{\epsilon}^{\bm{\sigma}}=0$.
This equation can be explicitly solved using the separation of variables
method. Note that in view of \eqref{ebdryis}, we would like to impose
boundary conditions of the form $p_{\epsilon}^{\bm{\sigma}}\simeq1$
on $\partial_{+}\mathcal{B}_{\epsilon}^{\boldsymbol{\sigma}}$ and
$p_{\epsilon}^{\bm{\sigma}}\simeq0$ on $\partial_{-}\mathcal{B}_{\epsilon}^{\boldsymbol{\sigma}}$.
A test function satisfying all these requirements is given by 
\begin{equation}
p_{\epsilon}^{\bm{\sigma}}(\bm{x})\,=\,\frac{1}{c_{\epsilon}}\int_{-\infty}^{(\bm{x}-\bm{\sigma})\cdot\bm{v}}\,e^{-\frac{\mu}{2\epsilon}t^{2}}\,dt\;\;\;\;;\;\boldsymbol{x}\in\overline{\mathcal{B}_{\epsilon}^{\boldsymbol{\sigma}}}\;,\label{epeps}
\end{equation}
where 
\begin{equation}
c_{\epsilon}\,=\,\int_{-\infty}^{\infty}e^{-\frac{\mu}{2\epsilon}t^{2}}\,dt\,=\,\sqrt{\frac{2\pi\epsilon}{\mu}}\;.\label{ecesig}
\end{equation}
Note that $\bm{v}$ and $\mu$ are defined at the beginning of the
current section. The crucial technical difficulty arises from the
fact that the function $p_{\epsilon}^{\bm{\sigma}}$ is not constant
along the boundary $\partial_{\pm}\mathcal{B}_{\epsilon}^{\boldsymbol{\sigma}}$
unless the dynamics is reversible since $\bm{e}_{1}^{\bm{\sigma}}$
and $\bm{v}$ are linearly independent if $\bm{\ell}\ne0$. This makes
it difficult to patch these functions together. This issue will be
thoroughly investigated in Section \ref{sec10}. 

Since $p_{\epsilon}^{\boldsymbol{\sigma}}$ is smooth on $\mathcal{B}_{\epsilon}^{\boldsymbol{\sigma}}$,
we can define $\Phi_{p_{\epsilon}^{\boldsymbol{\sigma}}}$ on $\mathcal{B}_{\epsilon}^{\boldsymbol{\sigma}}$.
Next, we must investigate the properties of $p_{\epsilon}^{\boldsymbol{\sigma}}$
and $\Phi_{p_{\epsilon}^{\boldsymbol{\sigma}}}$. \textit{For the
simplicity of notation, we assume that $\boldsymbol{\sigma}=\boldsymbol{0}$
in the remainder of the current section. }

\subsection{\label{sec84}Negligibility of $\mathscr{L}_{\epsilon}^{*}p_{\epsilon}^{\boldsymbol{\sigma}}$
on $\mathcal{B}_{\epsilon}^{\boldsymbol{\sigma}}$}

Our construction of $p_{\epsilon}^{\bm{\sigma}}$ suggests that $\mathscr{L}_{\epsilon}^{*}p_{\epsilon}^{\bm{\sigma}}$
is small on $\mathcal{B}_{\epsilon}^{\boldsymbol{\sigma}}$. The next
lemma precisely quantifies this heuristic observation. 
\begin{notation}
Let $C>0$ denote a positive constant independent of $\epsilon$ and
$\boldsymbol{x}$. Different appearances of $C$ may express different
values. 
\end{notation}

\begin{prop}
\label{p86}We have $\int_{\mathcal{B}_{\epsilon}^{\boldsymbol{\sigma}}}\,|\mathscr{L}_{\epsilon}^{*}p_{\epsilon}^{\boldsymbol{\sigma}}|\,d\mu_{\epsilon}\,=\,o_{\epsilon}(1)\,\alpha_{\epsilon}$.
\end{prop}

\begin{proof}
By inserting the explicit formula \eqref{epeps}, we get 
\begin{align*}
(\mathscr{L}_{\epsilon}^{*}p_{\epsilon}^{\boldsymbol{\sigma}})(\bm{x}) & \,=\,c_{\epsilon}^{-1}\,e^{-\frac{\mu}{2\epsilon}(\bm{x}\cdot\bm{v})^{2}}\,\Big[\,-(\nabla U-\boldsymbol{\ell})(\bm{x})\cdot\bm{v}-\mu(\bm{x}\cdot\bm{v})\,\Big]\;.
\end{align*}
Now, by applying the Taylor expansion of $\nabla U$ and $\ell$ around
$\boldsymbol{\sigma}$, for $\boldsymbol{x}\in\mathcal{B}_{\epsilon}^{\boldsymbol{\sigma}}$,
\begin{align*}
(\mathscr{L}_{\epsilon}^{*}p_{\epsilon}^{\boldsymbol{\sigma}})(\bm{x}) & \,=\,-c_{\epsilon}^{-1}\,e^{-\frac{\mu}{2\epsilon}(\bm{x}\cdot\bm{v})^{2}}\,\Big[\,\{\,(\mathbb{H}-\mathbb{L})\bm{x}+O(\delta^{2})\,\}\cdot\bm{v}+\mu(\bm{x}\cdot\bm{v})\,\Big]\\
 & \,=\,-c_{\epsilon}^{-1}\,e^{-\frac{\mu}{2\epsilon}(\bm{x}\cdot\bm{v})^{2}}\,\Big[\,\bm{x}\cdot(-\mu\bm{v})+\mu(\bm{x}\cdot\bm{v})+O(\delta^{2})\,\Big]\;,
\end{align*}
where the last line follows from the fact that $\bm{v}$ is an eigenvector
of $(\mathbb{H}-\mathbb{L})^{\dagger}=\mathbb{H}-\mathbb{L}^{\dagger}$
associated with the eigenvalue $-\mu$. Now, recall $c_{\epsilon}$
from \eqref{ecesig} to deduce that, for some constant $C>0$, 
\[
|\,(\mathscr{L}_{\epsilon}^{*}p_{\epsilon}^{\boldsymbol{\sigma}})(\bm{x})\,|\,\leq\,\frac{C\,\delta^{2}}{\epsilon^{1/2}}\,e^{-\frac{\mu}{2\epsilon}(\bm{x}\cdot\bm{v})^{2}}\;.
\]
By the second-order Taylor expansion, we can write 
\[
U(\bm{x})\,=\,H+\frac{1}{2}\bm{x}\cdot\mathbb{H}\bm{x}+O(\delta^{3})\;\;\;\text{for\;}\bm{x}\in\mathcal{B}_{\epsilon}^{\boldsymbol{\sigma}}\;.
\]
This expansion will be repeatedly used in the subsequent computation.
Since $e^{-O(\delta^{3})/\epsilon}=1+o_{\epsilon}(1)$ by the definition
\eqref{edelta} of $\delta$, we can conclude that 
\begin{equation}
\int_{\mathcal{B}_{\epsilon}^{\boldsymbol{\sigma}}}\,|\,\mathscr{L}_{\epsilon}^{*}p_{\epsilon}^{\boldsymbol{\sigma}}\,|\,d\mu_{\epsilon}\,\leq\,C\frac{\delta^{2}}{Z_{\epsilon}\epsilon^{1/2}}e^{-H/\epsilon}\int_{\mathcal{B}_{\epsilon}^{\boldsymbol{\sigma}}}\,e^{-\frac{1}{2\epsilon}\bm{x}\cdot(\mathbb{H}+\mu\bm{v}\otimes\bm{v})\bm{x}}\,d\bm{x}\;.\label{e95}
\end{equation}
Now, the estimation of the last integral remains. This part is similar
to \cite[Lemma 8.7]{LMS}; however, we repeat the argument here for
the completeness of the proof. By part (2) of Lemma \ref{lem83},
let $\rho_{1}=0$ and $\rho_{2}\,,\dots,\,\rho_{d}>0$ denote the
eigenvalues of $\mathbb{H}+\mu\bm{v}\otimes\bm{v}$ and let $\bm{u}_{1},\,\dots,\,\boldsymbol{u}_{d}$
denote the corresponding unit eigenvectors. Let $\langle\bm{u}_{2},\cdots,\bm{u}_{d}\rangle$
denote the subspace of $\mathbb{R}^{d}$ spanned by vectors $\bm{u}_{2},\cdots,\bm{u}_{d}$.
Since $\mathcal{B}_{\epsilon}^{\boldsymbol{\sigma}}\subset\mathcal{C}_{\epsilon}^{\boldsymbol{\sigma}}$,
there exists $M>0$ such that 
\[
\mathcal{B}_{\epsilon}^{\boldsymbol{\sigma}}\,\subset\,\bigcup_{a:|a|\leq M\delta}(\,a\bm{u}_{1}+\langle\,\bm{u}_{2},\,\dots,\bm{u}_{d}\,\rangle\,)\;.
\]
Hence, along with the change of variables $\bm{x}=\sum y_{i}\bm{u}_{i}$,
we can bound the last integral in \eqref{e95} by 
\[
\int_{-M\delta}^{M\delta}\,\Big[\,\int_{\mathbb{R}^{d-1}}\exp\,\Big\{\,-\frac{1}{2\epsilon}\sum_{k=2}^{d}\rho_{k}y_{k}^{2}\,\Big\}\,dy_{2}\cdots\,dy_{d}\,\Big]\,dy_{1}\,=\,C\,\delta\,\epsilon{}^{(d-1)/2}\;.
\]
By inserting this into \eqref{e95}, we get $\int_{\mathcal{B}_{\epsilon}^{\bm{\sigma}}}\,|\,\mathscr{L}_{\epsilon}^{*}p_{\epsilon}^{\boldsymbol{\sigma}}\,|\,d\mu_{\epsilon}\leq C\,\delta^{3}\,\epsilon^{-1}\,\alpha_{\epsilon}.$
Since $\delta^{3}\,\epsilon^{-1}=o_{\epsilon}(1)$, the proof is completed. 
\end{proof}

\subsection{\label{sec85}Property of $\Phi_{p_{\epsilon}^{\boldsymbol{\sigma}}}$
at the boundary of $\mathcal{B}_{\epsilon}^{\boldsymbol{\sigma}}$}

Next, we prove the following property of the vector field $\Phi_{p_{\epsilon}^{\boldsymbol{\sigma}}}$.
Recall $\omega^{\boldsymbol{\sigma}}$ from \eqref{eEKconst}. 
\begin{prop}
\label{p87}We have 
\begin{equation}
\epsilon\,\int_{\partial_{+}\mathcal{B}_{\epsilon}^{\boldsymbol{\sigma}}}\,\Big[\,\Big(\,\Phi_{p_{\epsilon}^{\boldsymbol{\sigma}}}-\frac{1}{\epsilon}\bm{\ell}\,\Big)\cdot\boldsymbol{e}_{1}\,\Big]\,\sigma(d\mu_{\epsilon})\,=\,[\,1+o_{\epsilon}(1)\,]\,\alpha_{\epsilon}\,\omega^{\boldsymbol{\sigma}}\;.\label{ep861}
\end{equation}
\end{prop}

This estimate is indeed the key estimate in the proof of Theorem \ref{t74}.
The left-hand side of \eqref{ep861} corresponds to the boundary term
in \eqref{edivf}. The proof of this proposition is slightly complicated.
Hence, we first establish some technical lemmas. For simplicity of
notation, we assume in this subsection that $\boldsymbol{e}_{i}$
is the $i$th standard normal vector of $\mathbb{R}^{d}$; hence,
we can write
\begin{align*}
\mathbb{H} & \,=\,\textrm{diag}(\,-\lambda_{1},\,\lambda_{2},\,\dots,\,\lambda_{d}\,)\;\;\;\text{and\;\;\;}\boldsymbol{v}\,=\,(\,v_{1},\,\dots,\,v_{d}\,)\;.
\end{align*}

\subsubsection*{Change of coordinate on $\partial_{+}\mathcal{B}_{\epsilon}^{\boldsymbol{\sigma}}$}

First, we introduce a change of coordinate that maps $\partial_{+}\mathcal{B}_{\epsilon}^{\boldsymbol{\sigma}}$
to a subset of $\mathbb{R}^{d-1}$ to simplify the integration in
\eqref{ep861}

For $\mathbb{A}\in\mathbb{R}^{d\times d}$ and $\boldsymbol{u}=(\,u_{1},\,\dots,\,u_{d}\,)\in\mathbb{R}^{d}$,
define $\widetilde{\mathbb{A}}\in\mathbb{R}^{(d-1)\times(d-1)}$ and
$\widetilde{\boldsymbol{u}}\in\mathbb{R}^{d-1}$ as 
\begin{equation}
\widetilde{\mathbb{A}}\,=\,(\mathbb{A}_{i,\,j})_{2\leq i,\,j\leq d}\;\;\;\text{and\;\;\;}\widetilde{\boldsymbol{u}}\,=\,(\,u_{2},\dots,u_{d}\,)\;,\label{etilde}
\end{equation}
respectively. It is important to select a point of $\partial_{+}\mathcal{B}_{\epsilon}^{\boldsymbol{\sigma}}$
corresponding to the origin of $\mathbb{R}^{d-1}$ to simplify our
computation. To this end, define $\boldsymbol{\gamma}=(\gamma_{2},\,\dots,\,\gamma_{d})\in\mathbb{R}^{d-1}$
as
\begin{equation}
\gamma_{k}\,=\,\frac{\lambda_{1}^{1/2}}{v_{1}}\cdot\frac{v_{k}}{\lambda_{k}}J\delta\;\;\;\;;\;k=2,\,\dots,\,d\;.\label{egamma}
\end{equation}
Note that $v_{1}\neq0$ by Lemma \ref{lem82}. Define a map $\Pi_{\epsilon}:\partial_{+}\mathcal{B}_{\epsilon}^{\boldsymbol{\sigma}}\rightarrow\mathbb{R}^{d-1}$
that represents the change of coordinate as 
\begin{equation}
\Pi_{\epsilon}(\boldsymbol{x})\,=\,\widetilde{\boldsymbol{x}}+\boldsymbol{\gamma}\;.\label{eProj}
\end{equation}
Our careful selection of $\boldsymbol{\gamma}$ ensures that this
map simplifies the computation of the crucial quadratic form. 
\begin{lem}
\label{lem88}For all $\boldsymbol{x}\in\partial_{+}\mathcal{B}_{\epsilon}^{\boldsymbol{\sigma}}$,
we have 
\[
\bm{x}\cdot(\,\mathbb{H}+\mu\,\bm{v}\otimes\bm{v}\,)\bm{x}\,=\,\Pi_{\epsilon}(\boldsymbol{x})\cdot(\,\widetilde{\mathbb{H}}+\mu\,\widetilde{\boldsymbol{v}}\otimes\widetilde{\boldsymbol{v}}\,)\,\Pi_{\epsilon}(\boldsymbol{x})\;.
\]
\end{lem}

\begin{proof}
Fix $\boldsymbol{x=}\big(\,\frac{J\delta}{\lambda_{1}^{1/2}},\,x_{2},\,\dots,\,x_{d}\,\big)\in\partial_{+}\mathcal{B}_{\epsilon}^{\boldsymbol{\sigma}}$
and write $\Pi_{\epsilon}(\boldsymbol{x})=\boldsymbol{y}=(y_{2},\,\dots,\,y_{d})$.
Then, by Lemma \ref{lem82}, we can write 
\[
\bm{x}\cdot\bm{v}\,=\,\frac{J\,\delta}{\lambda_{1}^{1/2}}v_{1}+\sum_{k=2}^{d}(y_{k}-\gamma_{k})v_{k}\,=\,\bm{y}\cdot\widetilde{\boldsymbol{v}}+\frac{J\,\delta\,\lambda_{1}^{1/2}}{\mu\,v_{1}}\;.
\]
Thus, we can write $\bm{x}\cdot(\mathbb{H}+\mu\bm{v}\otimes\bm{v})\bm{x}$
as
\[
-\lambda_{1}x_{1}^{2}+\sum_{k=2}^{d}\lambda_{k}x_{k}^{2}+\mu\,\Big(\,\bm{y}\cdot\widetilde{\boldsymbol{v}}+\frac{J\delta\lambda_{1}^{1/2}}{\mu v_{1}}\,\Big)^{2}\,=\,\bm{y}\cdot(\widetilde{\mathbb{H}}+\mu\widetilde{\boldsymbol{v}}\otimes\widetilde{\boldsymbol{v}})\bm{y}.
\]
The correction vector $\boldsymbol{\gamma}$ is designed to clear
the linear terms and constant term here. 
\end{proof}
We can now show that the image of $\Pi_{\epsilon}(\partial_{+}\mathcal{B}_{\epsilon}^{\boldsymbol{\sigma}})$
is comparable with a ball centered at the origin with a radius of
order $\delta$. 
\begin{lem}
\label{lem89}There exist constants $r,\,R>0$ such that 
\begin{equation}
\mathcal{D}_{r\delta}^{(d-1)}(\boldsymbol{0})\,\subset\,\Pi_{\epsilon}(\partial_{+}\mathcal{B}_{\epsilon}^{\boldsymbol{\sigma}})\,\subset\,\mathcal{D}_{R\delta}^{(d-1)}(\boldsymbol{0})\ ,\label{e910}
\end{equation}
where $\mathcal{D}_{a}^{(d-1)}(\boldsymbol{0})$ denotes a sphere
on $\mathbb{R}^{d-1}$ centered at the origin with radius $a$. 
\end{lem}

\begin{proof}
Since $\partial_{+}\mathcal{B}_{\epsilon}^{\boldsymbol{\sigma}}\subset\mathcal{D}_{C\delta}^{(d-1)}(\boldsymbol{0})$
for sufficiently large $C>0$ and $|\boldsymbol{\gamma}|=O(\delta)$,
the existence of $R$ is immediate from the definition of $\Pi_{\epsilon}$. 

Now we focus on the first inclusion of \eqref{e910}. For $\bm{\gamma}\in\mathbb{R}^{d-1}$
defined in \eqref{egamma}, we write 
\begin{align*}
\mathcal{P}_{\delta} & \,=\,\Big\{\,\bm{x}\in\mathbb{R}^{d}:x_{1}=\frac{J\delta}{\lambda_{1}^{1/2}}\,\Big\}\,\subset\mathbb{R}^{d}\;\;\;\text{and\;\;\;}\overline{\boldsymbol{\gamma}}\,=\,\Big(\,\frac{J\delta}{\lambda_{1}^{1/2}},\,-\gamma_{2},\,\dots,\,-\gamma_{d}\,\Big)\in\mathcal{P}_{\delta}\;.
\end{align*}
Then, by the Taylor expansion and Lemma \ref{lem82}, we can check
that 
\begin{equation}
U(\overline{\boldsymbol{\gamma}})\,=\,H-\frac{\lambda_{1}}{2\,\mu\,v_{1}^{2}}\,J^{2}\,\delta^{2}+O(\delta^{3})\,<\,H-c_{0}J^{2}\delta^{2}\label{e881}
\end{equation}
for all sufficiently small $\epsilon>0$, provided that we take $c_{0}$
to be sufficiently small. Therefore, there exists $r>0$ such that
$\mathcal{D}_{r\delta}(\overline{\boldsymbol{\gamma}})\cap\mathcal{P}_{\delta}\subset\partial_{+}\mathcal{B}_{\epsilon}^{\boldsymbol{\sigma}}$.
Since $\Pi_{\epsilon}(\overline{\boldsymbol{\gamma}})=\bm{0}$, we
have $\mathcal{D}_{r\delta}^{(d-1)}(\boldsymbol{0})=\Pi_{\epsilon}(\mathcal{D}_{r\delta}^{(d-1)}(\overline{\boldsymbol{\gamma}})\cap\mathcal{P}_{\delta})$.
This completes the proof. 
\end{proof}
Now, we present three auxiliary lemmas (Lemmas \ref{lem810}, \ref{lem811},
and \ref{lem812}) that will be used in several instances including
the proof of Proposition \ref{p87}. The proofs of these technical
results are deferred to the next subsection. 
\begin{lem}
\label{lem810}The matrix $\widetilde{\mathbb{H}}+\mu\,\widetilde{\bm{v}}\otimes\widetilde{\bm{v}}$
is positive definite and \textbf{
\[
\det\,(\,\widetilde{\mathbb{H}}+\mu\widetilde{\bm{v}}\otimes\widetilde{\bm{v}}\,)\,=\,\mu\,\frac{v_{1}^{2}}{\lambda_{1}}\,\prod_{k=2}^{d}\lambda_{k}\;.
\]
}
\end{lem}

\begin{proof}
By \eqref{edetm} and Lemma \ref{lem82}, 
\[
\det\,(\,\widetilde{\mathbb{H}}+\mu\widetilde{\bm{v}}\otimes\widetilde{\bm{v}}\,)\,=\,(1+\mu\widetilde{\bm{v}}^{\dagger}\,\widetilde{\mathbb{H}}^{-1}\,\widetilde{\bm{v}})\,\det\widetilde{\mathbb{H}}\,=\,\frac{\mu v_{1}^{2}}{\lambda_{1}}\,\det\widetilde{\mathbb{H}}\;.
\]
\end{proof}
Recall $\partial_{+}\mathcal{C}_{\epsilon}^{\boldsymbol{\sigma}}$
from \eqref{ebd_C} and define, for $a>0$,
\begin{align}
 & \partial_{+}^{1,\,a}\mathcal{C}_{\epsilon}^{\boldsymbol{\sigma}}\,=\,\left\{ \boldsymbol{x}\in\partial_{+}\mathcal{C}_{\epsilon}^{\boldsymbol{\sigma}}:\bm{x}\cdot\bm{v}\ge aJ\delta\right\} \ ,\label{ebd_C1}\\
 & \partial_{+}^{2,\,a}\mathcal{C}_{\epsilon}^{\boldsymbol{\sigma}}\,=\,\left\{ \boldsymbol{x}\in\partial_{+}\mathcal{C}_{\epsilon}^{\boldsymbol{\sigma}}:U(\boldsymbol{x})\geq H+aJ^{2}\delta^{2}\right\} \ .\label{ebd_C2}
\end{align}

\begin{lem}
\label{lem811}There exists $a_{0}>0$ such that, for all $a\in(0,\,a_{0})$,
\[
\partial_{+}^{1,\,a}\mathcal{C}_{\epsilon}^{\boldsymbol{\sigma}}\cup\partial_{+}^{2,\,a}\mathcal{C}_{\epsilon}^{\boldsymbol{\sigma}}\,=\,\partial_{+}\mathcal{C}_{\epsilon}^{\bm{\sigma}}\;.
\]
\end{lem}

Hereafter, the constant $a_{0}$ always refers to the one in the previous
lemma. For $a>0$, we write 
\begin{align}
 & \partial_{+}^{1,\,a}\mathcal{B}_{\epsilon}^{\boldsymbol{\sigma}}\,=\,\partial_{+}\mathcal{B}_{\epsilon}^{\boldsymbol{\sigma}}\cap\partial_{+}^{1,\,a}\mathcal{C}_{\epsilon}^{\boldsymbol{\sigma}}\,=\,\left\{ \boldsymbol{x}\in\partial_{+}\mathcal{B}_{\epsilon}^{\boldsymbol{\sigma}}:\bm{x}\cdot\bm{v}\ge aJ\delta\right\} \;,\label{ebd_B1}\\
 & \partial_{+}^{2,\,a}\mathcal{B}_{\epsilon}^{\boldsymbol{\sigma}}\,=\,\partial_{+}\mathcal{B}_{\epsilon}^{\boldsymbol{\sigma}}\cap\partial_{+}^{2,\,a}\mathcal{C}_{\epsilon}^{\boldsymbol{\sigma}}\,=\,\left\{ \boldsymbol{x}\in\partial_{+}\mathcal{B}_{\epsilon}^{\boldsymbol{\sigma}}:U(\boldsymbol{x})\geq H+aJ^{2}\delta^{2}\right\} ;\label{ebd_B2}
\end{align}
hence, we have 
\begin{equation}
\partial_{+}\mathcal{B}_{\epsilon}^{\bm{\sigma}}\,=\,\partial_{+}^{1,\,a}\mathcal{B}_{\epsilon}^{\boldsymbol{\sigma}}\cup\partial_{+}^{2,\,a}\mathcal{B}_{\epsilon}^{\boldsymbol{\sigma}}\label{decbdrb}
\end{equation}
for all $a\in(0,\,a_{0})$ by the previous lemma. Now, we introduce
the last lemma.
\begin{lem}
\label{lem812}Let $\mathbb{D}$ be a positive-definite $(d-1)\times(d-1)$
matrix, Then, for all $\boldsymbol{u}_{1},\,\boldsymbol{u}_{2}\in\mathbb{R}^{d-1}$
and $c\in(0,\,1)$, we have 
\[
\int_{\Pi_{\epsilon}(\partial_{+}\mathcal{B}_{\epsilon}^{\bm{\sigma}})\cap\{\bm{y}\in\mathbb{R}^{d-1}:\bm{y}\cdot\boldsymbol{u}_{1}\geq-c\delta\}}\frac{\bm{y}\cdot\boldsymbol{u}_{2}+\delta}{\bm{y}\cdot\boldsymbol{u}_{1}+\delta}\,e^{-1/(2\epsilon)\,\boldsymbol{y}\cdot\mathbb{D}\bm{y}}\,d\bm{y}\,=\,[\,1+o_{\epsilon}(1)\,]\,\frac{(2\pi\epsilon)^{(d-1)/2}}{\sqrt{\det(\mathbb{D})}}\;.
\]
\end{lem}

Now, we are ready to prove Proposition \ref{p87}. 
\begin{proof}[Proof of Proposition \ref{p87}]
In view of the definition of $\Phi_{p_{\epsilon}^{\boldsymbol{\sigma}}}$
given in \eqref{ePhi}, we can write 
\begin{align}
\epsilon\,\int_{\partial_{+}\mathcal{B}_{\epsilon}^{\bm{\sigma}}}\,\Big[\,\Phi_{p_{\epsilon}^{\boldsymbol{\sigma}}}-\frac{1}{\epsilon}\bm{\ell}\,\Big]\,\cdot\boldsymbol{e}_{1}\,\sigma(d\mu_{\epsilon}) & \,=\,I_{1}-I_{2}\;,\label{e917}
\end{align}
where 
\begin{align*}
I_{1} & \,=\,\epsilon\,\int_{\partial_{+}\mathcal{B}_{\epsilon}^{\bm{\sigma}}}\nabla p_{\epsilon}^{\bm{\sigma}}(\bm{x})\cdot\boldsymbol{e}_{1}\,\sigma(d\mu_{\epsilon})\;\;\;\text{and\;\;\;}I_{2}\,=\,\int_{\partial_{+}\mathcal{B}_{\epsilon}^{\bm{\sigma}}}\,(1-p_{\epsilon}^{\bm{\sigma}})\,(\boldsymbol{\ell}\cdot\boldsymbol{e}_{1})\,\sigma(d\mu_{\epsilon})\;.
\end{align*}
First, we compute $I_{1}$. By the explicit form of $p_{\epsilon}^{\boldsymbol{\sigma}}$
and the Taylor expansion of $U$, we can write 
\begin{equation}
I_{1}\,=\,[\,1+o_{\epsilon}(1)\,]\,v_{1}\,\frac{\epsilon}{Z_{\epsilon}}\sqrt{\frac{\mu}{2\pi\epsilon}}e^{-\frac{H}{\epsilon}}\int_{\partial_{+}\mathcal{B}_{\epsilon}}e^{-\frac{1}{2\epsilon}\bm{x}\cdot(\mathbb{H}+\mu\bm{v}\otimes\bm{v})\bm{x}}\sigma(d\bm{x})\;.\label{eI1_1}
\end{equation}
By the change of variables $\boldsymbol{y}=\Pi_{\epsilon}(\boldsymbol{x})$,
the last integral can be expressed as 
\[
\int_{\Pi_{\epsilon}(\partial_{+}\mathcal{B}_{\epsilon})}e^{-\frac{1}{2\epsilon}\bm{y}\cdot(\widetilde{\mathbb{H}}+\mu\widetilde{\bm{v}}\otimes\widetilde{\bm{v}})\bm{y}}\,d\bm{y}\,=\,[\,1+o_{\epsilon}(1)\,]\,\frac{(2\pi\epsilon)^{(d-1)/2}}{\sqrt{\textrm{det\,}(\,\widetilde{\mathbb{H}}+\mu\widetilde{\bm{v}}\otimes\widetilde{\bm{v}}\,)}}\;,
\]
where the equality follows from the change of variables $\boldsymbol{z}=\epsilon^{-1/2}\boldsymbol{y}$
and Lemma \ref{lem89}. Summing up, we get 
\begin{equation}
I_{1}\,=\,[\,1+o_{\epsilon}(1)\,]\,\frac{v_{1}\,\mu^{1/2}\,\alpha_{\epsilon}}{2\pi\sqrt{\textrm{det\,}(\,\widetilde{\mathbb{H}}+\mu\widetilde{\bm{v}}\otimes\widetilde{\bm{v}}\,)}}\;.\label{eI1_2}
\end{equation}

Next, we consider $I_{2}$. Let us take $a\in(0,\,a_{0})$, where
$a_{0}$ is the constant in Lemma \ref{lem811}, and decompose
\begin{equation}
I_{2}\,=\,I_{2,\,1}+I_{2,\,2}\;,\label{edecI2}
\end{equation}
where 
\[
I_{2,\,1}\,=\,\int_{\partial_{+}^{1,\,a}\mathcal{B}_{\epsilon}^{\boldsymbol{\sigma}}}\,\,(1-p_{\epsilon}^{\bm{\sigma}})\,(\boldsymbol{\ell}\cdot\boldsymbol{e}_{1})\,\sigma(d\mu_{\epsilon})\;,\;\;\;I_{2,\,2}\,=\,\int_{\partial_{+}\mathcal{B}_{\epsilon}^{\boldsymbol{\sigma}}\setminus\partial_{+}^{1,\,a}\mathcal{B}_{\epsilon}^{\boldsymbol{\sigma}}}\,\,(1-p_{\epsilon}^{\bm{\sigma}})\,(\boldsymbol{\ell}\cdot\boldsymbol{e}_{1})\,\sigma(d\mu_{\epsilon})\;.
\]
First, we compute $I_{2,\,1}$. Recall the elementary inequality 
\begin{equation}
\frac{b}{b^{2}+1}e^{-b^{2}/2}\,\leq\,\int_{b}^{\infty}e^{-t^{2}/2}\,dt\,\leq\,\frac{1}{b}e^{-b^{2}/2}\;\;\;\text{for }b>0\;.\label{e_eleeq}
\end{equation}
Now, for $\boldsymbol{x}\in\partial_{+}^{1,\,a}\mathcal{B}_{\epsilon}^{\boldsymbol{\sigma}}$,
since we have $\sqrt{\frac{\mu}{\epsilon}}(\bm{x}\cdot\bm{v})\to\infty$
as $\epsilon\to0$, we obtain from the definition of $p_{\epsilon}^{\bm{\sigma}}$
and \eqref{e_eleeq} that 
\begin{equation}
1-p_{\epsilon}^{\bm{\sigma}}(\bm{x})\,=\,[\,1+o_{\epsilon}(1)\,]\,\frac{\epsilon^{1/2}}{(2\pi\mu)^{1/2\,}(\bm{x}\cdot\bm{v})}\,\exp\left\{ -\frac{\mu}{2\epsilon}(\bm{x}\cdot\bm{v})^{2}\right\} .\label{e1-p}
\end{equation}
By the Taylor expansion of $\boldsymbol{\ell}$, we have 
\begin{equation}
\boldsymbol{\ell}(\bm{x})\cdot\bm{e}_{1}\,=\,\mathbb{L}\bm{x}\cdot\bm{e}_{1}+O(\delta^{2})\;.\label{e_Tayell}
\end{equation}
Our plan is to insert \eqref{e1-p} and \eqref{e_Tayell} into $I_{2,\,1}$
to complete the proof. To this end, we first explain that we can ignore
the $O(\delta^{2})$ term in \eqref{e_Tayell}. By \eqref{e1-p},
the Taylor expansion of $U$, and Lemma \ref{lem83}, we have 
\begin{align}
 & \Big|\,\delta^{2}\,\int_{\partial_{+}^{1,\,a}\mathcal{B}_{\epsilon}^{\boldsymbol{\sigma}}}\,(1-p_{\epsilon}^{\bm{\sigma}})\,\sigma(d\mu_{\epsilon})\,\Big|\nonumber \\
\le\; & C\,\frac{\delta\,\epsilon^{1/2}}{Z_{\epsilon}}\,e^{-H/\epsilon}\int_{\partial_{+}^{1,\,a}\mathcal{B}_{\epsilon}^{\boldsymbol{\sigma}}}\exp\,\Big\{\,-\frac{\mu}{2\epsilon}\bm{x}\cdot(\mathbb{H}+\mu\bm{v}\otimes\bm{v})\bm{x}\,\Big\}\,\sigma(d\bm{x})\nonumber \\
\le\; & C\,\frac{\delta\,\epsilon^{1/2}}{Z_{\epsilon}}\,e^{-H/\epsilon}\,\sigma(\partial_{+}\mathcal{B}_{\epsilon}^{\boldsymbol{\sigma}})\,=\,C\frac{\delta^{d}\,\epsilon^{1/2}}{Z_{\epsilon}}\,e^{-H/\epsilon}\,=\,o_{\epsilon}(1)\,\alpha_{\epsilon}\;.\label{e924}
\end{align}
Hence, by combining \eqref{e1-p}, \eqref{e_Tayell}, and \eqref{e924},
we can write 
\begin{equation}
I_{2,\,1}\,=\,o_{\epsilon}(1)\,\alpha_{\epsilon}+\frac{[\,1+o_{\epsilon}(1)\,]\,\alpha_{\epsilon}\,\epsilon}{(2\pi\epsilon)^{(d+1)/2}\,\mu^{1/2}}\,\int_{\partial_{+}^{1,\,a}\mathcal{B}_{\epsilon}^{\boldsymbol{\sigma}}}\,e^{-\frac{1}{2\epsilon}\,\bm{x}\cdot[\,\mathbb{H}+\mu\bm{v}\otimes\bm{v}\,]\bm{x}}\,\frac{\mathbb{L}\bm{x}\cdot\bm{e}_{1}}{\bm{x}\cdot\bm{v}}\,\sigma(d\bm{x})\;.\label{eI21}
\end{equation}
By the change of variables $\boldsymbol{y}=\Pi_{\epsilon}(\boldsymbol{x})$
and Lemma \ref{lem88}, we can write the last integral as 
\begin{align*}
 & \int_{\Pi(\partial_{+}\mathcal{B}_{\epsilon})\cap\{\bm{y}:\bm{y}\cdot\widetilde{\boldsymbol{v}}\geq c'J\delta\}}e^{-\frac{1}{2\epsilon}\bm{y}\cdot[\widetilde{\mathbb{H}}+\mu\widetilde{\boldsymbol{v}}\otimes\widetilde{\boldsymbol{v}}]\bm{y}}\,\frac{\bm{y}\cdot\widetilde{\mathbb{L}}^{\dagger}\widetilde{\boldsymbol{v}}-\frac{J\delta\lambda_{1}}{v_{1}}\sum_{k=2}^{d}\mathbb{L}_{1k}\frac{v_{k}}{\lambda_{k}}}{\bm{y}\cdot\widetilde{\boldsymbol{v}}+\frac{J\delta\lambda_{1}}{\mu v_{1}}}d\bm{y}\\
=\; & (-\mu\mathbb{L}\mathbb{H}^{-1}\bm{v})\int_{\Pi(\partial_{+}\mathcal{B}_{\epsilon})\cap\{\bm{y}:\bm{y}\cdot\widetilde{\boldsymbol{v}}\geq c'J\delta\}}e^{-\frac{1}{2\epsilon}\bm{y}\cdot[\widetilde{\mathbb{H}}+\mu\widetilde{\boldsymbol{v}}\otimes\widetilde{\boldsymbol{v}}]\bm{y}}\,\frac{\bm{y}\cdot\boldsymbol{w}+\frac{J\delta\lambda_{1}}{\mu v_{1}}}{\bm{y}\cdot\widetilde{\boldsymbol{v}}+\frac{J\delta\lambda_{1}}{\mu v_{1}}}d\bm{y}
\end{align*}
for some $\boldsymbol{w}\in\mathbb{R}^{d-1}$ and $c'=a-\frac{\lambda_{1}}{\mu v_{1}}$.
Take $a\in(0,\,a_{0})$ to be sufficiently small such that $c'<0$
(which is possible by the statement of Lemma \ref{lem811}). Evaluating
the last integral via Lemmas \ref{lem89} and \ref{lem812} and inserting
the result into \eqref{eI21}, we conclude that 
\begin{align}
I_{2,\,1} & \,=\,o_{\epsilon}(1)\,\alpha_{\epsilon}+[\,1+o_{\epsilon}(1)\,]\,\alpha_{\epsilon}\,\frac{\mu^{1/2}\,(-\mathbb{L}\mathbb{H}^{-1}\bm{v})\cdot\bm{e}_{1}}{2\pi\sqrt{\textrm{det\,}(\widetilde{\mathbb{H}}+\mu\widetilde{\bm{v}}\otimes\widetilde{\bm{v}})}}\;.\label{ei211}
\end{align}

Next, we consider $I_{2,\,2}$. By Lemma \ref{lem811}, we have $\partial_{+}\mathcal{B}_{\epsilon}^{\boldsymbol{\sigma}}\setminus\partial_{+}^{1,\,a}\mathcal{B}_{\epsilon}^{\boldsymbol{\sigma}}\subset\partial_{+}^{2,\,a}\mathcal{B}_{\epsilon}^{\boldsymbol{\sigma}};$
hence,
\begin{equation}
|\,I_{2,\,2}\,|\,\le\,\frac{C}{Z_{\epsilon}}\,\int_{\partial_{+}^{1,\,a}\mathcal{B}_{\epsilon}^{\boldsymbol{\sigma}}}e^{-U(\boldsymbol{x})/\epsilon}\,\sigma(d\bm{x})\,\le\,\frac{C}{Z_{\epsilon}}\,e^{-H/\epsilon}\,e^{-cJ^{2}\delta^{2}/\epsilon}\,\sigma(\partial_{+}\mathcal{B}_{\epsilon}^{\boldsymbol{\sigma}})\;,\label{eI22}
\end{equation}
where we applied trivial bounds\footnote{Since $\partial_{+}\mathcal{B}_{\epsilon}^{\boldsymbol{\sigma}}\subset\mathcal{K}$
where $\mathcal{K}$ is defined in \eqref{emck} we can bound $\boldsymbol{\ell}$
by the $L^{\infty}(\mathcal{K})$ norm of $\boldsymbol{\ell}.$ This
argument will be used repeatedly in the remainder of the article without
further mention.} for $|1-p_{\epsilon}^{\bm{\sigma}}(\bm{x})|$ and $\boldsymbol{\ell}$
in the first inequality, while we used the condition $U(\boldsymbol{x})\ge H+aJ^{2}\delta^{2}$
for $\boldsymbol{x}\in\partial_{+}^{2,\,a}\mathcal{B}_{\epsilon}$
in the second one. Since $\sigma(\partial_{+}\mathcal{B}_{\epsilon})=O(\delta^{d-1})$,
we get 
\begin{equation}
|\,I_{2,\,2}\,|\,\le\,\frac{C\,\delta^{d-1}}{Z_{\epsilon}}\,\epsilon^{cJ^{2}/2}=o_{\epsilon}(1)\,\alpha_{\epsilon}\label{ei22}
\end{equation}
for sufficiently large $J$. Hence, $I_{2,\,2}$ is negligible. By
combining \eqref{edecI2}, \eqref{ei211}, and \eqref{ei22}, we get
\begin{equation}
I_{2}\,=\,o_{\epsilon}(1)\,\alpha_{\epsilon}\,+[\,1+o_{\epsilon}(1)\,]\,\alpha_{\epsilon}\,\frac{\mu^{1/2}\,(-\mathbb{L}\mathbb{H}^{-1}\bm{v})\cdot\bm{e}_{1}}{2\pi\sqrt{\textrm{det\,}(\,\widetilde{\mathbb{H}}+\mu\widetilde{\bm{v}}\otimes\widetilde{\bm{v}}\,)}}\;.\label{eI2}
\end{equation}
By \eqref{eI1_2} and \eqref{eI2}, we obtain
\begin{equation}
I_{1}-I_{2}\,=\,[\,1+o_{\epsilon}(1)\,]\,\alpha_{\epsilon}\,\frac{\mu^{1/2}\,(\bm{v}+\mathbb{L}\mathbb{H}^{-1}\bm{v})\cdot\bm{e}_{1}}{2\pi\sqrt{\textrm{det\,}(\,\widetilde{\mathbb{H}}+\mu\widetilde{\bm{v}}\otimes\widetilde{\bm{v}}\,)}}\;.\label{eI1-I2}
\end{equation}
Since $\mathbb{H}\mathbb{L}=-\mathbb{L}^{\dagger}\mathbb{H}$ by the
skew-symmetry of $\mathbb{HL}$, we have $\mathbb{L}\mathbb{H}^{-1}=-\mathbb{H}^{-1}\mathbb{L}^{\dagger}$.
Hence, 
\begin{multline}
(\bm{v}+\mathbb{L}\mathbb{H}^{-1}\bm{v})\cdot\bm{e}_{1}\,=\,(\mathbb{I}-\mathbb{H}^{-1}\mathbb{L}^{\dagger})\bm{v}\cdot\bm{e}_{1}\,=\,\mathbb{H}^{-1}(\mathbb{H}-\mathbb{L}^{\dagger})\bm{v}\cdot\bm{e}_{1}\\
\,=\,-\mu\,\mathbb{H}^{-1}\bm{v}\cdot\bm{e}_{1}\,=\,\frac{\mu}{\lambda_{1}}\,\bm{v}\cdot\bm{e}_{1}\,=\,\frac{\mu\,v_{1}}{\lambda_{1}}\label{e836}
\end{multline}
since $-\mu$ is an eigenvalue of $\mathbb{H}-\mathbb{L}^{\dagger}$
associated with the eigenvector $\bm{v}$ and $\mathbb{H}^{-1}=\text{diag}(-1/\lambda_{1},1/\lambda_{2},\cdots,1/\lambda_{d})$.
Inserting this computation and Lemma \ref{lem810} into \eqref{eI1-I2},
we get
\[
I_{1}-I_{2}\,=\,[\,1+o_{\epsilon}(1)\,]\,\alpha_{\epsilon}\,\frac{\mu}{2\pi\sqrt{\prod_{k=1}^{d}\lambda_{k}}}\,=\,[1+o_{\epsilon}(1)]\,\alpha_{\epsilon}\,\omega^{\boldsymbol{\sigma}}\;.
\]
This completes the proof. 
\end{proof}

\subsection{\label{sec86}Proof of Lemmas \ref{lem811} and \ref{lem812}}

\begin{proof}[Proof of Lemma \ref{lem811}]
By Lemma \ref{lem82}, we have $\lambda_{1}\sum_{k=2}^{d}v_{k}^{2}/\lambda_{k}<v_{1}^{2}$.
Thus, there exists $\varepsilon_{0}\in(0,\,v_{1})$ such that 
\begin{equation}
(\lambda_{1}+\varepsilon_{0})\,\sum_{k=2}^{d}\frac{v_{k}^{2}}{\lambda_{k}}\,<\,(v_{1}-\varepsilon_{0})^{2}\;.\label{evare0}
\end{equation}
Let $a_{0}=\varepsilon_{0}\min\{1,\,\lambda_{1}^{-1/2},\,\lambda_{1}^{-1}\}$,
and we claim that this constant $a_{0}$ satisfies the requirement
of the lemma. 

Fix $a\in(0,\,a_{0})$, $\boldsymbol{x}\in\partial_{+}\mathcal{C}_{\epsilon}$
and suppose, on the other hand, that 
\begin{equation}
\bm{x}\cdot\bm{v}\,<\,aJ\delta\,\le\,\varepsilon_{0}\frac{J\delta}{\lambda_{1}^{1/2}}\;\;\;\text{and\;\;\;}U(\bm{x})-H\,<\,aJ^{2}\delta^{2}\,\le\,\varepsilon_{0}\frac{J^{2}\delta^{2}}{\lambda_{1}}\;.\label{ecla}
\end{equation}
Since $U(\boldsymbol{x})-H=\frac{1}{2}\bm{x}\cdot\mathbb{H}\bm{x}+O(\delta^{3})$
by the Taylor expansion, the latter condition implies that $\bm{x}\cdot\mathbb{H}\bm{x}<\varepsilon_{0}\frac{J^{2}\delta^{2}}{\lambda_{1}}$
for all sufficiently small $\epsilon>0$. 

Write $\boldsymbol{x}\in\partial_{+}\mathcal{C}_{\epsilon}$ as $\bm{x}=\frac{J\delta}{\lambda_{1}^{1/2}}\,\big(\,\bm{e}_{1}+\sum_{k=2}^{d}x_{k}\bm{e}_{k}\,\big)$
such that we can rewrite the two conditions of \eqref{ecla} respectively
as 
\[
0\,<\,\begin{aligned}v_{1}-\varepsilon_{0} & \,<\,-\sum_{k=2}^{d}v_{k}x_{k}\;\;\;\;\text{and\;\;\;\;}\sum_{k=2}^{d}\lambda_{k}x_{k}^{2}\,<\,\lambda_{1}+\varepsilon_{0}\;.\end{aligned}
\]
By these two inequalities and \eqref{evare0}, we have 
\[
\sum_{j=2}^{d}\lambda_{j}x_{j}^{2}\,\sum_{k=2}^{d}\frac{v_{k}^{2}}{\lambda_{k}}\,<\,(\lambda_{1}+\varepsilon_{0})\,\sum_{k=2}^{d}\frac{v_{k}^{2}}{\lambda_{k}}\,<\,(v_{1}-\varepsilon_{0})^{2}\,<\,\Big(\,\sum_{k=2}^{d}x_{k}v_{k}\,\Big)^{2}\ ,
\]
which contradicts the Cauchy--Schwarz inequality; hence, the claim
is proven. 
\end{proof}
\begin{proof}[Proof of Lemma \ref{lem812}]
Write $\zeta=\zeta(\epsilon)=\sqrt{\log\frac{1}{\epsilon}}$ and
let $\mathcal{Q_{\epsilon}}=\Pi_{\epsilon}(\partial_{+}\mathcal{B}_{\epsilon})$.
Then, by the change of variables $\boldsymbol{z}=\epsilon^{-1/2}\boldsymbol{y}$,
we can write the integral in the statement of the lemma as 
\[
\epsilon^{(d-1)/2}\,\int_{\epsilon^{-1/2}\mathcal{Q_{\epsilon}}\cap\{\bm{z}\in\mathbb{R}^{d-1}:\bm{z}\cdot\boldsymbol{u}_{1}\geq-c\zeta\}}\,\frac{\bm{z}\cdot\boldsymbol{u}_{2}+\zeta}{\bm{z}\cdot\boldsymbol{u}_{1}+\zeta}\,e^{-(1/2)\bm{z}\cdot\mathbb{D}\bm{z}}d\bm{z}\;.
\]
Fix $0<\alpha<1$. Then, since $\zeta\rightarrow\infty$ as $\epsilon\rightarrow0$,
by Lemma \ref{lem89}, 
\[
\mathcal{D}_{r\zeta^{\alpha}}^{(d-1)}(\boldsymbol{0})\,\subset\,\epsilon^{-1/2}\mathcal{Q}_{\epsilon}\cap\{\bm{z}\in\mathbb{R}^{d-1}:\bm{z}\cdot\boldsymbol{u}_{1}\ge-c\zeta\}
\]
for all sufficiently small $\epsilon>0$. Now we decompose the integral
into 
\begin{equation}
\Big[\,\int_{\mathcal{D}_{r\zeta^{\alpha}}^{(d-1)}(\boldsymbol{0})}+\int_{\{\epsilon^{-1/2}\mathcal{Q}_{\epsilon}\setminus\mathcal{D}_{r\zeta^{\alpha}}^{(d-1)}(\boldsymbol{0})\}\cap\{\bm{z}\in\mathbb{R}^{d-1}:\bm{z}\cdot\boldsymbol{u}_{1}\ge-c\zeta\}}\,\Big]\,\frac{\bm{z}\cdot\boldsymbol{u}_{2}+\zeta}{\bm{z}\cdot\boldsymbol{u}_{1}+\zeta}\,e^{-(1/2)\bm{z}\cdot\mathbb{D}\bm{z}}d\bm{z}\;.\label{e932}
\end{equation}

Let us consider the first integral. Note that 
\[
\sup_{\boldsymbol{z}\in\mathcal{D}_{r\zeta^{\alpha}}^{(d-1)}(\boldsymbol{0})}\Big|\,\frac{\bm{z}\cdot\boldsymbol{u}_{2}+\zeta}{\bm{z}\cdot\boldsymbol{u}_{1}+\zeta}-1\,\Big|\,=\,o_{\epsilon}(1)\;.
\]
Thus, the first integral is 
\begin{equation}
[\,1+o_{\epsilon}(1)\,]\,\int_{\mathcal{D}_{r\zeta^{\alpha}}^{(d-1)}(\boldsymbol{0})}\,e^{-(1/2)\bm{z}\cdot\mathbb{D}\bm{z}}d\bm{z}\,=\,[\,1+o_{\epsilon}(1)\,]\,\frac{(2\pi)^{(d-1)/2}}{\sqrt{\det(\mathbb{D})}}\;,\label{e933}
\end{equation}
since $\mathcal{D}_{r\zeta^{\alpha}}^{(d-1)}(\boldsymbol{0})\uparrow\mathbb{R}^{d-1}$
as $\epsilon\rightarrow0$. 

Now, we focus on the second integral. Since $\epsilon^{-1/2}\mathcal{Q}_{\epsilon}\subset\mathcal{D}_{R\zeta}^{(d-1)}(\boldsymbol{0})$
by Lemma \ref{lem89}, and since $\bm{z}\cdot\boldsymbol{u}_{1}\geq-c\zeta$
for $c\in(0,\,1)$ by the statement of the lemma, there exists $C>0$
such that 
\[
\sup_{\bm{z}\in\epsilon^{-1/2}\mathcal{Q}_{\epsilon}}\,\Big|\,\frac{\bm{z}\cdot\boldsymbol{u}_{2}+\zeta}{\bm{z}\cdot\boldsymbol{u}_{1}+\zeta}\,\Big|\,\le\,C\;.
\]
Hence, the absolute value of the second integral in \eqref{e932}
is bounded from above by 
\begin{equation}
C\,\int_{\mathcal{D}_{R\zeta}^{(d-1)}(\boldsymbol{0})\,\setminus\,\mathcal{D}_{r\zeta^{\alpha}}^{(d-1)}(\boldsymbol{0})}\,e^{-(1/2)\bm{z}\cdot\mathbb{D}\bm{z}}\,d\bm{z}\,=\,o_{\epsilon}(1)\;.\label{e934}
\end{equation}
By combining \eqref{e932}, \eqref{e933}, and \eqref{e934}, we complete
the proof. 
\end{proof}

\section{\label{sec9}Analysis of Equilibrium Potential }

In this section, we establish a bound on the equilibrium potential
$h_{\epsilon}$ and $h_{\epsilon}^{*}$ in Proposition \ref{p101}.
On the basis of this bound, we prove Proposition \ref{p73} in Section
\ref{sec94}. Further, we remark that this bound plays an important
role in the proof of Theorem \ref{t74} (cf. Section \ref{sec104}). 

For two disjoint non-empty sets $\mathcal{A},\,\mathcal{B}\subset\mathbb{R}^{d}$,
let $\Gamma_{\mathcal{A},\mathcal{\,B}}$ be a set of all $C^{1}$-paths
$\boldsymbol{\gamma}:[0,1]\rightarrow\mathbb{R}^{d}$ such that $\boldsymbol{\gamma}(0)\in\mathcal{A}$
and $\boldsymbol{\gamma}(1)\in\mathcal{B}$. Then, let $\mathfrak{H}_{\mathcal{A},\mathcal{\,B}}$
denote the height of the saddle points between $\mathcal{A}$ and
$\mathcal{B}$: 
\[
\mathfrak{H}_{\mathcal{A},\,\mathcal{B}}\,\coloneqq\,\inf_{\gamma\in\Gamma_{\mathcal{A},\,\mathcal{B}}}\,\sup_{t\in[0,\,1]}\,U(\,\boldsymbol{\gamma}(t)\,)\;.
\]

\subsection{Estimates of equilibrium potentials $h_{\epsilon}$ and $h_{\epsilon}^{*}$}

In this subsection, we prove the following proposition regarding the
so-called leveling property of the equilibrium potential. 
\begin{prop}
\label{p101}We can find a constant $C>0$ satisfying the following
bounds. 
\begin{enumerate}
\item For all $\bm{y}\in\mathcal{H}_{0}$, the following holds: 
\[
h_{\epsilon}(\bm{y}),\,h_{\epsilon}^{*}(\bm{y})\,\ge\,1-C\,\epsilon^{-d}\,\exp\frac{\mathfrak{H}_{\{\bm{y}\},\,\mathcal{D}_{\epsilon}(\bm{m}_{0})}-H}{\epsilon}\;.
\]
\item For all $\bm{y}\in\mathcal{H}_{1}$, the following holds:
\[
h_{\epsilon}(\bm{y}),\,h_{\epsilon}^{*}(\bm{y})\,\le\,C\,\epsilon^{-d}\,\exp\frac{U(\boldsymbol{y})-H}{\epsilon}\;.
\]
\end{enumerate}
\end{prop}

The proof of Proposition \ref{p101} relies on the following two bounds
on the capacity. 
\begin{lem}
\label{lem102}There exists $C>0$ such that for all $\bm{y}\in\mathcal{W}_{0}$
and $\bm{m}\in\mathcal{M}_{0}$, 
\[
\textup{cap}_{\epsilon}(\mathcal{D}_{\epsilon}(\bm{y}),\,\mathcal{D}_{\epsilon}(\bm{m}))\,\geq\,C\,\epsilon^{d}\,Z_{\epsilon}^{-1}\,e^{-\mathfrak{H}_{\{\bm{y}\},\mathcal{D}_{\epsilon}(\bm{m})}/\epsilon}\;.
\]
\end{lem}

\begin{lem}
\label{lem103}There exists $C>0$ such that for all $\bm{y}\in\mathcal{H}_{0}$,
\[
\textup{cap}_{\epsilon}(\mathcal{D}_{\epsilon}(\bm{y}),\,\mathcal{U}_{\epsilon})\,\leq\,C\,Z_{\epsilon}^{-1}\,e^{-H/\epsilon}\;.
\]
\end{lem}

We prove Lemmas \ref{lem102} and \ref{lem103} in Sections \ref{sec92}
and \ref{sec93}, respectively. Now, we prove Proposition \ref{p101}
\begin{proof}[Proof of Proposition \ref{p101}]
Since the proofs for $h_{\epsilon}$ and $h_{\epsilon}^{*}$ are
identical, we consider only $h_{\epsilon}.$ In \cite[Proposition 7.9]{LMS},
it has been shown that there exists $C>0$ such that 
\begin{equation}
h_{\mathcal{A},\,\mathcal{B}}(\bm{x})\,\leq\,C\,\frac{\textup{cap}_{\epsilon}(\mathcal{D}_{\epsilon}(\bm{x}),\,\mathcal{A})}{\textup{cap}_{\epsilon}(\mathcal{D}_{\epsilon}(\bm{x}),\,\mathcal{B})}\;,\label{ebdh}
\end{equation}
provided that $\mathcal{A}$ and $\mathcal{B}$ are disjoint domains
of sufficiently smooth bounds. For part (1), we can use this bound
to get 
\[
1-h_{\epsilon}(\boldsymbol{y})\,=\,h_{\mathcal{U}_{\epsilon},\,\mathcal{D}_{\epsilon}(\boldsymbol{m}_{0})}(\boldsymbol{y})\,\le\,C\,\frac{\textup{cap}_{\epsilon}(\mathcal{D}_{\epsilon}(\bm{y}),\,\mathcal{U}_{\epsilon})}{\textup{cap}_{\epsilon}(\mathcal{D}_{\epsilon}(\bm{y}),\,\mathcal{D}_{\epsilon}(\boldsymbol{m}_{0}))}\;.
\]
Now, by applying Lemmas \ref{lem102} and \ref{lem103}, we complete
the proof of part (1). 

For part (2), we fix $\boldsymbol{y}\in\mathcal{H}_{1}$. Then, again
by \eqref{ebdh}, 
\[
h_{\epsilon}(\boldsymbol{y})\,=\,h_{\mathcal{D}_{\epsilon}(\boldsymbol{m}_{0}),\,\mathcal{U}_{\epsilon}}(\boldsymbol{y})\,\le\,C\,\frac{\textup{cap}_{\epsilon}(\mathcal{D}_{\epsilon}(\bm{y}),\,\mathcal{D}_{\epsilon}(\boldsymbol{m}_{0}))}{\textup{cap}_{\epsilon}(\mathcal{D}_{\epsilon}(\bm{y}),\,\mathcal{U}_{\epsilon})}\;.
\]
By the same logic with the proofs of Lemmas \ref{lem102} and \ref{lem103},
we get 
\[
\textup{cap}_{\epsilon}(\mathcal{D}_{\epsilon}(\bm{y}),\,\mathcal{D}_{\epsilon}(\boldsymbol{m}_{0}))\,\le\,\frac{Ce^{-H/\epsilon}e^{-H/\epsilon}}{Z_{\epsilon}}\;\;\;\text{and}\;\;\;\textup{cap}_{\epsilon}(\mathcal{D}_{\epsilon}(\bm{y}),\,\mathcal{U}_{\epsilon})\,\ge\,\frac{C\epsilon^{d}}{Z_{\epsilon}}\,e^{-\mathfrak{H}_{\{\bm{y}\},\,\mathcal{U}_{\epsilon}}/\epsilon}\;.
\]
Since $\mathcal{U}_{\epsilon}$ contains all the local minima of $\mathcal{M}_{1}$
and $\mathcal{H}_{1}$ is a subset of the domain of attraction of
$\mathcal{M}_{1}$, we have $\mathfrak{H}_{\{\bm{y}\},\,\mathcal{U}_{\epsilon}}=U(\boldsymbol{y})$
and the proof is completed. 
\end{proof}

\subsection{\label{sec92}Proof of Lemma \ref{lem102}}

For the lower bound case, the proof is a consequence of the existing
estimate for the reversible case. Let $\textrm{cap}_{\epsilon}^{s}(\cdot,\,\cdot)$
denote the capacity with respect to the reversible process $\boldsymbol{z}_{\epsilon}(\cdot)$
given in \eqref{e_SDEy}, whose generator is $(1/2)(\mathscr{L}_{\epsilon}+\mathscr{L}_{\epsilon}^{*})$.
Then, it is well known that\textbf{ }(cf. \cite[Lemma 2.5]{GL}) for
any two disjoint non-empty domains $\mathcal{A},\,\mathcal{B}\subset\mathbb{R}^{d}$
with smooth boundaries, we have the following equation: 
\begin{equation}
\textrm{cap}_{\epsilon}(\mathcal{A},\,\mathcal{B})\,\ge\,\textrm{cap}_{\epsilon}^{s}(\mathcal{A},\,\mathcal{B})\;.\label{ecapcom}
\end{equation}
Therefore, it suffices to show the inequality for $\textrm{cap}_{\epsilon}^{s}(\mathcal{D}_{\epsilon}(\bm{y}),\,\mathcal{D}_{\epsilon}(\bm{m}))$,
instead. The lower bound for this capacity can be obtained by optimizing
the integration on the tube connecting $\mathcal{D}_{\epsilon}(\bm{y})$
and $\mathcal{D}_{\epsilon}(\bm{m})$. This is rigorously achieved
by a parametrization of this tube. When we parametrize the tube successfully,
we can use the idea of \cite[Proposition 4.7]{BEGK1} to complete
the proof. 

Let $\bm{\omega}:[0,L]\to\mathbb{R}^{d}$ be a smooth path such that
$|\dot{\bm{\omega}}(t)|=1$ for all $t\in[0,\,L]$. For $r>0$, define
$A_{r}(0),\,A_{r}(L)$ by
\begin{align*}
A_{r}(0) & \,=\,\{\,\bm{x}\in\mathbb{R}^{d}:\bm{x}\cdot\dot{\boldsymbol{\omega}}(0)<0,\,|\bm{x}-\boldsymbol{\omega}(0)|<r\,\}\\
A_{r}(L) & \,=\,\{\,\bm{x}\in\mathbb{R}^{d}:\bm{x}\cdot\dot{\boldsymbol{\omega}}(L)>0,\,|\bm{x}-\boldsymbol{\omega}(L)|<r\,\}
\end{align*}
and define the\textit{ tubular neighborhood of $\boldsymbol{\omega}$
of radius $r$} by
\[
\boldsymbol{\omega}_{r}\,=\,\{\bm{x}\in\mathbb{R}^{d}:|\bm{x}-\boldsymbol{\omega}(t)|<r\;\,\text{for some }t\in[0,L]\}\setminus(\,A_{r}(0)\cup A_{r}(L)\,).
\]
For $\rho>0$, let $\mathcal{D}_{\rho}^{(d-1)}$ be a $(d-1)$-dimensional
sphere of radius $\rho$ centered at the origin. 
\begin{lem}
\label{lem tb nbhd}There exists $r_{0}>0$ such that $[0,L]\times\mathcal{D}_{r_{0}}^{(d-1)}$
is diffeomorphic to $\boldsymbol{\omega}_{r_{0}}$. Furthermore, we
can find a diffeomorphism $\varphi:[0,L]\times\mathcal{D}_{r_{0}}^{(d-1)}\to\boldsymbol{\omega}_{r_{0}}$
of the form
\begin{equation}
\varphi(t,\,\bm{z})\,=\,\boldsymbol{\omega}(t)+\mathbb{A}(t)\bm{z}\label{vaph}
\end{equation}
for some smooth $d\times(d-1)$ matrix-valued function $\mathbb{A}(\cdot)$
of rank $d-1$, and it satisfies 
\begin{equation}
\Big|\,\det\frac{\partial\varphi}{\partial(t,\,\bm{z})}\,\Big|\,\ge\,\frac{1}{2}\;\;\;\;\text{on}\;[0,\,L]\times\mathcal{D}_{r_{0}}^{(d-1)}\;.\label{bdvaph}
\end{equation}
\end{lem}

\begin{proof}
The proof needs to recall several notions and results from differential
geometry. We refer to \cite{JMLee} for a reference.\textcolor{red}{{}
}We regard $\boldsymbol{\omega}=\boldsymbol{\omega}(\,[\,0,\,L\,]\,)$
as a one-dimensional compact manifold. Let $N\boldsymbol{\omega}\subset\mathbb{R}^{d}\times\mathbb{R}^{d}$
denote the normal bundle of $\boldsymbol{\omega}$. By the tubular
neighborhood theorem (cf. \cite[Theorem 6.24]{JMLee}), there exists
$r_{0}>0$ such that $\boldsymbol{\omega}_{r_{0}}$ is diffeomorphic
to $N\boldsymbol{\omega}_{r_{0}}=\{\,(\bm{p},\,\bm{v})\in N\boldsymbol{\omega}:|\bm{v}|<r_{0}\,\}$.
The diffeomorphism $E:N\boldsymbol{\omega}_{r_{0}}\to\boldsymbol{\omega}_{r_{0}}$
is given by $E(\bm{p},\,\bm{v})=\bm{p}+\bm{v}$. Since $\boldsymbol{\omega}$
is contractible, the vector bundle of $\boldsymbol{\omega}$ is trivial;
thus, $N\boldsymbol{\omega}$ is diffeomorphic to $\boldsymbol{\omega}\times\mathbb{R}^{d-1}$.
Let $\phi:\boldsymbol{\omega}\times\mathbb{R}^{d-1}\to N\boldsymbol{\omega}$
denote the corresponding diffeomorphism. Since this diffeomorphism
preserves the vector space structure, the function $\phi(\bm{p},\,\bm{z})$
is linear in $\bm{z}$ and satisfies $|\pi_{2}(\phi(\bm{p},\,\bm{z}))|=|\bm{z}|$
where $\pi_{2}:\mathbb{R}^{d}\times\mathbb{R}^{d}\rightarrow\mathbb{R}^{d}$
is the projection function for the second coordinate.

Since $\boldsymbol{\omega}\times\mathbb{R}^{d-1}$ is a trivial bundle
of rank $d-1$, there are $d-1$ smooth sections $\sigma_{j}:\boldsymbol{\omega}\to\mathbb{R}^{d-1}$
which are linearly independent. By the Gram--Schmidt operation, we
may assume that they are pointwise orthonormal, i.e., $\sigma_{i}(\bm{p})\cdot\sigma_{j}(\bm{p)}=\delta_{i,\,j}$
for all $i,\,j$ and $\bm{p}\in\boldsymbol{\omega}$. Define a $d\times(d-1)$
matrix $\mathbb{B}(\bm{p})=[\,\mathbb{B}_{1}(\bm{p}),\,\dots,\,\mathbb{B}_{d-1}(\bm{p})\,]$
by $\mathbb{B}_{i}(\bm{p})\,=\,\pi_{2}(\phi(\bm{p},\,\sigma_{i}(\bm{p})))$
for $j=1,\,\dots,\,d-1$. By the smoothness of $\phi$ and $\sigma_{j}$,
we can observe that all the elements of $\mathbb{B}(\cdot)$ are smooth.
Then, the diffeomorphism $\varphi:[0,L]\times\mathcal{D}_{r_{0}}^{(d-1)}\to\boldsymbol{\omega}_{r_{0}}$
can be written as 
\[
\varphi(t,\,\bm{z})\,=\,\phi(\boldsymbol{\omega}(t),\,\bm{z})\,=\,\boldsymbol{\omega}(t)+\mathbb{B}(\boldsymbol{\omega}(t))\bm{z}\;.
\]
We can now take $\mathbb{A}=\mathbb{B}\circ\omega$ to get \eqref{vaph}.
Now we consider \eqref{bdvaph}. We can write
\[
\frac{\partial\varphi}{\partial(t,\,\bm{z})}(t,\,\boldsymbol{0})\,=\,[\,\dot{\bm{\omega}}(t),\,\mathbb{A}(t)\,]\;.
\]
Since all the column vectors in the matrix on the right-hand sides
are normal and orthogonal to each other, we have $\big|\det\frac{\partial\varphi}{\partial(t,\,\bm{z})}(t,\,\bm{0})\big|=1$.
Hence, by taking $r_{0}$ to be sufficiently small, we get \eqref{bdvaph}. 
\end{proof}
\begin{prop}
\label{prop tb nbhd}Let $\bm{\omega}:[0,\,L]\rightarrow\mathbb{R}^{d}$
be a $C^{1}$-path connecting $\bm{y}$ and $\bm{m}$ such that $U(\bm{\omega}(t))\leq M$
and $|\dot{\bm{\omega}}(t)|=1$ for all $t$. Moreover, let $f$ be
a smooth function such that $f\equiv1$ on $\mathcal{D}_{\epsilon}(\bm{y})$
and $f\equiv0$ on $\mathcal{D}_{\epsilon}(\boldsymbol{m})$\textbf{.}
Then, there exists a constant $C>0$ such that 
\[
\epsilon\,\int_{\boldsymbol{\omega}_{r_{0}}}\,|\,\nabla f\,|^{2}\,d\mu_{\epsilon}\,\ge\,C\,L^{-1}\,\epsilon^{d}\,Z_{\epsilon}^{-1}\,e^{-M/\epsilon}\;,
\]
where $r_{0}$ is the constant obtained in Lemma \ref{lem tb nbhd}
for the path $\boldsymbol{\omega}$. 
\end{prop}

\begin{proof}
By Lemma \ref{lem tb nbhd}, we have 
\[
\begin{aligned}\epsilon\,\int_{\bm{\omega}_{r_{0}}}\,|\,\nabla f\,|^{2}d\mu_{\epsilon} & \,\ge\,\frac{\epsilon}{2Z_{\epsilon}}\,\int_{\mathcal{D}_{\epsilon}^{(d-1)}}\int_{0}^{L}\,|\,\nabla f(\bm{\omega}(t)+\mathbb{A}(t)\bm{z})\,|^{2}\,e^{-U(\bm{\omega}(t)+\mathbb{A}(t)\bm{z})/\epsilon}\,dt\,d\bm{z}\end{aligned}
\]
for $\epsilon\in(0,r_{0})$, where the factor of $2$ appears because
\eqref{bdvaph} is used for bounding the Jacobian of the change of
variables from below. For $(t,\bm{z})\in[0,\,L]\times\mathcal{D}_{\epsilon}^{(d-1)}$,
we have 
\[
\begin{aligned}\frac{d}{dt}f(\bm{\omega}(t)+\mathbb{A}(t)\bm{z}) & \,=\,\nabla f(\bm{\omega}(t)+\mathbb{A}(t)\bm{z})\cdot(\dot{\bm{\omega}}(t)+\dot{\mathbb{A}}(t)\bm{z})\,\le\,2|\,\nabla f(\bm{\omega}(t)+\mathbb{A}(t)\bm{z})\,|\;,\end{aligned}
\]
where the last inequality holds for sufficiently small $\epsilon$
since $|\dot{\omega}(t)|=1$ and $|\boldsymbol{z}|\le\epsilon$. Summing
up, we can write 
\begin{equation}
\epsilon\,\int_{\bm{\omega}_{r_{0}}}\,|\,\nabla f\,|^{2}\,d\mu_{\epsilon}\,\ge\,\frac{\epsilon}{4Z_{\epsilon}}\int_{\mathcal{D}_{\epsilon}^{(d-1)}}\int_{0}^{L}\,\Big|\,\frac{d}{dt}f(\bm{\omega}(t)+\mathbb{A}(t)\bm{z})\,\Big|^{2}\,e^{-U(\bm{w}(t)+\mathbb{A}(t)\bm{z})/\epsilon}\,dt\,d\bm{z}\;.\label{ebb1}
\end{equation}
Now, we can apply the idea of \cite[Proposition 4.7]{BEGK1}. Indeed,
we can fix $\boldsymbol{z}\in\mathcal{D}_{\epsilon}^{(d-1)}$ and
write $f_{\bm{z}}(t)=f(\bm{\omega}(t)+\mathbb{A}(t)\bm{z}).$ Then,
we can obtain the minimizer of the integral $\int_{0}^{L}\left|\frac{d}{dt}f_{\bm{z}}(t)\right|^{2}e^{-U(\bm{w}(t)+\mathbb{A}(t)\bm{z})/\epsilon}dt$
explicitly as 
\[
f_{\bm{z}}(t)\,=\,\frac{\int_{t}^{L}\,e^{U(\boldsymbol{\omega}(s)+\mathbb{A}(s)\bm{z})/\epsilon}\,ds}{\int_{0}^{L}\,e^{U(\boldsymbol{\omega}(s)+\mathbb{A}(s)\bm{z})/\epsilon}\,ds}\;.
\]
Inserting this solution into \eqref{ebb1} gives
\[
\epsilon\,\int_{\bm{\omega}_{r_{0}}}\,|\,\nabla f\,|^{2}d\mu_{\epsilon}\,\ge\,\frac{\epsilon}{4Z_{\epsilon}}\,\int_{\mathcal{D}_{\epsilon}^{(d-1)}}\,\Big[\,\int_{0}^{L}e^{U(\omega(t)+\mathbb{A}(t)\bm{z})/\epsilon}dt\,\Big]^{-1}\,d\bm{z}\;.
\]
Since $|\boldsymbol{z}|\le\epsilon$, we have $U(\omega(t)+\mathbb{A}(t)\bm{z})\le M+C\,\epsilon$
for some constant $C>0$, and the proof is completed. 
\end{proof}
Now, we are ready to prove Lemma \ref{lem102}.
\begin{proof}[Proof of Lemma \ref{lem102}]
Fix $\bm{y}\in\mathcal{H}_{0}$ and for some $L=L(\boldsymbol{y})$,
let $\boldsymbol{\omega}:[0,\,L]\rightarrow\mathbb{R}^{d}$ be a $C^{1}$-path
connecting $\bm{y}$ to $\mathcal{D}_{\epsilon}(\bm{m})$ such that
$U(\boldsymbol{\omega}(t))\leq\mathfrak{H}_{\{\bm{y}\},\,\mathcal{D}_{\epsilon}(\bm{m})}$
and $|\dot{\boldsymbol{\omega}}(t)|=1$ for all $t\in[0,\,L]$. Since
$\mathcal{H}_{0}$ is bounded, we can find $L_{0}$ such that $L(\boldsymbol{y})<L_{0}$
for all $\boldsymbol{y}\in\mathcal{H}_{0}$. Then, recall the diffeomorphism
$\varphi:[0,\,L]\times\mathcal{D}_{r_{0}}^{(d-1)}\to\bm{\omega}_{r_{0}}$
constructed in Lemma \ref{lem tb nbhd}. Then, 
\[
\text{cap}_{\epsilon}^{s}(\,\mathcal{D}_{\epsilon}(\bm{y}),\,\mathcal{D}_{\epsilon}(\bm{m})\,)\,\ge\,\epsilon\,\int_{\boldsymbol{\omega}_{r_{0}}}\,|\,\nabla h_{\mathcal{D}_{\epsilon}(\bm{y}),\,\mathcal{D}_{\epsilon}(\bm{m})}^{\epsilon,\,s}\,|^{2}\,d\mu_{\epsilon}\ ,
\]
where $h_{\mathcal{D}_{\epsilon}(\bm{y}),\,\mathcal{D}_{\epsilon}(\bm{m})}^{\epsilon,\,s}(\cdot)$
is the equilibrium potential between $\mathcal{D}_{\epsilon}(\bm{y})$
and $\mathcal{D}_{\epsilon}(\bm{m})$ with respect to the reversible
process $\boldsymbol{y}_{\epsilon}(\cdot)$. Hence, by Proposition
\ref{prop tb nbhd} and the fact that we can take $L(\boldsymbol{y})$
to be uniformly bounded by $L_{0}$, the proof is completed. 
\end{proof}

\subsection{\label{sec93}Proof of Lemma \ref{lem103}}

The upper bound cannot be proven by a comparison with reversible dynamics
as in the lower bound case unless the dynamics satisfies the so-called
sector condition, and that is exactly what has been used in \cite{LMS}.
However, the dynamics $\boldsymbol{x}_{\epsilon}(\cdot)$ does not
necessarily satisfy the sector condition; hence, we must develop a
new argument. We believe that our argument presented below is sufficiently
robust to treat a wide class of models. 
\begin{proof}[Proof of Lemma \ref{lem103}]
For each set $\mathcal{A}\subset\mathbb{R}^{d}$ and $r>0$, define
\begin{equation}
\mathcal{A}^{[r]}\,=\,\{\,\boldsymbol{x}\in\mathbb{R}^{d}:|\boldsymbol{x}-\boldsymbol{y}|\le r\text{ for some \ensuremath{\boldsymbol{y}\in}\ensuremath{\mathcal{A}\,\}}\;.}\label{eAr}
\end{equation}
Suppose that $\epsilon$ is sufficiently small such that $\mathcal{\mathcal{H}}_{0}^{[2\epsilon]}$
is disjoint from $\mathcal{U}_{\epsilon}$ and $\mathcal{\mathcal{H}}_{0}^{[2\epsilon]}\subset\mathcal{K}$
(cf. \eqref{emck}). Take a smooth function $q_{\epsilon}:\mathbb{R}^{d}\rightarrow\mathbb{R}$
such that, for some constant $C>0$, 
\begin{equation}
q_{\epsilon}\,\equiv\,1\text{ on }\mathcal{\mathcal{H}}_{0}^{[\epsilon]}\;\;,\;\;\;q_{\epsilon}\,\equiv\,0\text{ }\text{on }\mathbb{R}^{d}\setminus\mathcal{\mathcal{H}}_{0}^{[2\epsilon]}\;\;,\text{ and}\;\;\;|\nabla q_{\epsilon}|\,\le\,\frac{C}{\epsilon}\mathbf{1}_{\mathcal{H}_{0}^{[2\epsilon]}\setminus\mathcal{H}_{0}^{[\epsilon]}}\;.\label{econfe}
\end{equation}
Since $q_{\epsilon}\in\mathcal{\mathscr{C}}_{\mathcal{D}_{\epsilon}(\boldsymbol{y}),\,\mathcal{U}_{\epsilon}}$
(cf. \eqref{eCab}), we can deduce from Proposition \ref{p62} that
\begin{equation}
\textrm{cap}_{\epsilon}(\mathcal{D}_{\epsilon}(\boldsymbol{y}),\,\mathcal{U}_{\epsilon})\,=\,\epsilon\,\int_{\Omega_{\epsilon}}\,\Big[\,\nabla q_{\epsilon}\cdot\nabla h_{\epsilon}+\frac{1}{\epsilon}q_{\epsilon}\boldsymbol{\ell}\cdot\nabla h_{\epsilon}\,\Big]\,d\mu_{\epsilon}\;.\label{ecDU}
\end{equation}
By the divergence theorem and \eqref{econ_ell2}, the second term
on the right-hand side can be rewritten as 
\begin{equation}
\int_{\partial\Omega_{\epsilon}}h_{\epsilon}\,q_{\epsilon}\,[\,\boldsymbol{\ell}\cdot\bm{n}_{\Omega_{\epsilon}}\,]\,\sigma(d\mu_{\epsilon})\,-\,\int_{\Omega_{\epsilon}}h_{\epsilon}\,[\,\nabla q_{\epsilon}\cdot\boldsymbol{\ell}\,]\,d\mu_{\epsilon}\;.\label{e107}
\end{equation}
Since $h_{\epsilon}=\mathbf{1}_{\partial\mathcal{D}_{\epsilon}(\boldsymbol{y})}$
on $\partial\Omega_{\epsilon}=\partial\,\mathcal{U}_{\epsilon}\cup\partial\mathcal{D}_{\epsilon}(\boldsymbol{y})$,
$q_{\epsilon}\equiv1$ on $\partial\mathcal{D}_{\epsilon}(\boldsymbol{y})$,
and $\bm{n}_{\Omega_{\epsilon}}=-\boldsymbol{n}_{\mathcal{D}_{\epsilon}(\boldsymbol{y})}$,
the first integral of \eqref{e107} becomes 
\begin{equation}
-\int_{\partial\mathcal{D}_{\epsilon}(\boldsymbol{y})}\,[\,\boldsymbol{\ell}\cdot\boldsymbol{n}_{\mathcal{D}_{\epsilon}(\boldsymbol{y})}\,]\,\sigma(d\mu_{\epsilon})\,=\,\int_{\mathcal{D}_{\epsilon}(\boldsymbol{y})}(\,\nabla\cdot\boldsymbol{\ell\,})\,d\mu_{\epsilon}+\int_{\mathcal{D}_{\epsilon}(\boldsymbol{y})}\,[\,\boldsymbol{\ell}\cdot\nabla\mu_{\epsilon}\,](\boldsymbol{x})\,d\boldsymbol{x}\label{e108}
\end{equation}
by the divergence theorem again. Note that the last two integrals
are $0$ by \eqref{econ_ell2} and \eqref{econ_ell1}, respectively.
Hence the first integral of \eqref{e107} vanishes. For the second
integral of \eqref{e107}, by the trivial bound $|h_{\epsilon}|\leq1$
and the last condition of \eqref{econfe}, we have 
\begin{equation}
\Big|\,\int_{\Omega_{\epsilon}}h_{\epsilon}\,[\nabla q_{\epsilon}\cdot\boldsymbol{\ell}]\,d\mu_{\epsilon}\,\Big|\,\le\,\frac{C}{\epsilon Z_{\epsilon}}\,\int_{\mathcal{H}_{0}^{[2\epsilon]}\setminus\mathcal{H}_{0}^{[\epsilon]}}\,e^{-U(\bm{x})/\epsilon}\,d\boldsymbol{x}\,\le\,\frac{C}{Z_{\epsilon}}\,e^{-H/\epsilon}\;,\label{e109}
\end{equation}
where the second inequality follows from the fact that $U(\bm{x})=H+O(\epsilon)$
on $\mathcal{H}_{0}^{[2\epsilon]}\setminus\mathcal{H}_{0}^{[\epsilon]}$
and that $\text{vol}\,(\mathcal{H}_{0}^{[2\epsilon]}\setminus\mathcal{H}_{0}^{[\epsilon]})=O(\epsilon)$.
Summing up, we obtain from \eqref{ecDU} that 
\begin{equation}
\textrm{cap}_{\epsilon}(\mathcal{D}_{\epsilon}(\bm{y}),\,\mathcal{U}_{\epsilon})\,\le\,\epsilon\,\int_{\Omega_{\epsilon}}\,[\nabla q_{\epsilon}\cdot\nabla h^{\epsilon}]\,d\mu_{\epsilon}+\frac{C}{Z_{\epsilon}}\,e^{-H/\epsilon}\;.\label{e1010}
\end{equation}
By the Cauchy--Schwarz inequality and part (2) of Lemma \ref{lem61},
the integral on the right-hand side is bounded from above by the square
root of 
\[
\epsilon\,\int_{\Omega_{\epsilon}}\,|\,\nabla q_{\epsilon}\,|^{2}\,d\mu_{\epsilon}\times\textrm{cap}_{\epsilon}(\mathcal{D}_{\epsilon}(\bm{y}),\,\mathcal{U}_{\epsilon})\;.
\]
By a computation similar to \eqref{e109}, we get 
\[
\epsilon\,\int_{\Omega_{\epsilon}}\,|\nabla q_{\epsilon}|^{2}\,d\mu_{\epsilon}\,\le\,\frac{C}{\epsilon Z_{\epsilon}}\,\int_{\mathcal{H}_{0}^{[2\epsilon]}\setminus\mathcal{H}_{0}^{[\epsilon]}}\,e^{-U(\bm{x})/\epsilon}\,d\boldsymbol{x}\,\le\,\frac{C}{Z_{\epsilon}}\,e^{-H/\epsilon}\;.
\]
Therefore, we can bound the integral on the right-hand side of \eqref{e1010}
by 
\[
\Big[\,\frac{C}{Z_{\epsilon}}\,e^{-H/\epsilon}\,\textrm{cap}_{\epsilon}(\mathcal{D}_{\epsilon}(\bm{y}),\,\mathcal{U}_{\epsilon})\,\Big]^{1/2}\,\le\,\frac{1}{2}\,\Big[\,\frac{C}{Z_{\epsilon}}\,e^{-H/\epsilon}+\textrm{cap}_{\epsilon}(\mathcal{D}_{\epsilon}(\bm{y}),\,\mathcal{U}_{\epsilon})\,\Big]\;.
\]
Inserting this into \eqref{e1010} completes the proof. 
\end{proof}

\subsection{\label{sec94}Proof of Proposition \ref{p73}}

Now, we are ready to prove Proposition \ref{p73}, which is a crucial
step in the proof of the Eyring--Kramers formula. 
\begin{proof}[Proof of Proposition \ref{p73}]
Take $\beta>0$ to be sufficiently small such that there is no critical
point $\boldsymbol{c}$ of $U$ such that $U(\boldsymbol{c})\in[H-\beta,\,H)$.
Then, we can decompose $\mathcal{G=}\{\boldsymbol{x}:U(\boldsymbol{x})<H-\beta\}$
into $\mathcal{G}_{0},\,\mathcal{G}_{1}$, where $\mathcal{G}_{0}\subset\mathcal{H}_{0}$
and $\mathcal{G}_{1}\subset\mathcal{H}_{1}$. Write 
\begin{equation}
\int_{\mathbb{R}^{d}}\,h_{\epsilon}^{*}\,d\mu_{\epsilon}\,=\,\Big[\,\int_{\mathcal{G}_{0}}+\int_{\mathcal{G}_{1}}+\int_{\mathcal{G}^{c}}\,\Big]\,h_{\epsilon}^{*}\,d\mu_{\epsilon}\label{edechd}
\end{equation}
and consider the three integrals separately. First, for $\boldsymbol{y}\in\mathcal{G}_{0}$,
we have $\mathfrak{H}_{\{\bm{y}\},\mathcal{D}_{\epsilon}(\bm{m}_{0})}<H-\beta$;
thus, by part (1) of Proposition \ref{p101}, we have $|\,h_{\epsilon}^{*}(\bm{y})-1\,|\le C\,\epsilon^{-d}\,e^{-\beta/\epsilon}=o_{\epsilon}(1)$.
This bound ensures that 
\begin{equation}
\int_{\mathcal{G}_{0}}\,h_{\epsilon}^{*}\,d\mu_{\epsilon}\,=\,[\,1+o_{\epsilon}(1)\,]\,\mu_{\epsilon}(\mathcal{G}_{0})\,=\,[\,1+o_{\epsilon}(1)\,]\,Z_{\epsilon}^{-1}\,(2\pi\epsilon)^{d/2}\,e^{-h_{0}/\epsilon}\,\nu_{0}\;,\label{ehmu1}
\end{equation}
where the second identity follows from the Laplace asymptotics for
the function $e^{-U/\epsilon}$. 

For the second integral, by part (2) of Proposition \ref{p101}, 
\begin{equation}
\int_{\mathcal{G}_{1}}h_{\epsilon}^{*}\,d\mu_{\epsilon}\,\le\,\frac{C}{Z_{\epsilon}\,\epsilon^{d}}\,\int_{\mathcal{G}_{1}}\,e^{[U(\boldsymbol{x})-H]/\epsilon}\,e^{-U(\boldsymbol{x})/\epsilon}\,d\boldsymbol{x}\,=\,o_{\epsilon}(1)\,\,Z_{\epsilon}^{-1}\,(2\pi\epsilon)^{d/2}\,e^{-h_{0}/\epsilon}\,\nu_{0}\ ,\label{ehmu2}
\end{equation}
where the last line follows from $H>h_{0}$. Finally, for the last
integral, by the bound $|h_{\epsilon}^{*}|\le1$ and \eqref{etight_U},
\begin{equation}
\int_{\mathcal{G}^{c}}h_{\epsilon}^{*}\,d\mu_{\epsilon}\,\le\,\mu_{\epsilon}(\mathcal{G}^{c})\,\le\,Z_{\epsilon}^{-1}\,e^{-(H-\beta)/\epsilon}\,=\,o_{\epsilon}(1)\,Z_{\epsilon}^{-1}\,(2\pi\epsilon)^{d/2}\,e^{-h_{0}/\epsilon}\,\nu_{0}\;.\label{ehmu3}
\end{equation}
By inserting \eqref{ehmu1}, \eqref{ehmu2}, and \eqref{ehmu3} into
\eqref{edechd}, the proof is completed. 
\end{proof}

\section{\label{sec10}Construction of Test Function and Proof of Theorem
\ref{t74}}

In this section, we finally construct the test function $g_{\epsilon}\in\mathscr{C}_{\mathcal{D}_{\epsilon}(\boldsymbol{m}_{0}),\,\mathcal{U}_{\epsilon}}$
satisfying Theorem \ref{t74}.

\subsection{\label{sec101}Construction of $g_{\epsilon}$ and proof of Theorem
\ref{t74}}

\textcolor{black}{Recall $\mathcal{H}_{0}^{\epsilon}$ and $p_{\epsilon}^{\bm{\sigma}}$
from Section \ref{sec81} and \eqref{epeps}, respectively,} and define
$f_{\epsilon}:\mathbb{R}^{d}\rightarrow\mathbb{R}$ as 
\[
f_{\epsilon}(\boldsymbol{x})\,=\,\begin{cases}
p_{\epsilon}^{\boldsymbol{\sigma}}(\boldsymbol{x}) & \boldsymbol{x}\in\mathcal{B}_{\epsilon}^{\boldsymbol{\sigma}}\text{ for some }\boldsymbol{\sigma}\in\Sigma_{0}\;,\\
\mathbf{1}_{\mathcal{H}_{0}^{\epsilon}}(\boldsymbol{x}) & \text{otherwise}\;.
\end{cases}
\]
The function $f_{\epsilon}$ is not continuous on $\mathcal{K}_{\epsilon}$
in general; instead, it is discontinuous along the boundaries $\partial_{\pm}\mathcal{B}_{\epsilon}^{\boldsymbol{\sigma}}$
and $\partial\mathcal{K}_{\epsilon}$.
\begin{rem}
\label{rem101}It can be readily checked that the function $f_{\epsilon}$
is continuous on $\mathcal{K}_{\epsilon}$ if we consider the reversible
case, i.e., $\boldsymbol{\ell}\equiv\boldsymbol{0}$. 
\end{rem}

For convenience, we formally define $\nabla f_{\epsilon}(\boldsymbol{x})$
as 
\begin{equation}
\nabla f_{\epsilon}(\boldsymbol{x})\,=\,\begin{cases}
\nabla p_{\epsilon}^{\boldsymbol{\sigma}}(\boldsymbol{x}) & \boldsymbol{x}\in\mathcal{B}_{\epsilon}^{\boldsymbol{\sigma}}\text{ for some }\boldsymbol{\sigma}\in\Sigma_{0}\;,\\
0 & \text{otherwise}\;.
\end{cases}\label{e: nabla fe}
\end{equation}
Note that this is not a weak derivative of $f_{\epsilon}$; hence,
elementary theorems such as the divergence theorem cannot be applied
to this gradient. With this formal gradient, we can define $\Phi_{f_{\epsilon}}$
formally as 
\[
\Phi_{f_{\epsilon}}(\boldsymbol{x})\,=\,\nabla f_{\epsilon}(\bm{x})+\frac{1}{\epsilon}\,f_{\epsilon}(\bm{x})\,\boldsymbol{\ell}(\bm{x})\,=\,\begin{cases}
\epsilon^{-1}\,\boldsymbol{\ell}(\boldsymbol{x}) & \boldsymbol{x}\in\mathcal{H}_{0}^{\epsilon}\;,\\
\Phi_{p_{\epsilon}^{\boldsymbol{\sigma}}}(\boldsymbol{x}) & \boldsymbol{x}\in\mathcal{B}_{\epsilon}^{\boldsymbol{\sigma}}\text{ for some }\boldsymbol{\sigma}\in\Sigma_{0}\;,\\
\boldsymbol{0} & \text{otherwise}\;.
\end{cases}
\]
Note that this is a formal definition, and Proposition \ref{p62}
is not applicable to $\Phi_{f_{\epsilon}}$. 

Now, we mollify the function $f_{\epsilon}$ as in \cite{LMS} to
get the genuine test function $g_{\epsilon}$. To this end, consider
a smooth, positive, and symmetric function $\phi:\mathbb{R}^{d}\rightarrow\mathbb{R}$
that is supported on the unit sphere of $\mathbb{R}^{d}$ and satisfies
$\int_{\mathbb{R}^{d}}\phi(\boldsymbol{x})\,d\boldsymbol{x}=1$. Then,
for $r>0$, define $\phi_{r}(\boldsymbol{x})=r^{-d}\phi(r^{-1}\boldsymbol{x})$.
For the function $f:\mathbb{R}^{d}\rightarrow\mathbb{R}$ and vector
field $\boldsymbol{V}:\mathbb{R}^{d}\rightarrow\mathbb{R}^{d}$, we
write 
\[
f^{(r)}\,=\,f*\phi_{r}\;\;\;\text{and\;\;\;\ensuremath{\boldsymbol{V}^{(r)}\,=\,\boldsymbol{V}*\phi_{r}}}\ ,
\]
where $*$ represents the usual convolution. In the remaining subsections,
we prove the following two propositions. Hereafter, we write $\eta=\epsilon^{2}$.
The first one asserts that we can approximate $\Phi_{f_{\epsilon}^{(\eta)}}$
by $\Phi_{f_{\epsilon}}$. 
\begin{prop}
\label{p111}We have 
\[
\epsilon\int_{\mathbb{R}^{d}}\,|\,\Phi_{f_{\epsilon}^{(\eta)}}-\Phi_{f_{\epsilon}}\,|^{2}\,d\mu_{\epsilon}\,=\,o_{\epsilon}(1)\,\alpha_{\epsilon}\;.
\]
\end{prop}

Next, we prove the following estimate. 
\begin{prop}
\label{p112}We have 
\[
\epsilon\int_{\mathbb{R}^{d}}\,[\,\Phi_{f_{\epsilon}}\cdot\nabla h_{\epsilon}\,]\,d\mu_{\epsilon}\,=\,[\,1+o_{\epsilon}(1)\,]\,\alpha_{\epsilon}\,\omega_{0}\;.
\]
\end{prop}

Before proving these propositions, we explain why Theorem \ref{t74}
is a consequence of these propositions. We define the test function
$g_{\epsilon}$ explicitly as 
\begin{equation}
g_{\epsilon}\,=\,f_{\epsilon}^{(\eta)}\;\;\;\text{where }\eta\,=\,\epsilon^{2}\;.\label{egeps}
\end{equation}

\begin{proof}[Proof of Theorem \ref{t74}]
By Proposition \ref{p112}, it suffices to prove that
\[
\epsilon\,\int_{\mathbb{R}^{d}}\,[\,(\Phi_{g_{\epsilon}}-\Phi_{f_{\epsilon}})\cdot\nabla h_{\epsilon}\,]\,d\mu_{\epsilon}\,=\,o_{\epsilon}(1)\,[\,\alpha_{\epsilon}\,\textrm{cap}_{\epsilon}\,]^{1/2}\;.
\]
With the selection \eqref{egeps}, this is immediate from the Cauchy--Schwarz
inequality, Lemma \ref{lem61}, and Proposition \ref{p111}. 
\end{proof}
In Sections \ref{sec102} and \ref{sec103}, we shall prove Propositions
\ref{p111} and \ref{p112}, respectively. We remark that the proof
of Proposition \ref{p111} is nearly model-independent and is similar
to the proof of \cite[Lemma 6.4]{LMS}. Hence, we explain the structure
of the proof and refer to \cite{LMS} for most of the details. Of
course, there are several differences in the proofs, and we present
the full details for such parts. 

\subsection{\label{sec102}Proof of Proposition \ref{p111}}

By the Cauchy-Schwarz inequality, we can write 
\[
\epsilon\,\int_{\mathbb{R}^{d}}\,|\,\Phi_{f_{\epsilon}^{(\eta)}}-\Phi_{f_{\epsilon}}\,|^{2}\,d\mu_{\epsilon}\,\le\,3\,(I_{1}+I_{2}+I_{3})\;,
\]
where 
\begin{align*}
I_{1}\,=\, & \epsilon\,\int_{\mathbb{R}^{d}}\,|\,\nabla(f_{\epsilon}^{(\eta)})-(\nabla f_{\epsilon})^{(\eta)}\,|^{2}\,d\mu_{\epsilon}\;,\;\;\;I_{2}\,=\,\epsilon\int_{\mathbb{R}^{d}}|\,(\nabla f_{\epsilon})^{(\eta)}-\nabla f_{\epsilon}\,|^{2}\,d\mu_{\epsilon}\;\,\;\text{and}\\
I_{3}\,=\, & \frac{1}{\epsilon}\,\int_{\mathbb{R}^{d}}\,(f_{\epsilon}^{(\eta)}-f_{\epsilon})^{2}\,|\boldsymbol{\ell}|^{2}\,d\mu_{\epsilon}\;.
\end{align*}
To conclude the proof of Proposition \ref{p111}, it suffices to prove
that $I_{1},\,I_{2},\,I_{3}=o_{\epsilon}(1)\,\alpha_{\epsilon}$.
The proofs of $I_{1}=o_{\epsilon}(1)\,\alpha_{\epsilon}$ and $I_{2}=o_{\epsilon}(1)\,\alpha_{\epsilon}$
are identical to those of \cite[Lemma 8.5]{LMS} and \cite[Assertions 8.C and 8.D]{LMS},
respectively. The term $I_{3}$ has not been investigated previously.
We present the proof of $I_{3}=o_{\epsilon}(1)\,\alpha_{\epsilon}$.
Note that the functions $f_{\epsilon}^{(\eta)}$ and $f_{\epsilon}$
are supported on $\mathcal{K}$ for sufficiently small $\epsilon>0$,
and since $|\boldsymbol{\ell}|$ is bounded on $\mathcal{K}$, it
suffices to prove the following lemma. 
\begin{lem}
\label{lem113}We have 
\begin{equation}
\frac{1}{\epsilon}\,\int_{\mathbb{R}^{d}}\,(f_{\epsilon}^{(\eta)}-f_{\epsilon})^{2}\,d\mu_{\epsilon}\,=\,o_{\epsilon}(1)\,\alpha_{\epsilon}\;.\label{e112}
\end{equation}
\end{lem}

\begin{proof}
Recall the notation $\mathcal{A}^{[r]}$ from \eqref{eAr} and define
\begin{align*}
\widetilde{\mathcal{B}}_{\epsilon}^{\bm{\sigma}} & \,=\,\mathcal{B}_{\epsilon}^{\bm{\sigma}}\setminus(\partial\mathcal{B}_{\epsilon}^{\bm{\sigma}})^{[\eta]}\;\;\;\text{and\;\;\;}\widetilde{\mathcal{H}}_{i}^{\epsilon}\,=\,\mathcal{H}_{i}^{\epsilon}\setminus\,\Big[\,(\partial\mathcal{K}_{\epsilon})^{[\eta]}\,\cup\,\Big(\,{}_{\bm{\sigma}\in\Sigma_{0}}(\partial\mathcal{B}_{\epsilon}^{\bm{\sigma}})^{[\eta]}\,\Big)\,\Big]\;\;\;;\;i=1,\,2\;.
\end{align*}
By the Cauchy--Schwarz inequality, we have 
\[
\begin{aligned}[][\,(f_{\epsilon}^{(\eta)}-f_{\epsilon})(\bm{x})\,]^{2} & \,=\,\Big(\,\int_{\mathbb{R}^{d}}\,(\,f_{\epsilon}(\bm{x})-f_{\epsilon}(\bm{x}-\bm{y})\,)\,\phi_{\eta}(\bm{y})\,d\bm{y}\,\Big)^{2}\\
 & \,\leq\,\int_{\mathbb{R}^{d}}\,(\,f_{\epsilon}(\bm{x})-f_{\epsilon}(\bm{x}-\bm{y})\,)^{2}\,\phi_{\eta}(\bm{y})\,d\bm{y}\;.
\end{aligned}
\]
Since 
\begin{equation}
f_{\epsilon}(\bm{x})\,=\,f_{\epsilon}(\bm{x}-\bm{y})\;\;\text{\;\;if\;}\boldsymbol{x}\notin\mathcal{K}_{\epsilon}^{[\eta]}\text{ and }|\boldsymbol{y}|\le\eta\;,\label{e113}
\end{equation}
the left-hand side of \eqref{e112} is bounded from above by 
\[
\int_{\mathcal{K}_{\epsilon}^{[\eta]}}\,\int_{\mathbb{R}^{d}}\,\frac{1}{\epsilon}\,|f_{\epsilon}(\bm{x})-f_{\epsilon}(\bm{x}-\bm{y})|^{2}\,\phi_{\eta}(\bm{y})\,d\bm{y}\,\mu_{\epsilon}(d\bm{x})\;.
\]
Now, we divide the integral $\int_{\mathcal{K}_{\epsilon}^{[\eta]}}$
in the previous case into 
\begin{equation}
\int_{\widetilde{\mathcal{H}}_{0}^{\epsilon}}\,+\,\int_{\widetilde{\mathcal{H}}_{1}^{\epsilon}}\,+\,\int_{(\partial\mathcal{K}_{\epsilon})^{[\eta]}}\,+\,\sum_{\boldsymbol{\sigma}\in\Sigma_{0}}\int_{\widetilde{\mathcal{B}}_{\epsilon}^{\bm{\sigma}}}\,+\,\sum_{\boldsymbol{\sigma}\in\Sigma_{0}}\int_{(\partial\mathcal{B}_{\epsilon}^{\bm{\sigma}})^{[\eta]}\setminus(\partial\mathcal{K}_{\epsilon})^{[\eta]}}\;\label{eintdec}
\end{equation}
and consider the five integrals separately. 

The first two integrals are $0$ for the same reason with regard to
\eqref{e113}. Now, we consider the third one. Since $|\,f_{\epsilon}(\bm{x})-f_{\epsilon}(\bm{x}-\bm{y})\,|\leq1$
for all $\bm{x},\,\bm{y}\in\mathbb{R}^{d}$, the integral is bounded
from above by 
\begin{equation}
\int_{(\partial\mathcal{K}_{\epsilon})^{[\eta]}}\,\int_{\mathbb{R}^{d}}\frac{1}{\epsilon}\,\phi_{\eta}(\bm{y})\,d\bm{y}\,\mu_{\epsilon}(d\bm{x})\,=\,\frac{1}{\epsilon}\,\mu_{\epsilon}(\,(\partial\mathcal{K}_{\epsilon})^{[\eta]}\,)\;.\label{e115}
\end{equation}
Since $U(\bm{y})=H+J^{2}\,\delta^{2}$ for $\boldsymbol{y}\in\partial\mathcal{K}_{\epsilon}$,
there exists $C>0$ such that 
\[
U(\bm{x})\,\geq\,H+J^{2}\,\delta^{2}-C\,\eta\;\;\;\;\text{ for all }\bm{x}\in(\partial\mathcal{K}_{\epsilon})^{[\eta]}\;.
\]
Hence, the right-hand side of \eqref{e115} is bounded by 
\[
\frac{C}{\epsilon\,Z_{\epsilon}}\,e^{-H/\epsilon}\,\int_{(\partial\mathcal{K}_{\epsilon})^{[\eta]}}\epsilon^{J^{2}}\,e^{C\eta/\epsilon}\,d\bm{x}\,\le\,C\,\epsilon^{J^{2}-d/2-1}\,\alpha_{\epsilon}\,\text{vol}((\partial\mathcal{K}_{\epsilon})^{[\eta]})\,=\,o_{\epsilon}(1)\,\alpha_{\epsilon}
\]
for sufficiently large $J$, since $\text{vol}\,(\,(\partial\mathcal{K}_{\epsilon})^{[\eta]}\,)=O(1)$. 

Next, we consider the fourth term in \eqref{eintdec}. Fix $\boldsymbol{\sigma}\in\Sigma_{0}$
and assume, for simplicity of notation, that $\boldsymbol{\sigma}=\boldsymbol{0}$.
By the mean value theorem, for $\boldsymbol{x}\in\mathcal{\widetilde{B}}_{\epsilon}^{\bm{\sigma}}$
and $\bm{y}\in\mathcal{D}_{\eta}(\boldsymbol{0})$, 
\begin{equation}
|\,f_{\epsilon}(\bm{x})-f_{\epsilon}(\bm{x}-\bm{y})\,|\,\leq\,|\bm{y}|\,\sum_{k=1}^{d}\,\sup_{\bm{z}\in\mathcal{D}_{\eta}(\bm{x})}\,|\,\nabla_{k}f_{\epsilon}(\bm{z})\,|\;.\label{e116}
\end{equation}
First, we remark from the expression \eqref{e: nabla fe} that, for
$\boldsymbol{u}\in\mathcal{B}_{\epsilon}^{\bm{\sigma}}$,
\begin{equation}
\nabla_{k}f_{\epsilon}(\bm{u})\,=\,\frac{1}{c_{\epsilon}}\exp\,\Big\{\,-\frac{\mu}{2\epsilon}(\boldsymbol{u}\cdot\bm{v}^{\boldsymbol{\sigma}})^{2}\,\Big\}\,v_{k}\;.\label{e117}
\end{equation}
Since $\eta\ll\delta$ and $|\bm{x}|=O(\delta)$, we have
\begin{equation}
(\boldsymbol{z}\cdot\bm{v}^{\boldsymbol{\sigma}})^{2}\,\ge\,(\boldsymbol{x}\cdot\bm{v}^{\boldsymbol{\sigma}})^{2}-C\eta\delta\;\;\;\text{for }\boldsymbol{x}\in\mathcal{\widetilde{B}}_{\epsilon}^{\bm{\sigma}}\;\;\text{and \;}\bm{z}\in\mathcal{D}_{\eta}(\bm{x})\;.\label{e118}
\end{equation}
By combining \eqref{e117} and \eqref{e118}, we get
\[
|\nabla_{k}f_{\epsilon}(\bm{z})|^{2}\,\leq\,\frac{C}{\epsilon}\exp\,\Big\{\,-\frac{\mu}{\epsilon}(\boldsymbol{x}\cdot\bm{v}^{\boldsymbol{\sigma}})^{2}\,\Big\}\,\;.
\]
Inserting this into \eqref{e116}, we obtain, for $\boldsymbol{x}\in\mathcal{\widetilde{B}}_{\epsilon}^{\bm{\sigma}}$,
\[
\int_{\mathbb{R}^{d}}\,|\,f_{\epsilon}(\bm{x})-f_{\epsilon}(\bm{x}-\bm{y})\,|^{2}\,\phi_{\eta}(\bm{y})d\bm{y}\,\leq\,\frac{C\eta^{2}}{\epsilon}\exp\,\Big\{\,-\frac{\mu}{\epsilon}(\boldsymbol{x}\cdot\bm{v}^{\boldsymbol{\sigma}})^{2}\,\Big\}\;.
\]
Therefore, the integral in the fourth term of \eqref{eintdec} is
bounded by 
\[
\frac{1}{\epsilon\,Z_{\epsilon}}\,\frac{C\,\eta^{2}}{\epsilon}\,e^{-H/\epsilon}\int_{\mathcal{\widetilde{B}}_{\epsilon}^{\bm{\sigma}}}\exp\,\Big\{\,-\frac{1}{2\epsilon}\,\boldsymbol{x}\cdot(\mathbb{H}^{\boldsymbol{\sigma}}+2\mu\bm{v}^{\boldsymbol{\sigma}}\otimes\bm{v}^{\boldsymbol{\sigma}})\boldsymbol{x}\,\Big\}\,d\bm{x}
\]
by the Taylor expansion of $U$ around $\boldsymbol{\sigma}$. By
Lemma \ref{lem83}, the last integral is $O(\epsilon^{d/2})$; hence,
the whole expression is $o_{\epsilon}(1)\,\alpha_{\epsilon}$. 

Now, we consider the last integral of \eqref{eintdec}. We also fix
$\boldsymbol{\sigma}$ and assume that $\boldsymbol{\sigma}=\boldsymbol{0}$.
Since
\[
(\partial\mathcal{B}_{\epsilon}^{\bm{\sigma}})^{[\eta]}\setminus(\partial\mathcal{K}_{\epsilon})^{[\eta]}\,\subset\,(\partial_{+}\mathcal{B}_{\epsilon}^{\boldsymbol{\sigma}})^{[\eta]}\,\cup\,(\partial_{-}\mathcal{B}_{\epsilon}^{\boldsymbol{\sigma}})^{[\eta]}\;,
\]
it suffices to prove that the integral over $(\partial_{+}\mathcal{B}_{\epsilon}^{\boldsymbol{\sigma}})^{[\eta]}$
is small, as the argument for $(\partial_{-}\mathcal{B}_{\epsilon}^{\boldsymbol{\sigma}})^{[\eta]}$
is identical. Since $\eta\ll\delta$, by Lemma \ref{lem811}, there
exists a constant $a>0$ such that 
\begin{equation}
U(\bm{x})\,\geq\,aJ^{2}\delta^{2}\;\;\;\;\text{or\;\;\;\;}\boldsymbol{x}\cdot\bm{v}^{\boldsymbol{\sigma}}\,\geq\,aJ\delta\label{edecet}
\end{equation}
holds for all $\bm{x}\in(\partial_{+}\mathcal{B}_{\epsilon}^{\boldsymbol{\sigma}})^{[\eta]}$.
Let us first assume that the former holds. Then, since $|f_{\epsilon}|\leq1$
and $\textup{vol}\,(\,(\partial_{+}\mathcal{B}_{\epsilon}^{\boldsymbol{\sigma}})^{[\eta]}\,)=O(1)$,
by the first condition of \eqref{edecet}, the integral over $\boldsymbol{x}\in(\partial_{+}\mathcal{B}_{\epsilon}^{\boldsymbol{\sigma}})^{[\eta]}$
satisfying the former condition of \eqref{edecet} is bounded from
above by
\begin{equation}
\frac{C}{\epsilon\,Z_{\epsilon}}\,\int_{(\partial_{+}\mathcal{B}_{\epsilon}^{\boldsymbol{\sigma}})^{[\eta]}}\,e^{-U(\bm{x})/\epsilon}\,d\bm{x}\,\le\,\frac{C}{\epsilon\,Z_{\epsilon}}\,e^{-H/\epsilon}\,\epsilon^{aJ^{2}}\,\textup{vol}(\,(\partial_{+}\mathcal{B}_{\epsilon}^{\boldsymbol{\sigma}})^{[\eta]}\,)\,=\,o_{\epsilon}(1)\,\alpha_{\epsilon}\label{e10.10}
\end{equation}
for sufficiently large $J$. 

Now, assume that the second condition of \eqref{edecet} holds for
$\boldsymbol{x}\in(\partial_{+}\mathcal{B}_{\epsilon}^{\boldsymbol{\sigma}})^{[\eta]}$.
As in the proof of Lemma \ref{p87}, we can rewrite $1-f_{\epsilon}(\bm{x})$
as 
\[
[\,1+o_{\epsilon}(1)\,]\,\frac{\epsilon^{1/2}}{(2\pi\mu)^{1/2}\,(\bm{x}\cdot\bm{v}^{\boldsymbol{\sigma}})}\,\exp\,\Big\{\,-\frac{\mu}{2\epsilon}(\boldsymbol{x}\cdot\bm{v}^{\boldsymbol{\sigma}})^{2}\,\Big\}\,\leq\,\frac{C\,\epsilon^{1/2}}{\delta}\,\exp\,\Big\{\,-\frac{\mu}{2\epsilon}(\boldsymbol{x}\cdot\bm{v}^{\boldsymbol{\sigma}})^{2}\,\Big\}\;.
\]
Similarly, we can check that, for $\boldsymbol{y}\in\mathcal{D}_{\eta}(\boldsymbol{0})$,
\[
|\,1-f_{\epsilon}(\bm{x}-\bm{y})\,|\leq\frac{C\,\epsilon^{1/2}}{\delta}\,\exp\,\Big\{\,-\frac{\mu}{2\epsilon}(\boldsymbol{x}\cdot\bm{v}^{\boldsymbol{\sigma}})^{2}\,\Big\}\;.
\]
By the two bounds above, we can bound $|f_{\epsilon}(\bm{x})-f_{\epsilon}(\bm{x}-\bm{y})|^{2}$
from above by
\[
2\,[\,|1-f_{\epsilon}(\bm{x})|^{2}+|1-f_{\epsilon}(\bm{x}-\bm{y})|^{2}\,]\,\le\,\frac{C\,\epsilon}{\delta^{2}}\exp\,\Big\{\,-\frac{\mu}{\epsilon}(\boldsymbol{x}\cdot\bm{v}^{\boldsymbol{\sigma}})^{2}\,\Big\}\;.
\]
Hence, we can bound the last integral of \eqref{eintdec} and restrict
it to $\boldsymbol{x}\in(\partial_{+}\mathcal{B}_{\epsilon}^{\boldsymbol{\sigma}})^{[\eta]}$,
satisfying the second condition of \eqref{edecet}, from above by
\[
\frac{C}{\delta^{2}}\,\int_{(\partial_{+}\mathcal{B}_{\epsilon}^{\boldsymbol{\sigma}})^{[\eta]}}\,\exp\,\Big\{\,-\frac{\mu}{\epsilon}(\boldsymbol{x}\cdot\bm{v}^{\boldsymbol{\sigma}})^{2}\,\Big\}\,\mu_{\epsilon}(d\bm{x})\;.
\]
By applying the Taylor expansion of $U$ around $\boldsymbol{\sigma}$,
this is bounded by 
\[
\frac{1}{\delta^{2}\,Z_{\epsilon}}\,e^{-H/\epsilon}\,\int_{(\partial_{+}\mathcal{B}_{\epsilon}^{\boldsymbol{\sigma}})^{[\eta]}}\,\exp\,\Big\{\,-\frac{\mu}{2\epsilon}\boldsymbol{x}\cdot[\,\mathbb{H}^{\boldsymbol{\sigma}}+2\mu\bm{v}^{\boldsymbol{\sigma}}\otimes\bm{v}^{\boldsymbol{\sigma}}\,]\,\boldsymbol{x}\,\Big\}\,d\bm{x}\;.
\]
By Lemma \ref{lem83}, there exists $c>0$ such that $\boldsymbol{x}\cdot[\,\mathbb{H}^{\boldsymbol{\sigma}}+2\mu\,\bm{v}^{\boldsymbol{\sigma}}\otimes\bm{v}^{\boldsymbol{\sigma}}\,]\boldsymbol{x}\ge c\,|\boldsymbol{x}|^{2}$.
Furthermore, there exists $C>0$ such that $|\boldsymbol{x}|\ge C\delta$
for all $\boldsymbol{x}\in(\partial_{+}\mathcal{B}_{\epsilon}^{\boldsymbol{\sigma}})^{[\eta]}$.
Therefore, we can bound the last centered display from above by 
\begin{equation}
\frac{1}{Z_{\epsilon}}\,e^{-H/\epsilon}\text{\,}\epsilon^{cJ^{2}}\,\text{vol}((\partial_{+}\mathcal{B}_{\epsilon}^{\boldsymbol{\sigma}})^{[\eta]})\,=\,o_{\epsilon}(1)\,\alpha_{\epsilon}\label{e552}
\end{equation}
for sufficiently large $J$ since $\text{vol}((\partial_{+}\mathcal{B}_{\epsilon}^{\boldsymbol{\sigma}})^{[\eta]})=O(1)$.
By \eqref{e10.10} and \eqref{e552}, we can verify that the last
integral of \eqref{eintdec} is $o_{\epsilon}(1)\,\alpha_{\epsilon}$,
and this completes the proof. 
\end{proof}

\subsection{\label{sec103}Proof of Proposition \ref{p112}}

First, note that we can write
\begin{equation}
\epsilon\,\int_{\mathbb{R}^{d}}\,[\,\Phi_{f_{\epsilon}}\cdot\nabla h_{\epsilon}\,]\,d\mu_{\epsilon}\,=\,A_{1}+\sum_{\boldsymbol{\sigma}\in\Sigma_{0}}A_{2}(\boldsymbol{\sigma})\ ,\label{edelev1}
\end{equation}
where 
\begin{align*}
A_{1} & \,=\,\int_{\mathcal{H}_{0}^{\epsilon}}[\,\boldsymbol{\ell}\cdot\nabla h_{\epsilon}\,]\,d\mu_{\epsilon}\;\;\;\;\text{and\;\;\;\;}A_{2}(\boldsymbol{\sigma})\,=\,\epsilon\int_{\mathcal{B}_{\epsilon}^{\bm{\sigma}}}[\,\Phi_{p_{\epsilon}^{\boldsymbol{\sigma}}}\cdot\nabla h_{\epsilon}\,]\,d\mu_{\epsilon}\;.
\end{align*}
To estimate these integrals, we first mention a technical result. 
\begin{lem}
\label{lem14}There exists $C>0$ such that 
\[
\int_{\mathcal{\partial K}^{\epsilon}}\,\sigma(d\mu_{\epsilon})\,\le\,C\,\epsilon^{J^{2}-d/2}\,\alpha_{\epsilon}\;.
\]
\end{lem}

\begin{proof}
Since $U(\boldsymbol{x})=H+J^{2}\delta^{2}$\textbf{ }on $\mathcal{\partial K}^{\epsilon}$,
we have 
\[
\int_{\mathcal{\partial K}^{\epsilon}}\,\sigma(d\mu_{\epsilon})\,=\,\int_{\mathcal{\partial K}^{\epsilon}}\,\mu_{\epsilon}(\boldsymbol{x})\,\sigma(d\boldsymbol{x})\,=\,Z_{\epsilon}^{-1}\,e^{-H/\epsilon}\,\epsilon^{J^{2}}\,\sigma(\mathcal{\partial K}^{\epsilon})\;.
\]
Since $\sigma(\mathcal{\partial K}^{\epsilon})=O(1)$, the proof is
completed by the definition \eqref{ealphae} of $\alpha_{\epsilon}$. 
\end{proof}
We now consider $A_{1}$. 
\begin{lem}
\label{lem115}We can write
\[
A_{1}\,=\,o_{\epsilon}(1)\,\alpha_{\epsilon}+\sum_{\boldsymbol{\sigma}\in\Sigma_{0}}A_{1,\,1}(\boldsymbol{\sigma})\;,
\]
where
\begin{equation}
A_{1,\,1}(\boldsymbol{\sigma})\,=\,\int_{\partial_{+}\mathcal{B}_{\epsilon}^{\boldsymbol{\sigma}}}\,[\,\boldsymbol{\ell}\cdot\boldsymbol{n}_{\mathcal{H}_{0}^{\epsilon}}\,]\,h_{\epsilon}\,\sigma(d\mu_{\epsilon})\;.\label{eA11}
\end{equation}
\end{lem}

\begin{proof}
By the divergence theorem, we have 
\[
\int_{\mathcal{H}_{0}^{\epsilon}}\,[\,\boldsymbol{\ell}\cdot\nabla h_{\epsilon}\,]\,d\mu_{\epsilon}\,=\,\int_{\partial\mathcal{H}_{0}^{\epsilon}}\,[\,\boldsymbol{\ell}\cdot\boldsymbol{n}_{\mathcal{H}_{0}^{\epsilon}}\,]\,h_{\epsilon}\,\sigma(d\,\mu_{\epsilon})\ .
\]
Write 
\[
\partial\widehat{\mathcal{H}}_{0}^{\epsilon}\,=\,\partial\mathcal{H}_{0}^{\epsilon}\setminus\Big[\,\bigcup_{\boldsymbol{\sigma}\in\Sigma_{0}}\partial_{+}\mathcal{B}_{\epsilon}^{\boldsymbol{\sigma}}\,\Big]\,\subset\,\partial\mathcal{K}_{\epsilon}\;.
\]
Then, it suffices to prove that 
\[
\int_{\partial\widehat{\mathcal{H}}_{0}^{\epsilon}}\,[\,\boldsymbol{\ell}\cdot\boldsymbol{n}_{\mathcal{H}_{0}^{\epsilon}}\,]\,h_{\epsilon}\,\sigma(d\mu_{\epsilon})\,=\,o_{\epsilon}(1)\,\alpha_{\epsilon}\;.
\]
Since $|h_{\epsilon}|$ and $|\boldsymbol{\ell}|$ are bounded on
$\partial\widehat{\mathcal{H}}_{0}^{\epsilon}\subset\mathcal{K}$,
and since $\partial\widehat{\mathcal{H}}_{0}^{\epsilon}\subset\partial\mathcal{K}_{\epsilon}$,
the absolute value of the left-hand side of the previous case is bounded
by $\int_{\partial\mathcal{K}_{\epsilon}}\sigma(d\mu_{\epsilon})$,
which is $o_{\epsilon}(1)\,\alpha_{\epsilon}$ for sufficiently large
$J$ by Lemma \ref{lem14}. This completes the proof. 
\end{proof}
Now, we focus on $A_{2}(\boldsymbol{\sigma})$. 
\begin{lem}
\label{lem116}For $\boldsymbol{\sigma}\in\Sigma_{0}$, we can write
\[
A_{2}(\boldsymbol{\sigma})\,=\,o_{\epsilon}(1)\,\alpha_{\epsilon}+A_{2,\,1}(\boldsymbol{\sigma})\;,
\]
where
\begin{equation}
A_{2,\,1}(\boldsymbol{\sigma})\,=\,\epsilon\,\int_{\partial_{+}\mathcal{B}_{\epsilon}^{\bm{\sigma}}\cup\partial_{-}\mathcal{B}_{\epsilon}^{\bm{\sigma}}}\,[\,\Phi_{p_{\epsilon}^{\boldsymbol{\sigma}}}\cdot\bm{n}_{\mathcal{B}_{\epsilon}^{\bm{\sigma}}}\,]\,h_{\epsilon}\,\sigma(d\mu_{\epsilon})\;.\label{eA21}
\end{equation}
\end{lem}

\begin{proof}
By the divergence theorem, we can write 
\begin{align*}
A_{2}(\boldsymbol{\sigma})\,=\, & -\int_{\mathcal{B}_{\epsilon}^{\bm{\sigma}}}(\,\mathscr{L}_{\epsilon}^{*}\,p_{\epsilon}^{\bm{\sigma}}\,)\,h_{\epsilon}\,d\mu_{\epsilon}+\epsilon\int_{\partial\mathcal{B}_{\epsilon}^{\bm{\sigma}}}[\,\Phi_{p_{\epsilon}^{\boldsymbol{\sigma}}}\cdot\bm{n}_{\mathcal{B}_{\epsilon}^{\bm{\sigma}}}\,]\,h_{\epsilon}\,\sigma(d\mu_{\epsilon})\;.
\end{align*}
By Proposition \ref{p86}, the first integral on the right-hand side
is $o_{\epsilon}(1)\,\alpha_{\epsilon}$. Hence, it suffices to prove
that 
\begin{equation}
\epsilon\,\int_{\partial_{0}\mathcal{B}_{\epsilon}^{\bm{\sigma}}}\,[\,\Phi_{p_{\epsilon}^{\boldsymbol{\sigma}}}\cdot\bm{n}_{\mathcal{B}_{\epsilon}^{\bm{\sigma}}}\,]\,h_{\epsilon}\,\sigma(d\mu_{\epsilon})\,=\,o_{\epsilon}(1)\,\alpha_{\epsilon}\;.\label{e1117}
\end{equation}
By the explicit formula for $p_{\epsilon}^{\bm{\sigma}}$ and by the
boundedness of $\boldsymbol{\ell}$ on $\mathcal{K}$, we can check
that there exists $C>0$ such that $|\,\Phi_{p_{\epsilon}^{\boldsymbol{\sigma}}}\,|\le C\epsilon^{-1}$
on $\partial_{0}\mathcal{B}_{\epsilon}^{\bm{\sigma}}$. Therefore,
the absolute value of the left-hand side of \eqref{e1117} is bounded
from above by $C\int_{\partial_{0}\mathcal{B}_{\epsilon}^{\bm{\sigma}}}\,\sigma(d\mu_{\epsilon})$.
Since $\partial_{0}\mathcal{B}_{\epsilon}^{\bm{\sigma}}\subset\partial\mathcal{K}_{\epsilon}$,
the proof is completed by Lemma \ref{lem14}, provided that we take
$J$ to be sufficiently large.
\end{proof}
By \eqref{edelev1} and Lemmas \ref{lem115} and \ref{lem116}, it
suffices to check the following Lemma to complete the proof of Proposition
\ref{p112}. 
\begin{lem}
\label{lem117}For $\boldsymbol{\sigma}\in\Sigma_{0}$, we have 
\[
A_{1,\,1}(\boldsymbol{\sigma})+A_{2,\,1}(\boldsymbol{\sigma})\,=\,[\,1+o_{\epsilon}(1)\,]\,\alpha_{\epsilon}\,\omega^{\boldsymbol{\sigma}}\;.
\]
\end{lem}

We defer the proof of Lemma \ref{lem117} to the next subsection and
conclude the proof of Proposition \ref{p112} first. 
\begin{proof}[Proof of Proposition \ref{p112}]
The proof is completed by combining \ref{edelev1} and Lemmas \ref{lem115},
\ref{lem116}, and \ref{lem117}. 
\end{proof}

\subsection{\label{sec104}Proof of Lemma \ref{lem117}}

As a consequence of Proposition \ref{p101}, we can get the following
estimate of the equilibrium potential at the boundaries $\mathcal{\partial}_{+}\mathcal{B}_{\epsilon}^{\boldsymbol{\sigma}}$
and $\mathcal{\partial}_{-}\mathcal{B}_{\epsilon}^{\boldsymbol{\sigma}}$
for $\boldsymbol{\sigma}\in\Sigma_{0}$. 
\begin{lem}
\label{lem118}There exists a constant $C>0$ such that, for all $\boldsymbol{\sigma}\in\Sigma_{0}$,
\begin{align*}
 & h_{\epsilon}(\bm{x})\,\ge\,1-C\,\epsilon^{-d}\,\exp\frac{U(\bm{x})-H}{2\epsilon}\;\;\;\;\forall\bm{x}\in\partial_{+}\mathcal{B}_{\epsilon}^{\boldsymbol{\sigma}}\;\;\;\text{and}\\
 & h_{\epsilon}(\bm{x})\,\leq\,C\,\epsilon^{-d}\,\exp\frac{U(\bm{x})-H}{2\epsilon}\;\;\;\;\forall\bm{x}\in\partial_{-}\mathcal{B}_{\epsilon}^{\boldsymbol{\sigma}}\;.
\end{align*}
\end{lem}

\begin{proof}
Let us consider the first inequality. If $\bm{x}\in\partial_{+}\mathcal{B}_{\epsilon}$
satisfies $U(\bm{x})\geq H$, then the inequality is obvious for all
sufficiently small $\epsilon$. Otherwise, $\bm{x}\in\mathcal{H}_{0}$;
hence, the bound follows from part (1) of Proposition \ref{p101}
since we have $\mathfrak{H}_{\{\bm{x}\},\,\mathcal{D}_{\epsilon}(\bm{m}_{0})}=U(\bm{x})$
for all sufficiently small $\epsilon$. The proof of the second one
is similar and left to the reader. 
\end{proof}
In the next lemma, we provide a consequence of the previous lemma. 
\begin{lem}
\label{lem119}For $\boldsymbol{\sigma}\in\Sigma_{0}$, we have 
\begin{align}
\epsilon\,\int_{\partial_{+}\mathcal{B}_{\epsilon}^{\boldsymbol{\sigma}}}\,|\,\nabla p_{\epsilon}^{\bm{\sigma}}\,|\,(\,1-h_{\epsilon}\,)\,\sigma(d\mu_{\epsilon}) & \,=\,o_{\epsilon}(1)\,\alpha_{\epsilon}\;,\label{e1118}\\
\int_{\partial_{+}\mathcal{B}_{\epsilon}^{\boldsymbol{\sigma}}}\,(\,1-p_{\epsilon}^{\bm{\sigma}}\,)\,(\,1-h_{\epsilon}\,)\,\sigma(d\mu_{\epsilon}) & \,=\,o_{\epsilon}(1)\,\alpha_{\epsilon}\;,\label{e1119}\\
\epsilon\,\int_{\partial_{-}\mathcal{B}_{\epsilon}^{\boldsymbol{\sigma}}}\,|\,\nabla p_{\epsilon}^{\bm{\sigma}}\,|\,h_{\epsilon}\,\sigma(d\mu_{\epsilon}) & \,=\,o_{\epsilon}(1)\,\alpha_{\epsilon}\;,\label{e1120}\\
\int_{\partial_{-}\mathcal{B}_{\epsilon}^{\boldsymbol{\sigma}}}\,p_{\epsilon}^{\bm{\sigma}}\,h_{\epsilon}\,\sigma(d\mu_{\epsilon}) & \,=\,o_{\epsilon}(1)\,\alpha_{\epsilon}\;.\label{e1121}
\end{align}
\end{lem}

\begin{proof}
\textcolor{black}{Since the proofs of \eqref{e1120} and \eqref{e1121}
are identical to those of \eqref{e1118} and \eqref{e1119}, respectively,
we focus only on \eqref{e1118} and \eqref{e1119}. }

Let us first consider \eqref{e1118}. We use the explicit formula
for $p_{\epsilon}^{\bm{\sigma}}$ and Lemma \ref{lem118} to bound
the left-hand side of \eqref{e1118} by
\begin{equation}
C\,\epsilon^{-1/2-3d/2}\,\alpha_{\epsilon}\,\int_{\partial_{+}\mathcal{B}_{\epsilon}^{\boldsymbol{\sigma}}}\,\exp\,\Big\{\,-\frac{U(\bm{x})-H}{2\epsilon}-\frac{\mu}{2\epsilon}(\bm{x}\cdot\bm{v}^{\boldsymbol{\sigma}})^{2}\,\Big\}\,\sigma(d\bm{x})\;.\label{e1122}
\end{equation}
By the Taylor expansion, the last line can be further bounded by 
\begin{align}
 & C\,\epsilon^{-1/2-3d/2}\,\alpha_{\epsilon}\,\int_{\partial_{+}\mathcal{B}_{\epsilon}^{\boldsymbol{\sigma}}}\,\exp\,\Big\{\,-\frac{1}{4\epsilon}\bm{x}\cdot[\mathbb{H}^{\boldsymbol{\sigma}}+2\mu\bm{v}^{\boldsymbol{\sigma}}\otimes\bm{v}^{\boldsymbol{\sigma}}]\bm{x}\,\Big\}\,\sigma(d\bm{x})\nonumber \\
 & \,\le\,C\,\epsilon^{-1/2-3d/2}\,\alpha_{\epsilon}\,\int_{\partial_{+}\mathcal{B}_{\epsilon}^{\boldsymbol{\sigma}}}\,\exp\,\Big\{\,-\frac{\gamma}{4\epsilon}|\bm{x}|^{2}\,\Big\}\,\sigma(d\bm{x})\ ,\label{e1123}
\end{align}
where $\gamma>0$ is the smallest eigenvalue of the positive-definite
matrix $\mathbb{H}+2\mu\bm{v}\otimes\bm{v}$ (cf. Lemma \ref{lem83}).
Since there exists $C>0$ such that $|\bm{x}|\ge CJ\delta$ for all
$\bm{x}\in\partial_{+}\mathcal{B}_{\epsilon}^{\boldsymbol{\sigma}}$,
and since $\sigma(\partial_{+}\mathcal{B}_{\epsilon}^{\boldsymbol{\sigma}})=O(\delta^{d-1})$,
we can bound \eqref{e1123} from above, for some $c,\,C>0$, by 
\[
C\,\epsilon^{-1/2-3d/2}\,\delta^{d-1}\,\epsilon^{cJ^{2}}\,\alpha_{\epsilon}\,=\,C\,\Big(\log\frac{1}{\epsilon}\Big)^{\frac{d-1}{2}}\,\epsilon^{cJ^{2}-d-1}\,=\,o_{\epsilon}(1)\,\alpha_{\epsilon}
\]
for sufficiently large $J$. This completes the proof of \eqref{e1118}. 

For \eqref{e1119}, recall $\partial_{+}^{1,\,a}\mathcal{B}_{\epsilon}^{\boldsymbol{\sigma}}$
and $\partial_{+}^{2,\,a}\mathcal{B}_{\epsilon}^{\boldsymbol{\sigma}}$
from \eqref{ebd_B1} and \eqref{ebd_B2}, respectively. By Lemma \ref{lem811},
it suffices to prove that, for $a\in(0,\,a_{0})$, 
\begin{align}
\int_{\partial_{+}^{k,\,a}\mathcal{B}_{\epsilon}^{\boldsymbol{\sigma}}}\,(\,1-p_{\epsilon}^{\bm{\sigma}}\,)\,(\,1-h_{\epsilon}\,)\,\sigma(d\mu_{\epsilon}) & \,=\,o_{\epsilon}(1)\,\alpha_{\epsilon}\;\;\;\;;\;k=1,\,2\;.\label{eintt}
\end{align}
For $k=1$, by \eqref{e1-p} and Lemma \ref{lem118}, we can bound
the integral from above by 
\[
\frac{C\,e^{-H\,}\epsilon^{1/2}}{Z_{\epsilon}\,\epsilon^{d}\,\delta}\,\int_{\partial_{+}\mathcal{B}_{\epsilon}}\,\exp\,\Big\{\,-\frac{U(\bm{x})-H}{2\epsilon}-\frac{\mu}{2\epsilon}(\bm{x}\cdot\bm{v}^{\boldsymbol{\sigma}})^{2}\,\Big\}\,\sigma(d\bm{x})\;.
\]
Hence, we can proceed as in the computation of \eqref{e1122} to prove
that this is $o_{\epsilon}(1)\,\alpha_{\epsilon}$. 

Now, we finally consider the $k=2$ case of \eqref{eintt}. Since
$U(\boldsymbol{x})\geq H+aJ^{2}\delta^{2}$ for $\boldsymbol{x}\in\partial_{+}^{2,\,a}\mathcal{B}_{\epsilon}^{\boldsymbol{\sigma}}$,
the left-hand side of \eqref{eintt} with $k=2$ is bounded from above
by 
\[
\frac{1}{Z_{\epsilon}}\,e^{-H/\epsilon}\,\epsilon^{aJ^{2}}\,\sigma(\partial_{+}\mathcal{B}_{\epsilon}^{\boldsymbol{\sigma}})\,\le\,\frac{C}{Z_{\epsilon}}\,e^{-H/\epsilon}\,\epsilon^{aJ^{2}}\,\delta^{d-1}\,=\,o_{\epsilon}(1)\,\alpha_{\epsilon}
\]
for sufficiently large $J$. This completes the proof. 
\end{proof}
Now, we are ready to prove Lemma \ref{lem117}.
\begin{proof}[Proof of Lemma \ref{lem117}]
In view of the expressions \eqref{eA11} and \eqref{eA21} for $A_{1,\,1}(\boldsymbol{\sigma})$
and $A_{2,\,1}(\boldsymbol{\sigma})$, respectively, it suffices to
prove the following estimates: 
\begin{align}
\epsilon\,\int_{\partial_{+}\mathcal{B}_{\epsilon}^{\bm{\sigma}}}\,\Big[\,\Big(\Phi-\frac{1}{\epsilon}\bm{\ell}\,\Big)\cdot\bm{n}_{\mathcal{B}_{\epsilon}^{\bm{\sigma}}}\,\Big]\,h_{\epsilon}\,\sigma(d\mu_{\epsilon}) & \,=\,[\,1+o_{\epsilon}(1)\,]\,\alpha_{\epsilon}\,\omega^{\boldsymbol{\sigma}}\;,\label{e1001}\\
\epsilon\,\int_{\partial_{-}\mathcal{B}_{\epsilon}^{\bm{\sigma}}}[\,\Phi_{p_{\epsilon}^{\boldsymbol{\sigma}}}\cdot\bm{n}_{\mathcal{B}_{\epsilon}^{\bm{\sigma}}}\,]\,h_{\epsilon}\,\sigma(d\mu_{\epsilon}) & \,=\,o_{\epsilon}(1)\,\alpha_{\epsilon}\;.\label{e1002}
\end{align}
Let us first consider \eqref{e1001}. By \eqref{e1118} and \eqref{e1119}
of Lemma \ref{lem119}, we can replace the $h_{\epsilon}(\bm{x})$
term with $1$ with an error term of order $o_{\epsilon}(1)\,\alpha_{\epsilon}$.
Then, we can apply Proposition \ref{p87} to prove \eqref{e1001}.
On the other hand, the estimate \eqref{e1002} is a direct consequence
of \eqref{e1120} and \eqref{e1121} of Lemma \ref{lem119}. 
\end{proof}
\begin{acknowledgement*}
IS was supported by the National Research Foundation of Korea (NRF)
grant funded by the Korea government (MSIT) (No. 2016K2A9A2A13003815,
2017R1A5A1015626 and 2018R1C1B6006896). JL was supported by the NRF
grant funded by the Korea government  (No. 2017R1A5A1015626 and 2018R1C1B6006896).
\end{acknowledgement*}

\end{document}